\documentclass[12pt]{article}
\usepackage{a4,amsmath,amssymb,amsthm,cite,graphicx}
\usepackage{verbatim,numprint,siunitx,mathrsfs,esint}
\usepackage{rotating}
\newtheorem{theorem}{Theorem}[section]
\newtheorem{prop}[theorem]{Proposition}

\newtheorem{lemma}[theorem]{Lemma}
\newcommand{\ud}{\textrm{d}} \newcommand{\be}{\begin{equation}}
  \newcommand{\ee}{\end{equation}}

\newcommand{\wt}[1]{\widetilde{#1}}
\newcommand{\mc}[1]{\mathcal{#1}} \newcommand{\tr}[1]{\textrm{#1}}
\newcommand{\bb}{\beta} 
\catcode`@=11 \@addtoreset{equation}{section} \catcode`@=12

\begin{document}
\bibliographystyle{amsplain}
\begin{titlepage}
  \renewcommand{\thefootnote}{\fnsymbol{footnote}}
  \noindent {\tt IITM/PH/TH/2013/7}\hfill
  {\tt arXiv:1311.7227v2} \\[4pt]
  \mbox{}\hfill \hfill{\fbox{\textbf{v2; July 2014 }}}
  % \hfill{\textbf{v2.0}}

  % \vspace{1.0cm}
  \begin{center}
    \large{\sf A superasymptotic formula for the number of
      plane partitions}
  \end{center}
  \bigskip 
 \begin{center}
    {\sf Suresh Govindarajan${}^*$ and  Naveen S. Prabhakar${}^\dagger$} \\[3pt]
    \textit{${}^*$Department of Physics, Indian Institute of Technology Madras,\\ Chennai 600036, India} \\[9pt]
    \textit{${}^\dagger$Department of Physics and Astronomy,
      Stony Brook University\\
      Stony Brook, NY 11794-3800 USA} \\[4pt]
  \end{center}
  \bigskip \bigskip
  \begin{abstract}
    We revisit a formula for the number of plane partitions due to
    Almkvist. Using the circle method, we provide modifications to his
    formula along with estimates of the errors. We show that  the
    improved formula continues to be an asymptotic
    series. Nevertheless, an optimal truncation (i.e.,
    superasymptotic) of the formula provides exact numbers of plane
    partitions for all positive integers $n\lesssim 6400$ and numbers with
    estimated errors for larger values. For instance, the formula correctly
    reproduces 305 of the 316 digits of the numbers of plane
    partitions of $6999$ as predicted by the estimated error. We believe that an hyperasymptotic
    truncation might lead to exact numbers for positive integers up to
    $50000$.
  \end{abstract}
\end{titlepage}

\section{Introduction}

Let $p_d(n)$ denote the number of $d$-dimensional partitions of $n$. In this notation,
$d=1$ corresponds to the usual partitions of $n$ and $d=2$ denotes
plane partitions. The generating function of plane partitions,
$P_2(q)$, has a product representation due to MacMahon\cite{Macmahon1896memoir}.
\begin{equation}
  P_2(q) := 1+ \sum_{n=1}^\infty p_2(n) \ q^n = \frac1{\prod_{m=1}^\infty (1-q^m)^m}\ .
\end{equation}
Unlike the case of usual partitions, this generating function is not
related to anything modular. Nevertheless, one can apply the circle
method that was used successfully by Hardy and Ramanujan for  the partition function\cite{Hardy1918asymptotic} here as well -- this is  easy to see as the dominant
contributions to the product appear at primitive roots of unity. One
expects a formula of the form\cite[Chapters 5,6]{Andrews1998book}
\begin{equation}
  p_2(n) \sim \sum_{k=1}^\infty \sum_{\substack{h=1\\(h,k)=1}}^{k-1} \psi_{h,k}(n) = \sum_{k=1}^\infty \phi_k(n)\ .
\end{equation}
Since the generating function is non-modular, one generally expects $\psi_{h,k}(n)$ to be of
the form\cite{Almkvist1991}
\begin{equation}
  \psi_{h,k}(n) = e^{-2\pi i n h/k} \times  S_{h,k}(D) \times g(d_k(n))\ ,
\end{equation}
where $D:=\ud /\ud n$, $S_{h,k}(z)$ is some function given either by an
integral representation or by its Taylor series and $g(z)$ is some
`known' special function such as the exponential or  the Bessel function. 

For instance,
the Hardy-Ramanujan-Rademacher exact formula for the numbers of partitions of an integer is\cite{Rademacher1937} 
\begin{equation}
p_1(n)\sim  \sum_{k=1}^\infty \sum_{\substack{h=1\\(h,k)=1}}^{k-1}  e^{-2\pi i n h/k} \times \tfrac{2\pi}k  (\tfrac{\pi}{6k})^{3/2}\, e^{\pi i s(h,k)} \times d_k(n)^{-3/2}I_{3/2}\left(d_k(n)\right)\ ,
\end{equation}
where $s(h,k)=\sum_{m=1}^{k-1} ((\tfrac{m}{k})) ((\tfrac{mh}{k}))$ is the Dedekind sum.
We see that $S_{h,k}=(\tfrac{\pi}k)^{5/2}\,6^{-3/2}\,
e^{\pi i s(h,k)} $, $d_k(n)=\tfrac\pi{k}\sqrt{\tfrac23
  \left(n-\tfrac1{24}\right)}$ and $g(x)=x^{-3/2}I_{3/2}(x)$. In
\cite{Almkvist1991}, Almkvist considers a family of generating
functions which are non-modular where he obtains exact formulae
similar to the above one. Here $S_{h,k}(D)$ turns out to be polynomial
or a convergent series. The non-modular situation is illustrated by
the formula for the number of partitions of $n$ with up to $r$ parts,
$p_{1}(n,r)$, 
\begin{equation}
p_1(n,r)\sim  \sum_{k=1}^\infty \sum_{\substack{h=1\\(h,k)=1}}^{k-1}  e^{-2\pi i n h/k}  \times \tfrac{2\pi}k  (\tfrac{\pi}{6k})^{3/2}\,e^{\pi i s(h,k)} \, S_n(e^{-D+(2\pi ih/k)}) \times g(d_k(n))\ ,
 %d_k(n)^{-3/2}I_{3/2}\left(d_k(n)\right)\ ,
\end{equation}
where $g(x)$ and $d_k(n)$  are as defined for unrestricted partitions and $S_n(x)$ is a polynomial of degree $n$ in $x$ which equals $1$ if $r\rightarrow\infty$ keeping $n$ fixed. The main innovation in the above formula is the use of the differential operator $D:=\ud/ \ud n$. This innovation first appeared in \cite{Almkvist1991}.

We will provide a similar formula for plane partitions extending earlier work of Almkvist\cite{alm1,alm2}. We call the special function that appears in this context as the \textit{Almkvist} function, $\mathcal{A}(x|\gamma)$, as it is a solution to the following third-order ordinary differential equation first considered by Almkvist (a prime denotes $\ud/\ud x$ and $\gamma$ is a real parameter)
\begin{equation*}
x\ y'''(x) -(\gamma-3)\ y''(x) -2 y(x) =0 \ . \tag{\ref{eq:Almkvistde}}
\end{equation*}
However, our solution is different from the one introduced by Almkvist\cite{alm1}. The use of the Almkvist function improves the convergence properties of the sum over $(h,k)$. 
 The analog of  the differential operator $S_{h,k}(D)$ was obtained by Almkvist in \cite{alm2} as the infinite sum defined in Eq. \eqref{bmdef}
 \begin{equation}
 e^{\widetilde{V}_{h,k}(D)}=\sum_{m=0}^\infty b_{h,k}^{(m)}\ D^m\ ,  \textrm{ with } b_{h,k}^{(0)}=1\ .
 \end{equation}
By analysing the behaviour of  $b_{h,k}^{(m)}$ at large $m$, we show that the series is an asymptotic one  and thus our formula is \textit{not} an exact one. We use a standard method to handle  the asymptotic series by truncating it at a point, $m=[M^*(n,k)]$, where the error is minimised -- this is called the \textit{superasymptotic truncation}\cite{Berry1990}. Our formula thus contains two sources of error for which we provide estimates -- the first arises  from the minor arcs and the second one arises from the superasymptotic truncation.  

%For positive non-zero integers, $h$ and $k$ with $(h,k)=1$, let $C_{h,k}$ denote the following generalized Dedekind sum 
%\begin{equation}
%  C_{h,k} := \frac{k}2\sum_{j=1}^{k-1} B_2(j/k) \log\big|2\sin(j h
%  \pi/k)\big|\ . 
%  \label{Chkdef}
%\end{equation}
%\begin{conj}\label{conjecture} The following upper bound for $C_{h,k}$ holds for some $0.49< \alpha <3$.
%\begin{equation}
%\boxed{
%C_{h,k} < \frac{k\log k}{12} -\alpha\ \frac{k\log2}{12}}
%\ . \label{bound3}
%\end{equation}
%\end{conj}
%\noindent \textbf{Remark:} We have numerically tested this conjecture for all $k\leq 1000$ as well as the first $500$ primes and random values
%of $k<10000$ and find that $\alpha$ is close to $3$. The value $\alpha=3$ fails for some small values of $k\leq 34$ but holds for all larger integers to
%the extent we have tested.  \\
%
%\clearpage

\noindent The following theorem is our main result. 
\begin{theorem}\label{maintheorem}
Let $f_1(\lambda)=-\lambda^2 + \frac{\lambda^3}3 +\mathcal{O}(\lambda^5)$ be the function defined in Eq.\! \eqref{f1def}. Further,
Let $a=\zeta(3)$, $c_1=(2a)^{1/36}2^{-\alpha/12}\exp(\zeta'(-1))$, $c_2=3\ 2^{-2/3}a^{1/3}$,  $c(\lambda)= \tfrac {4\pi^2e^{ -\frac12 f_1'(\lambda)}}{(2a)^{1/3}}$ and $\alpha = 3$ is the constant appearing in Proposition \ref{conjecture}. Then,
\begin{equation}\label{eq:mainp2n}
p_2(n)\sim   \sum_{k=1}^{[N(n)]}\sum_{\substack{h=1\\(h,k)=1}}^{k-1} \psi_{h,k}(n) + \mathcal{O}(n^{-\kappa})
\ , 
%\mathcal{O}\left( \tfrac{c^N  (N \nu(n))^{1+ \frac N{24}}}{(a/2)^{1/6}\sqrt{6\pi N^3}} \    e^{\frac{c'}{\nu(n)}}\right)\ , \label{term1a}
\end{equation}
where $N(n)=(2.948 n^{1/3}+(2.936\kappa -1.468) \log n +6.388)$ for some $\kappa>0$ and
\begin{multline}\label{psinumber}
\psi_{h,k}(n) =  e^{-(2\pi i n h/k) +k\zeta'(-1)+C_{h,k}}\ \tfrac{1}{k} \left(\tfrac{a}{k}\right)^{\frac{1}{2} + \frac{k}{24}}\  \sum_{m=0}^{[M^*(n,k)]} b_{h,k}^{(m)} (\tfrac{a}{k^{3}})^{\frac m2} \mathcal{A}\left((\tfrac{a}{k^{3}})^{\frac 12}n\,\big|
    \tfrac{-k}{12}-m\right)\\
  + \mathcal{O}\left(\tfrac{k^{1/2}c_1^k (k^2n^{-2/3})^{1+ \frac k{24}}}{\pi^3  (2a)^{-1/6}\sqrt{12  M^*}}\     \exp\big(\tfrac1k(- \tfrac{c(\lambda)^2}{4c_2}-c(\lambda)\ n^{1/3}+c_2\ n^{2/3})\big)\right) \ . 
\end{multline}
%$M^*(n,k)$, the superasymptotic truncation, is  estimated via a saddle point approximation  to be
 with $\lambda=\frac{k^2n^{-2/3}}{24c_2}$ and $k M^*(n,k) = c(\lambda) n^{\frac13}   - \tfrac{(c(\lambda))^2}{4c_2} f_1''(\lambda)$. 
%with $kM^*(n,k)= c(\lambda) n^{1/3} + \tfrac{(c(\lambda))^2}{2c'}+\mathcal{O}(n^{-1/3})\sim 29.47 n^{1/3} + 216.09$.
 % is to be fixed by the condition that $c^N N^{-3/2}(N \nu(n))^{1+ \frac N{24}} e^{\frac{c'}{\nu(n)}}\sim 1$ and 
\end{theorem}
%The error in Eq. \eqref{term1a} can always be made $\mathcal{O}(1)$ or better. However, the error in Eq. \eqref{term1b} turns out to be $O(1)$ only for $n<6400$ and provides estimates of the error for $n>6400$. We illustrate this for two examples, $n=750<6400$ for which our formula is exact and $n=6491$ for which the estimated error is about $7.54$ while the true error is $-2.58$. 

\section*{\begin{center}Notation\end{center}}
\begin{tabular}{p{2.75cm}p{15cm}}
$p_2(n)$ & Number of unrestricted plane partitions of $n$\\
$P_2(q)$ & Generating function of $p_2(n)$\\
$\mathcal{F}_N$ & Farey sequence with highest denominator $N$\\
$\zeta(s, a)$ & The Hurwitz zeta function\\
$\mathcal{A}(x|\gamma)$ & The Almkvist function\\
$\omega_{h,k}$ & Primitive $(h,k)$-th root of unity, $e^{2\pi ih/k}$, with g.c.d.$(h,k) = 1$\\
$\psi_{h,k}(n)$ & Contribution to $p_2(n)$ due to the root of unity $\omega_{h,k}$\\
$\phi_{k}(n)$ &Sum of $\psi_{h,k}(n)$ over $1 \leq h < k$ for all $(h,k)=1$\\
$\wt{\psi}_{h,k}(n)$ & Series asymptotic to $\psi_{h,k}(n)$\\
$\wt{\phi}_{k}(n)$ & Series asymptotic to $\phi_{k}(n)$\\
$M^*(n,k)$ & Superasymptotic truncation point for $\wt{\phi}_{k}(n)$\\
$f_1(\lambda)$, $f_2(\lambda)$ & Two functions determining the saddle-point for $\mc{A}(x|\gamma)$\\
$a$ & The constant $\zeta(3)\approx 1.20206$ \\
 $c_1$ &The constant $(2a)^{1/36}2^{-\alpha/12}\exp(\zeta'(-1))\approx 0.730207$ for $\alpha=3$\\
  $c_2$ &The constant $3\ 2^{-2/3}a^{1/3}\approx 2.00945$\\
    $c(\lambda)$ & $\tfrac {4\pi^2e^{ -\frac12 f_1'(\lambda)}}{(2a)^{1/3}}\approx 29.4696\ e^{ -\frac12 f_1'(\lambda)}$ \\
    $d(\lambda)$& $72 a\lambda\, 4^{-\alpha} \exp\left(24\zeta'(-1)+\frac{1+f_1(\lambda)}{\lambda}\right)\approx 0.02552\  \lambda\exp\left(\frac{1+f_1(\lambda)}{\lambda}\right) $
\end{tabular}

\subsection*{Organization of the paper}

The paper is organized as follows. Following the introductory section 1, in section 2 we use the circle method to obtain a formula for $p_2(n)$ that naturally leads to a function that we call the Almkvist function. We also describe  important properties of the Almkvist function (relegating details of its saddle-point estimate to appendix D) and also compare our results with that of Almkvist. In section 3, we prove that the formula for $p_2(n)$ is asymptotic and provide estimates of the error using bounds on generalized Dedekind sums given in Appendix C.   These error estimates when combined with the superasymptotic truncation leads to Theorem \ref{maintheorem}, our main result. Section 4 is a numerical study of the formula for $p_2(n)$, in particular on the efficacy of the error estimates as well as the asymptotic behavior. Two tables explicitly detail the computations for $p_2(750)$ and $p_2(6491)$.  We conclude in a section 5 with a short discussion on our results. There are several appendices where more details of computations are provided. In appendix A, we show details of the computation of various residues that were used in Section 2. Appendix B provides an estimate of an important remainder term that is used in section 3. Appendix C provides bounds on the generalized Dedekind sums that appear in Section 2. In appendix D, we prove a reciprocity related for a particular generalised Dedekind sum which enables us to prove a bound that was conjectured in an earlier version of the manuscript\cite{v1}. Finally, in appendix E, we derive the saddle-point estimate for the Almkvist function.

\section{Evaluating $p_2(n)$}

Here we essentially follow the circle method as outlined in \cite{Andrews1998book}. We have
\begin{equation} \label{eq:genf} P_2(q) := 1+ \sum_{n=1}^\infty p_2(n)
  \ q^n = \frac1{\prod_{m=1}^\infty (1-q^m)^m}\ .
\end{equation}
We invert this equation to get
\begin{equation} \label{eq:invgenf} p_2(n) = \frac{1}{2\pi
    i}\int_{\mc{C}_\rho} \ud q\ \frac{P_2(q)}{q^{n+1}}\ ,
\end{equation}
where $C_\rho$ is a circle of radius $\rho < 1$ centered at the origin
in the $q$-plane. From \eqref{eq:genf}, we see that $P_2(q)$ has poles
at the rational points $\omega_{h,k} = e^{2\pi ih/k}$ such that $1
\leq h < k$, $(h,k) = 1$ and $\omega_{0,1} = 1$. We also make the
crucial observation that the strength of the pole $\omega_{h,k}$ is
highest for $k = 1$ and decreases as $k$ increases. (The strength is
the same for a given $k$ and different $h$). Hence, the contribution
to $p_2(n)$ will be greatest from $\omega_{0,1}$ and the contribution
is smaller for $\omega_{1,2}$ and so on. Based on this observation, we
divide $\mc{C}_\rho$ into arcs $\gamma_{h,k}$ which hug $\omega_{h,k}$
such that $\bigcup \gamma_{h,k} = \mc{C}_\rho$.

\subsection{The contour $\mc{C}_\rho$}
\noindent In order to go about calculating the contour integral, we
shall compute the contributions due to $\omega_{h,k}$ such that
$(h,k)$ appear in the Farey sequence $\mc{F}_N$, $N \geq 1$. We then
take $N \to \infty$ in the end. The Farey sequence $\mc{F}_N$ is
defined as follows. A Farey sequence $\mc{F}_N$ consists of the
rational numbers $0 \leq h/k < 1$ such that $k \leq N$. Further, any
two adjacent terms in $\mc{F}_N$, $\tfrac ab$, $\tfrac cd$, are such
that $bc - ad = 1$.\footnote{For example, $\mc{F}_4 = \left\{\tfrac
    01, \tfrac 14, \tfrac 13, \tfrac 12, \tfrac 23, \tfrac 34, \tfrac
    11\right\}\ .$}
\noindent We now describe the contour $\gamma_{h,k}$. First, we write
\begin{equation}
  q = \exp\left[-\varrho(N) + 2\pi i\theta\right]\ .
\end{equation}
$\varrho(N)$ is a positive function of $N$ whose exact form we fix
later. Hence, $\mc{C}_\rho$ is a circle of radius $\rho =
e^{-\varrho(N)} < 1$ parametrized by $\theta \in [0,1]$. The arcs
$\gamma_{h,k}$ are thus neighbourhoods of $\theta = h/k$. To make the
presence of $\theta = h/k$ more transparent in $\gamma_{h,k}$, we
define $\zeta = \theta - h/k$. Suppose $h_0/k_0$, $h/k$ and $h_1/k_1$
are successive terms in $\mc{F}_N$. Then, the end points
$(-\zeta'_{h,k}, \zeta''_{h,k})$ of $\gamma_{h,k}$ on $\mc{C}_\rho$ are
defined to be
\begin{align}
  \zeta_{0,1} & = \frac{1}{N+1}\ ,\\
  \zeta'_{h,k} & = \frac{h}{k} - \frac{h_0 + h}{k_0 + k} \quad \tr{for}\quad h > 0\ ,\\
  \zeta''_{h,k} & = \frac{h + h_1}{k + k_1} - \frac{h}{k}\ .
\end{align}
We make a further change of variables to $z = \varrho(N) - 2\pi
i\zeta$. At this stage \eqref{eq:invgenf} becomes
\begin{equation} \label{eq:finalinvg} p_2(n) = \lim_{N \to \infty}
  \sum_{(h,k) \in \mc{F}_N} \int_{z'_{h,k}}^{z''_{h,k}} \frac{\ud
    z}{-2\pi i}\ e^{nz -2\pi in\frac
    hk}\,P_2\left(e^{-z}\omega_{h,k}\right)\ ,
\end{equation}
with $z'_{h,k} = \varrho(N) + 2\pi i\zeta'_{h,k}$ and $z''_{h,k} =
\varrho(N) - 2\pi i\zeta''_{h,k}$. Then, the pole $q = \omega_{h,k}$
corresponds to $z \to 0$ with $\tr{Re}(z) > 0$. We shall refer to this
limit as $z \to 0^{+}$ from now on. In order to obtain the
contribution to $p_2(n)$ due to $\omega_{h,k}$, we need to obtain an
expression for $P_2(q)$ as $z \to 0^+$. We do this next.

\subsection{An expansion for $P_2\left(e^{-z}\omega_{h,k}\right)$ as $z \to 0^+$.}
We have, for $z \to 0^+$, (required for the Mellin-Barnes
representation of $e^{-mnz}$)
\begin{align}
  \log P_2\left(e^{-z}\omega_{h,k}\right) & = -\sum_{n=1}^{\infty}n
  \log
  \left(1 - e^{-nz}\omega_{h,k}^n\right)\ ,\\
  & = \sum_{n=1}^{\infty}\sum_{m=1}^{\infty}\frac{n\, e^{-nmz}\,
    \omega_{h,k}^{mn}}{m}\ ,\\
  & = \sum_{n=1}^\infty\sum_{m=1}^\infty\
  \omega_{h,k}^{mn}\int^{\sigma + i\infty}_{\sigma - i\infty}
  \frac{\ud s}{2\pi i}\ n^{-s+1}\, m^{-s-1}\,z^{-s}\,\Gamma(s)\
  ,\label{eq:P2qmellin}
\end{align}
where, in the last line, we have used the Mellin-Barnes representation
of $e^{-x}$:
\begin{equation}
  e^{-x} = \int^{\sigma +
    i\infty}_{\sigma - i\infty} \frac{\ud s}{2\pi i}\ x^{-s}\,\Gamma(s)\ ,\ \ \tr{for  }\ \tr{Re}(x) > 0,\ \ \sigma > 0\ .
\end{equation}
Here, we have to choose $\sigma$ such that the contour $\tr{Re}(s) =
\sigma$ lies to the right of all the poles of the integrand in
\eqref{eq:P2qmellin}. We shall assume that $\sigma$ has been fixed
(which we shall, in a moment) such that the integral in
\eqref{eq:P2qmellin} is convergent. Then, we can take the summation
over $m$, $n$ across the integral over $s$. Setting $m = \mu k + d$
and $n = \mu'k + d'$ with $\mu, \mu' \in [0, \infty)$ and $d, d' \in
[1, k]$, we get
\begin{align}
  \log P_2\left(e^{-z}\omega_{h,k}\right) & = \int^{\sigma +
    i\infty}_{\sigma - i\infty} \frac{\ud s}{2\pi i}\,(zk^2)^{-s}\,\Gamma(s)\sum_{d,d'=1}^k \sum_{\mu,\mu'=0}^\infty\frac{\omega_{h,k}^{dd'}}{\left(\mu'+\tfrac{d'}{k}\right)^{s-1}\, \left(\mu+ \tfrac{d}{k}\right)^{s+1}}\ ,\\
  & = \int^{\sigma + i\infty}_{\sigma - i\infty} \frac{\ud s}{2\pi
    i}\,(zk^2)^{-s}\,\Gamma(s)\sum_{d,d'=1}^k\,\omega_{h,k}^{dd'}\
  \zeta(s-1,\tfrac{d'}{k})\ \zeta(s+1,\tfrac{d}{k})\
  , \label{eq:P2qhur}
\end{align}
where we have used $\zeta(r, d/k) = \sum_{\mu=0}^\infty\left(\mu +
  \tfrac dk\right)^{-r}$, $\zeta(r, d/k)$ being the Hurwitz
$\zeta$-function.

 Let us look at the pole structure of the integrand in
\eqref{eq:P2qhur}. This will enable us to fix the value of $\sigma$
and also lead the path to the next step of the computation. The pole
structure of the each of the factors in the integrand is as follows:
\begin{enumerate}
\item $\Gamma(s)$ has simple poles at $s = -p$, $p \geq 0$ with residues $(-1)^p / p!$ respectively.
\item $\zeta(s-1, d/k)$ has a simple pole at $s = 2$ with residue 1.
\item $\zeta(s+1, d/k)$ has a simple pole at $s =0$ with residue 1.
\end{enumerate}
Hence, we have a simple pole at $s = 2$, a double pole at $s = 0$ and
simple poles at $s = -1, -2, -3, \dots$. It is then sufficient to fix
$\sigma = 2 + \epsilon$, $\epsilon \gtrsim 0$ to ensure convergence of
the integral in \eqref{eq:P2qhur}.

Next, we move the contour from $\tr{Re}(s) = 2 + \epsilon$
to $\tr{Re}(s) = -1-\epsilon$. This will essentially pick out the
residues at $s = 2, 0, -1$. Then we, get
\begin{equation} \label{eq:P2qres} \boxed{ \log
    P_2\left(e^{-z}\omega_{h,k}\right) = \tr{Res}_{s=2} +
    \tr{Res}_{s=0} + \tr{Res}_{s=-1} + L_{h,k}(z) }\ ,
\end{equation}
with
\begin{equation} \label{eq:Lhkz} L_{h,k}(z) := \int^{-1-\epsilon +
    i\infty}_{-1-\epsilon - i\infty} \frac{\ud s}{2\pi
    i}\,(zk^2)^{-s}\,\Gamma(s)\sum_{d,d'=1}^k\,\omega_{h,k}^{dd'}\
  \zeta(s-1,\tfrac{d'}{k})\ \zeta(s+1,\tfrac{d}{k})\ ,
\end{equation}
The poles $s = -2, -3, \ldots$ are now present in $L_{h,k}(z)$. The
residues in \eqref{eq:P2qres}  and for the poles in $L_{h,k}(z)$ are
computed in  Appendix \ref{residues}. They are as follows:
\begin{align}
  \tr{Res}_{s=2} & = \frac{\zeta(3)}{z^2 k^3}\
  ,\\ %z^{-2}\Gamma(-2)k^{-4}\sum_{d,d'=1}^k \zeta(3, d/k) \omega_{h,k}^{dd'} =
  \tr{Res}_{s=0} & = \frac{k}{12}\log (zk) + k\zeta'(-1) + C_{h,k}\ ,\\
  \tr{Res}_{s=-1} &  := v^{(1)}_{h,k}\,z = \frac{iz
    k^{2}}{6}\sum_{d=1}^{k-1}B_{3}(d/k)\cot(\pi dh/k)\ .
\end{align}
The residue at $s = -p$ for integer $p > 1$ will be useful later on:
\begin{align} \label{eq:Ressp}
  \tr{Res}_{s=-p} & = \frac{(-z)^p k^{1+p}}{p! p
    (p+2)}\bigg[B_{p+2}B_p +
  \frac{p}{(2i)^p}\sum_{d=1}^{k-1}B_{p+2}(d/k)\cot^{(p-1)}(\pi
  dh/k)\bigg]\ ,\\
  & =: v^{(p)}_{h,k}\, z^p\ . \label{eq:vphk}
\end{align}
We then have the following theorem which is originally due to Almkvist\cite{alm2}:
\begin{theorem}[Almkvist\cite{alm2}] \label{th:logP2q} Let $z \to 0^+$ and $k\geq1$, $1\leq h <k$ and $(h,k)=1$. Then we have
  \begin{equation}
    \log P_2\left(e^{-z}\omega_{h,k}\right) = \frac{a}{k^3z^2} + \frac{k}{12} \log (zk) + k\zeta'(-1) + C_{h,k} +V_{h,k}(z)\ ,
  \end{equation}
    where $a=\zeta(3)$ and
    \begin{align}
      V_{h,k}(z) & :=   z\, v^{(1)}_{h,k} + L_{h,k}(z)\ .
    \end{align}
\end{theorem}
It is important to note that $L_{h,k}(z)$ here is defined by the integral \eqref{eq:Lhkz} and not the power series that we will obtain later. The dominant $k=1$ term was originally computed by Wright\cite{Wright1931} and the  $k=2$ term was also obtained by Knessl\cite{Knessl1994}.\\

\subsection{Evaluating $p_2(n)$: The Almkvist
  function.}

\begin{comment}
We pull out $e^{L_{h,k}(z)}$ as a differential operator acting on the integral:
 \begin{multline} \label{eq:p2nwithD}
   p_2(n) = \sum_{(h,k)\in \mathcal{F}_N}e^{-2\pi ni \frac{h}{k} +
     C_{h,k} + L_{h,k}(D)}\int_{z'_{h,k}}^{z''_{h,k}} \frac{dz}{-2\pi i}\ (zk)
 ^{k/12}\, \exp\bigg[\frac{a}{k^3z^2} \\+
     (n + iS_{h,k})z\bigg]\ ,
\end{multline}
with $D := \ud/\ud n$.
\end{comment}
We put in the expression for $\log P_2(q)$ from Theorem
\eqref{th:logP2q} into \eqref{eq:finalinvg} to obtain
\begin{multline}
  p_2(n) = \sum_{(h,k)\in \mathcal{F}_N} e^{-2\pi i n \frac{h}{k}}\int_{z'_{h,k}}^{z''_{h,k}} \frac{\ud z}{-2\pi i}\
  \exp\bigg[\frac{a}{k^3z^2} + \frac{k}{12}\log (zk) + 
    nz + \\+k\zeta'(-1) + C_{h,k}+ V_{h,k}(z) \bigg]\ .
\end{multline}
Let us consider
\begin{align}
  I_{h,k} & := \int_{z'_{h,k}}^{z''_{h,k}} \frac{\ud z}{-2\pi i}\ g_k(z) = \frac{1}{2\pi i}\int_{z''_{h,k}}^{z'_{h,k}} \ud z\ g_k(z)\ ,\\
g_k(z) & := (zk)^{k/12}\,\exp\bigg[k\zeta'(-1) +C_{h,k}+V_{h,k}(z) + \frac{a}{k^3z^2} + nz\bigg]\ .
\end{align}
Recall that $z'_{h,k} = \varrho(N) + 2\pi i\zeta'_{h,k}$ and $z''_{h,k}
= \varrho(N) - 2\pi i\zeta''_{h,k}$. With this, we can write $I_{h,k}$ as 
\begin{equation}
  I_{h,k} = \Bigg(\int_{\mc{C}^{\epsilon}} - \int_{-i\epsilon}^{-2\pi i\zeta''_{h,k}} - \int_{-2\pi i\zeta''_{h,k}}^{z''_{h,k}} - \int_{z'_{h,k}}^{2\pi i\zeta'_{h,k}} - \int_{2\pi i\zeta'_{h,k}}^{+i\epsilon}\Bigg)\frac{\ud z}{2\pi i}\,g_k(z)\ ,
\end{equation}
where $\mc{C}^{\epsilon}$ is the contour in Figure 1. For brevity, we write
\begin{equation}
  I_{h,k} = J_{0} - J_1 - J_2 - J_3 - J_4\ .
\end{equation}
In the next subsection, we show that $J_1$, $J_2$, $J_3$ and $J_4$ are
negligible in the limit $N \to \infty$ for a particular choice of
$\varrho(N)$ and that $J_0$ is non-negligible. Then, $J_0$ will be the
dominant contribution to $I_{h,k}$ as $N \to \infty$.
\begin{figure}[htbp!]
\begin{center}
%\scalebox{0.6}{\input{contour.pdf_t}}
\includegraphics[scale=0.6]{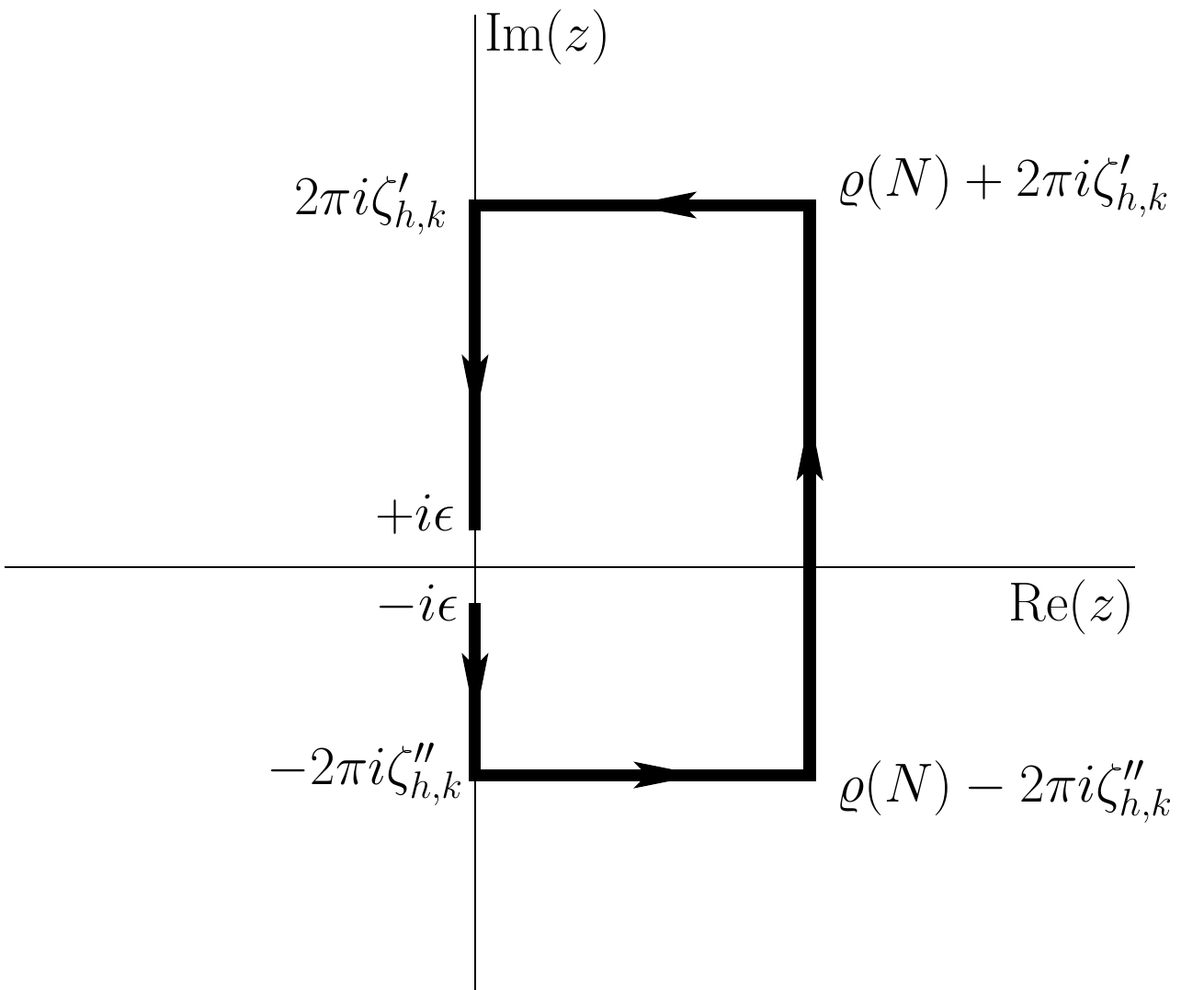}
\caption{Contour for the evaluation of $I_{h,k}$.}
\end{center}
\end{figure}

\subsubsection{Bounds on $J_1$,
  $J_2$, $J_3$, and
  $J_4$}
  
The Farey arcs satisfy the following inequality \cite{Andrews1998book}
\begin{equation}
  \frac{1}{2kN}\  \leq\ \zeta'_{h,k},\,  \zeta''_{h,k}\ \leq\ \frac{1}{kN}\ .\label{eq:zetab}
\end{equation}
In appendix \ref{boundedsums} we derive bounds on the various generalized Dedekind sums that appear in this paper. 
However, there is one bound that remains unproven and 
is stated as  conjecture \ref{conjecture}.
%\begin{conj}\label{conjecture} The following upper bound for $C_{h,k}$ holds for some $0.31< \alpha <3$.
%\begin{equation}
%\boxed{
%C_{h,k} < \frac{k\log k}{12} -\alpha\ \frac{k\log2}{12}}
%\ . \label{bound3}
%\end{equation}
%\end{conj}
%\noindent \textbf{Remark:} We have numerically tested this conjecture for all $k\leq 1000$ as well as the first $500$ primes and random values
%of $k<10000$ and find that $\alpha$ is close to $3$. $\alpha=3$ fails for some small values of $k\leq 34$ but holds for all larger integers to
%the extent we have tested.  
  
We also use the following inequalities on the Dedekind sums that we
derive in Appendix \ref{boundedsums}. Here we use a slightly weaker upper
bound on $C_{h,k}$ than the one in Proposition \ref{conjecture}.
\begin{align} 
C_{h,k}\ \ & <\ \ \frac{k}{12}\log k\ ,\label{eq:chkb}\\
|v^{(1)}_{h,k} z|\ \ & \leq\ \ \frac{2a}{(2\pi)^3}k^3\,|z|\ ,\label{eq:shkb}\\
|L_{h,k}(z)|\ \ & \leq\ \ \tfrac{a_1}{(2\pi)^2}\,  k^3\,|z|^{2}\label{eq:Lhkb}\ .
\end{align}
with $a_1 = \frac{12\zeta(2)\zeta(4)}{(2\pi)^4}$. Let us
look at $J_1$ first. (The bound for $J_4$ is obtained by working with
$z \to -z$.) We set $z = -i u$ since $z$ assumes purely
imaginary values on the contour for $J_1$. Then
\begin{multline} \label{eq:J1} |J_1|\ \leq\ 
  \int_{\epsilon}^{2\pi\zeta'_{h,k}} \frac{\ud
    u}{2\pi}\,(uk)^{k/12}\,\exp\bigg[\frac{-a}{k^3 u^2} +k\zeta'(-1) + C_{h,k} + \\ -i v^{(1)}_{h,k} u + \tr{Re}\left[L_{h,k}(-iu)\right]\bigg]\ .
\end{multline}
On the contour, we have $\epsilon \leq u \leq 2\pi\zeta'_{h,k} \leq 2\pi/kN$. Hence, we have
\begin{align}
  \frac{k^2 N^2}{4\pi^2} \ \ &\leq  \ \ \frac{1}{u^2} \ \ \leq\ \ \frac{1}{\epsilon^2}\ ,\\
  (k^2\epsilon)^{k/12}\ \ &\leq  \ \ (uk^2)^{k/12} \ \ \leq\ \ \left(\frac{2\pi k}{N}\right)^{k/12}\ ,\\
 -\frac{2a k^3\epsilon}{8\pi^3} \ \ &\leq  \ \ -iv^{(1)}_{h,k} u \ \ \leq\ \ \frac{ak^2}{2\pi^2 N}\ ,\\
  \tr{Re}\left[L_{h,k}(-iu)\right]\ \  &\leq\ \   \left|L_{h,k}(-iu)\right|\ \  \leq\ \  a_1\,\frac{k}{N^{2}}\ .
\end{align}
Using the above bounds and $\eqref{eq:chkb}$, we get
\begin{equation}
  \label{eq:J1preprebound}
  J_1 \leq \frac{1}{kN}\,\left(\frac{2\pi k}{N}\right)^{k/12}\exp\left[-\frac{aN^2}{4\pi^2 k} + \frac{ak^2}{2\pi^2N} + a_1\frac{k}{N^{2}}\right]\ .
\end{equation}
The above bound is valid for all finite $k$ and $N$. To obtain the
bound that goes to 0 the slowest as $N \to \infty$, we substitute $k =
k_{\tr{max}} = \lfloor N^{1-\varepsilon}\rfloor, \varepsilon \to 0^+$ in all factors
except $1/kN$ in which we let $k = 1$. The $N^{-\varepsilon}$ is to
ensure that the error goes to 0 as $N \to \infty$. Then we get
\begin{equation}
  \label{eq:J1prebound}
  J_1 \leq N^{-1-\varepsilon}\exp\left[-\frac{aN^{1+\varepsilon}}{4\pi^2} + \frac{aN^{1-\varepsilon}}{2\pi^2} + N^{1-\varepsilon}\log (2\pi)^{1/12}\right]\ .
\end{equation}
Thus, we have, with $b_1 = a/4\pi^2$ and $b_2 = \frac{a}{2\pi^2}  +
\log(2\pi)^{1/12}$,
\begin{equation}
  \label{eq:J1bound}
  \boxed{
  J_1, J_4 = \mc{O}\left(N^{-1-\varepsilon}\, e^{-b_1 N^{1+\varepsilon} + b_2 N^{1-\varepsilon}}\right)}
\end{equation}

\noindent Next, let us look at $J_3$. (The bound for $J_2$ can be obtained
by working with $z \to \overline{z}$.) We set $z = u + 2\pi i\zeta'_{h,k}$. Then, we have
\begin{equation}  |J_3|\ \leq\  e^{2\pi i v^{(1)}_{h,k}\zeta'_{h,k} + n\varrho(N) + C_{h,k}}\int_{0}^{\varrho(N)} \frac{\ud
    u}{2\pi}\,|zk|^{k/12}\,\exp\left[\frac{a}{k^3}\tr{Re}({1}/{z^2}) + \tr{Re}\left[L_{h,k}(z)\right]\right]\ .
\end{equation}
We then use \eqref{eq:chkb} and the following:
\begin{align}
  \tr{Re}(1/z^2) = \frac{u^2 - 4\pi^2\zeta^{'2}_{h,k}}{|z|^4} < 0 \ \ &\Longrightarrow\  \exp\left[\frac{a}{k^3}\tr{Re}(1/z^2)\right]\ \leq\ 1\ ,\\
  \frac{-ak^3}{4\pi^2 N}\ \ \leq\ \   -2\pi & i v^{(1)}_{h,k} \,\zeta'_{h,k}\ \  \leq\ \  \frac{ak^2}{2\pi^2 N}\ ,\\
  \tr{Re}\left[L_{h,k}(z)\right]\ \ \leq\ \ \left|L_{h,k}(z)\right|\ &\  \leq\ \  \frac{a_1}{(2\pi)^2}\,k^3\left[\varrho(N)^2 + \frac{4\pi^2}{k^2N^2}\right],
\end{align}
and
\begin{equation}
  e^{C_{h,k}}|zk|^{k/12}\ \leq\ (k^2)^{k/12}[u^2 + 4\pi^2\zeta^{'2}_{h,k}]^{k/24}\ \leq\ k^{k/12}[\varrho(N)^2k^2 + 4\pi^2 / N^2]^{k/24}\ . 
\end{equation}
Assuming $\varrho(N) \ll \frac{2\pi}{kN}$, we get
\begin{equation} \label{eq:J3} |J_3|\ \leq\ e^{n\varrho(N)}\left(\frac{2\pi k}{N}\right)^{k/12}\frac{\varrho(N)}{2\pi}\,\exp\left[\frac{ak^2}{4\pi^2 N} + a_1\frac{k}{N^{2}}\right]\ .
\end{equation}
Now, if we fix $\varrho(N) = N^{-1-\varepsilon}e^{-b_1N^{1+\varepsilon} + \frac{a}{4\pi^2}N^{1-\varepsilon}}$, we again get
\begin{equation}
  \label{eq:J3bound}
  \boxed{
  J_2, J_3 = \mc{O}\left(N^{-1-\varepsilon}\, e^{-b_1N^{1+\varepsilon} + b_2 N^{1-\varepsilon}}\right)}
\end{equation}
This proves that
\begin{equation} \label{eq:Ihkbound}
I_{h,k} = J_0\ +\ \mc{O}\left(N^{-1-\varepsilon}\, e^{-b_1N^{1+\varepsilon} + b_2 N^{1-\varepsilon}}\right)\ .
\end{equation}

\subsubsection{A closed formula for $p_2(n)$: The Almkvist function.}
We have
\begin{equation} 
J_0 = \int_{\mc{C}^{\epsilon}}\frac{\ud z}{2\pi i}\,\,(zk)^{k/12}\,\exp\left[k\zeta'(-1) +C_{h,k}+ V_{h,k}(z) + \frac{a}{k^3z^2} + n z\right]\ ,
\end{equation}
which we write as
\begin{align}
J_0 &= e^{k\zeta'(-1) +C_{h,k}+V_{h,k}(D)}\int_{\mc{C}^{\epsilon}}\frac{\ud z}{2\pi i}\,\,(zk)^{k/12}\,\exp\left[\frac{a}{k^3z^2} + n z\right]\ ,\label{eq:pulloutVhk}\\
& := k^{-1}\left(\frac{a}{k}\right)^{\frac{1}{2} + \frac{k}{24}}\,e^{k\zeta'(-1) +C_{h,k}+V_{h,k}(D)}\,\mc{A}\left(\left(ak^{-3}\right)^{\frac{1}{2}}n\,\Big| \tfrac{-k}{12}\right)\ .\label{eq:J0}
\end{align}
with $D := \ud/\ud n$. The \textit{Almkvist function}
$\mc{A}(x|\gamma)$ is a solution of the following differential
equation\cite{alm1}
\begin{equation} \label{eq:Almkvistde}
x\ y'''(x) -(\gamma-3)\ y''(x) -2 y(x) =0 \ .
\end{equation}
where $'$ denotes derivative w.r.t $x$ and $\gamma$ is a real
number. $\mc{A}(x|\gamma)$ has the following integral representation:
\begin{equation} \label{eq:almkvistfn}
  \mc{A}(x| \gamma) = \int_{\mc{C}^{\epsilon}}\frac{\ud z}{2\pi i}\ z^{-\gamma}\exp\left[\frac{1}{z^2} + xz\right]\ ,
\end{equation}
where $\mc{C}^{\epsilon}$ is the contour in Figure 1. Substituting the
expression for $J_0$ and the bounds for $J_1$, $J_2$, $J_3$ and $J_4$,
we get
\begin{multline} \label{eq:p2nwithD} p_2(n) =
  \sum_{k=1}^{k_{\tr{max}}}\sum_{\substack{h =
      1\\(h,k)=1}}^{k-1}\,k^{-1}\left(\tfrac{a}{k}\right)^{\frac{1}{2} +
    \frac{k}{24}}e^{-2\pi n i\frac{h}{k} + k\zeta'(-1)+C_{h,k}+
    V_{h,k}(D)}\mc{A}\left((ak^{-3})^{\frac 12}n\,\Big| \tfrac
    {-k}{12}\right) \\ +\
  \sum_{k=1}^{k_{\tr{max}}}\sum_{\substack{ h = 1\\(h,k) =
      1}}^{k-1}\mc{O}\left(N^{-1-\varepsilon}\,e^{-b_1N^{1+\varepsilon} +
      b_2N^{1-\varepsilon}}\right)\ .
\end{multline}
Carrying out the sum in the error term above and letting $N \to
\infty$, we get the sum
\begin{align} \label{eq:p2nwithDfull}
  p_2(n) = \lim_{N\to\infty}\sum_{k=1}^{k_{\tr{max}}}\sum_{\substack{h =1\\(h,k) =
      1}}^{k-1}\,k^{-1}\left(\tfrac{a}{k}\right)^{\frac{1}{2} + \frac{k}{24}}
  & e^{-2\pi ni \frac{h}{k} +k\zeta'(-1)+C_{h,k}+
    V_{h,k}(D)}\mc{A}\left(({a}{k^{-3}})^{\frac 12}n\,\Big|
    \tfrac{-k}{12}\right)\ .
\end{align}
with $k_{\tr{max}} = \lfloor N^{1-\varepsilon}\rfloor \to \infty$ as
$N \to \infty$. We see that the contribution from the additional contours go to zero as
$N\to\infty$.  

 \textbf{Remark:} In writing
\eqref{eq:pulloutVhk}, we have pulled $e^{V_{h,k}(z)}$ out of the
contour integral \textit{formally} without saying anything about the
convergence of the $L_{h,k}(z)$ term inside $V_{h,k}(z)$. Strictly
speaking, one must investigate the domain of convergence of
$e^{L_{h,k}(z)}$ and check that the contour $C^{\epsilon}$ lies inside
it before we can pull $e^{V_{h,k}(z)}$ out of the integral. We shall
address the convergence of $L_{h,k}(z)$ in the next section. At this
point, the above exchange of $e^{V_{h,k}(z)}$ with the contour
integration is best thought of as a formal manipulation. Hence, we
\textbf{cannot} yet conclude that the series for $p_2(n)$ above is
truly convergent. Here, we wish to emphasize the role of the function
$\mc{A}(x|\gamma)$ in the above formula for $p_2(n)$.\\

\subsection{A comparison with Almkvist's formula for $p_2(n)$.}

The function $\mc{A}(x|\gamma)$ has the following properties in addition to the integral representation given in Eq. \eqref{eq:almkvistfn} :
\begin{enumerate}
\item The Frobenius power series solution to the differential equation \eqref{eq:Almkvistde} leads to the following representation
for the Almkvist function.
 \begin{equation}
\label{property1}
\mc{A}(x|\gamma)=\frac12 \sum_{k=0}^\infty \frac{z^{k}}{k!\ \Gamma\left(\frac{3-\gamma+k}2\right)}\ .
\end{equation}
The integral representation in Eq. \eqref{eq:almkvistfn} can be obtained by using Hankel's representation for the inverse of the Gamma function
\begin{equation*}
\frac1{\Gamma(z)} = \int _{-\infty}^{(0+)} \frac{\ud t}{2\pi i}\ t^{-z}\, e^t\,\ ,
\end{equation*}
where the contour is the Hankel contour encircling the cut along negative part of the real axis in the counterclockwise sense (see Fig. \ref{hankelcontour}).
\item \begin{equation}
\label{property2}
\frac{\ud}{\ud x}\mc{A}(x| \gamma) = \mc{A}(x| \gamma - 1)\ ,
\end{equation}
\item When $x>0$, $ 2 \gamma^3+27 x^2> 0$ and $\frac{\gamma^2}{x^{4/3}}\ll1$, one has the following saddle-point estimate
\begin{equation}
\mc{A}(x| \gamma) \sim
    \frac{1}{\sqrt{12\pi}}\,(x/2)^{\gamma/3 -
          2/3}e^{\left[3(x/2)^{2/3}\left(1 + \mathcal{O}\left( {\gamma^2}{x^{4/3}}\right)\right)\right]}\left(1 + \mathcal{O}( {\gamma}/{x^{2/3}})\right)\  . \label{SP1}
\end{equation}
It turns out that properties 2 and 3 are also
satisfied by another solution of \eqref{eq:Almkvistde}, $g(x| \gamma)$.
This function was used by Almkvist in \cite{alm1} in place of $\mc{A}(x| \gamma)$ and it has the following integral representation:
\begin{equation}
g(x| \gamma) =  \int_{-\infty}^{(0^+)} \frac{\ud t}{2\pi i}\ t^{-\gamma}\,\exp\left[\frac{1}{t^2} + xt\right]\ ,
\end{equation}
where the contour now is the Hankel contour.  A quick look at the contours for $\mc{A}(x|\gamma)$
and $g(x|\gamma)$ shows that they differ by the two semi-infinite segments $\tr{Im}(z) = \pm i\epsilon$, $\tr{Re}(z) \leq 0$ (refer to Figure \ref{hankelcontour}).
\begin{figure}[htbp!]
\centering
%\begin{center}
%\scalebox{0.6}{\input{diff.pdf_t}}
\includegraphics[scale=0.55]{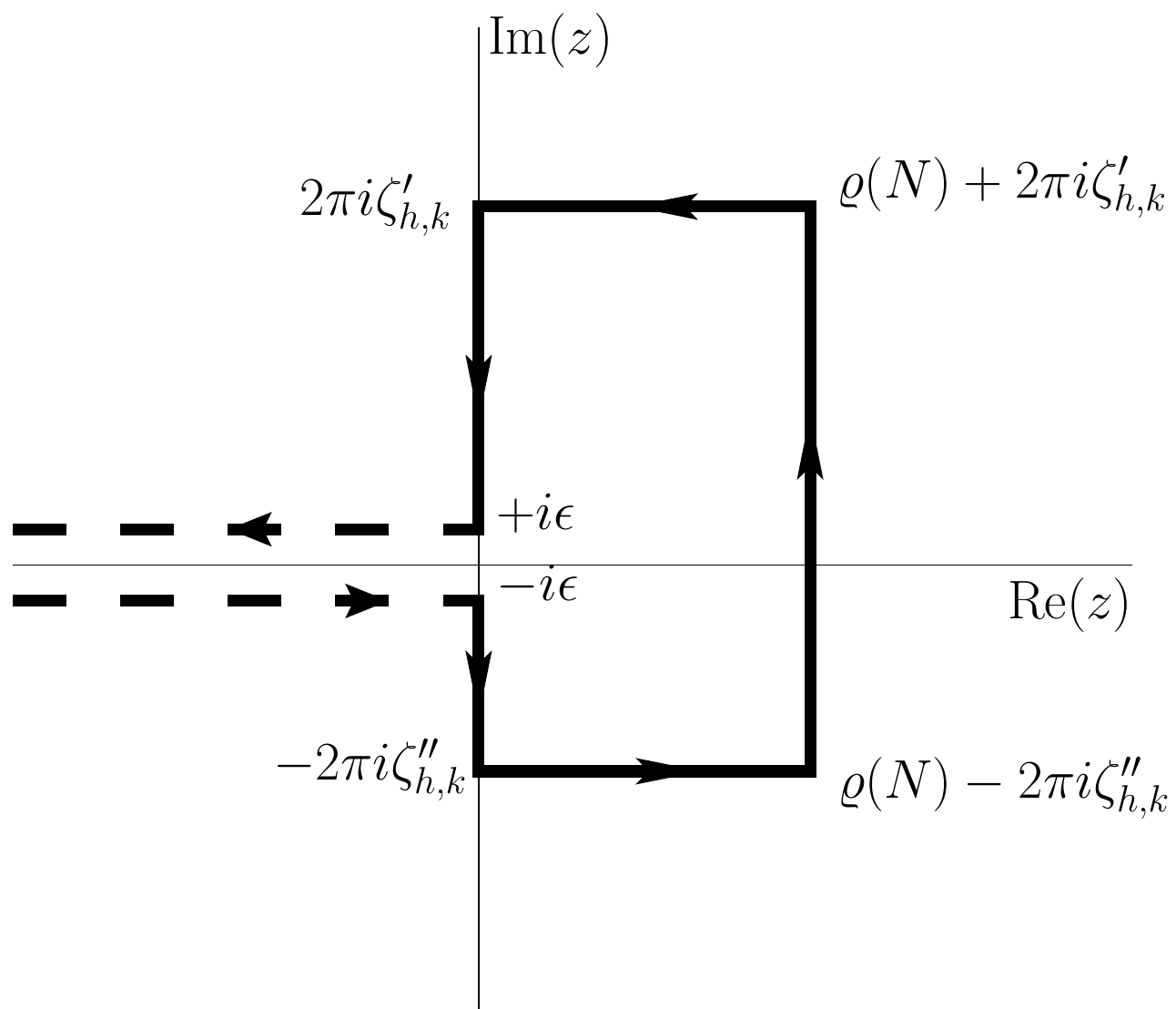}
\caption{The dotted lines along with the solid lines form the Hankel
  contour used in $g(x, \gamma)$ whereas the solid lines form the
  contour for $\mc{A}(x, \gamma)$.} \label{hankelcontour}
%\end{center}
\end{figure}
The value of the integral over these two segments is \textit{not}
negligible  for small $x$. It turns out that subtracting these two
contributions from $g(x| \gamma)$ renders the series
\eqref{eq:p2nwithDfull} over $(h,k)$ convergent. This is tantamount to
replacing $g$ by $\mc{A}$.  The replacement of one solution of
\eqref{eq:Almkvistde}, $g(x| \gamma)$, by another better-behaved
solution, $\mc{A}(x| \gamma)$ that is naturally chosen by the circle
method parallels the Rademacher modification to the formula for the partition function by Hardy and Ramanujan.

\item The corrections indicated in Eq. \eqref{SP1} are due to the shift in the saddle-point that can be organized as a power series in $\lambda:=\frac{-\gamma x^{-2/3}}{3\ 2^{1/3}}$. One has (see Appendix \ref{SaddlePoint})
\begin{equation}
\mc{A}(x| \gamma) \sim
    \tfrac{1}{\sqrt{12\pi}}\,(x/2)^{\gamma/3 -
          2/3}e^{\left[3(\frac{x}2)^{2/3}\left(1 + f_1(\lambda)\right)\right]}\times \Big(1 + f_2(\lambda)\Big)\  , \label{SP1a}
\end{equation}
where the functions $f_1(\lambda)$ and $f_2(\lambda)$ are defined in Eq. \eqref{f1def} and Eq. \eqref{f2def} respectively. For positive $\lambda$, both are monotonically decreasing functions of $\lambda$. When $\lambda\approx 1.2$, $(1+f_1(\lambda))\approx 0$ and $(1+f_2(\lambda))\approx 0.29$. Thus for $\lambda\approx 1.2 $, the Almkvist function is given by
\begin{equation}
\mc{A}(x| \gamma)\Big|_{\lambda=1.2} \sim 0.047
\,(x/2)^{-2.4 (x/2)^{2/3} -
          2/3}\  .
\end{equation}
The functions $f_1(\lambda)$ and $f_2(\lambda)$ have the following series expansions 
\begin{align}\label{f1exp}
(1+f_1(\lambda))&= \begin{cases}
 1-\lambda^2 +\frac{\lambda^3}3 -\frac{\lambda^5}6 + \cdots , & \textrm{as } \lambda \rightarrow0 \\[10pt] 
\lambda - \lambda \log (3\lambda) + \frac{2}{\sqrt{27\lambda}}-\frac1{54 \lambda^2}+\cdots   &\textrm{as }  \lambda \rightarrow +\infty
 \end{cases}\quad , \\
 (1+f_2(\lambda))&= \begin{cases}
 1-\frac{3\lambda}2 +\frac{11\lambda^2}8-\frac{13\lambda^3}{48}  + \cdots , & \textrm{as } \lambda \rightarrow 0 \\[10pt] 
\frac1{\sqrt{6}\lambda} - \frac{5}{36\sqrt2\lambda^{5/2}} +\cdots   &\textrm{as }  \lambda \rightarrow \infty
 \end{cases}\quad .\label{f2exp}
 \end{align}

\end{enumerate}

\noindent{\textbf{Historical remark:} The asymptotics of plane partitions was originally worked out by Wright\cite{Wright1931}. However, a typographical error in the form of a missing factor of $3^{-1/2}$  in the main formula given in \cite[Eq. (2.21)]{Wright1931}, has lead to an erroneous formula permeating the literature. This was pointed out by Mutafchiev and Kamenov who provided the corrected formula\cite{Mutafchiev2006}. With this in mind, we provide the formula again. This follows from the $k=1$ term in Eq. \eqref{eq:p2nwithDfull} after using the saddle point estimate \eqref{SP1} for the Almkvist function.
\begin{equation*}
\boxed{
p_2(n)\sim \frac{ \zeta(3)^{7/36}}{\sqrt{12\pi}}\ \left(\frac{n}{2}\right)^{-25/36} \ \exp\left(3\ \zeta(3)^{1/3} \left(\frac{n}2\right)^{2/3}+ \zeta'(-1)\right)\ .} \tag{corrected Eq. (2.21) of \cite{Wright1931}}
\end{equation*}
%Equivalently, one has 
%\[n^{-2/3} \log p_2(n) \sim 2.00945 - n^{-2/3}\left(0.69444 \log n +1.46310\right)\ .
%\]

\section{A (super)asymptotic formula for $p_2(n)$}
Let us rewrite Eq. \eqref{eq:p2nwithDfull} as follows
\begin{align} 
  p_2(n) = \sum_{k=1}^{\infty} \phi_k(n) =  \sum_{k=1}^{\infty}\sum_{\substack{h = 1\\(h,k) =
      1}}^{k-1} \psi_{h,k}(n)\ ,
\end{align}
where
 \begin{align}
 \psi_{h,k}(n)=k^{-1}\left(\tfrac{a}{k}\right)^{\frac{1}{2} + \frac{k}{24}}
  & e^{-2\pi ni \frac{h}{k} +k\zeta'(-1)+C_{h,k}+
    V_{h,k}(D)}\mc{A}\left(({a}{k^{-3}})^{\frac 12}n\,\Big|
    \tfrac{-k}{12}\right)\ . \label{phikn}
\end{align}

\subsection{An asymptotic series for $V_{h,k}(z)$}

Let us look at the term $L_{h,k}(z)$ in $V_{h,k}(z)$. It is
represented as an integral \eqref{eq:Lhkz}. We can shift the contour
$\tr{Re}(s) = -1-\epsilon$ to more and more negative values of
$\tr{Re}(s)$ and pick up the residues $v^{(p)}_{h,k}\ z^p$ from
\eqref{eq:Ressp} corresponding to the poles of the integrand at $s =
-p$, for integer $p > 1$. This was first done in Almkvist
\cite{alm2}. Then, we can write
\begin{equation}
  L_{h,k}(z) =  \wt{L}^{(M)}_{h,k}(z) + R^{(M)}_{h,k}(z)\ ,
\end{equation}
where $\wt{L}^{(M)}_{h,k}(z) = \sum_{p=2}^M v^{(p)}_{h,k}\,z^p$ and
$R^{(M)}_{h,k}(z)$ is the \emph{error} (or \emph{remainder}) term:
\begin{equation}
  R^{(M)}_{h,k}(z) = \int^{-M-\epsilon +
    i\infty}_{-M-\epsilon - i\infty} \frac{\ud s}{2\pi i}\,(zk^2)^{-s}\,\Gamma(s)\sum_{d,d'=1}^k\,\omega_{h,k}^{dd'}\ \zeta(s-1,\tfrac{d'}{k})\ \zeta(s+1,\tfrac{d}{k})\ .
\end{equation}
Define $\wt{V}^{(M)}_{h,k}(z)=  v^{(1)}_{h,k}z + \wt{L}^{(M)}_{h,k}(z)$.  Formally, taking $M\rightarrow\infty$ in $\wt{V}^{(M)}_{h,k}(z)$, we obtain a power series for $V_{h,k}(z)$ of the form
\begin{equation}
\widetilde{V}_{h,k}(z)=  \sum_{m=1}^\infty v^{(m)}_{h,k}\,z^m \ ,
\end{equation}
with  $v^{(m)}_{h,k}$ for $m>1$ is defined in Eq. \eqref{eq:vphk}.
We use a different notation for the power series as it will turn out that the power series  is asymptotic to $V_{h,k}(z)$. 

\begin{lemma}\label{eq:errorRMhk} As $|z|\rightarrow 0$, one has
\begin{equation} 
    \left|R^{(M)}_{h,k}(z)\right| = \left|{v_{h,k}^{(M+1)}z^{M+1}}\right| + \mathcal{O}(\tfrac1M)\  .
\end{equation}
\end{lemma}
\begin{proof}
This follows from Eq. \eqref{Rhkboundfinal} that we have proved in  appendix \ref{Lhkasymptotics}.
\end{proof}
\begin{prop}
$\wt{V}_{h,k}(z)$ is asymptotic to $V_{h,k}(z)$ as $z\rightarrow 0$.
\end{prop}
\begin{proof}
Using Lemma \ref{eq:errorRMhk}, we see that 
\begin{equation}
\left|V_{h,k}(z)-\wt{V}^{(M)}_{h,k}(z)\right|=\left|V_{h,k}(z)-\sum_{m=1}^M v^{(m)}_{h,k}z^m\right| \leq \left|v^{(M+1)}_{h,k}z^{M+1}\right|\ . \label{vhkasymptotic}
\end{equation}
%since $\left|R^{(M)}_{h,k}(z)/(v^{(M+1)}_{h,k}\ z^{M+1})\right| \leq 1$ as $|z| \rightarrow 0$.
\end{proof}

\subsection{An asymptotic series for $e^{V_{h,k}(z)}$}

Now consider the series
\begin{equation}
e^{\widetilde{V}_{h,k}(z)} := \sum_{m=0}^\infty b^{(m)}_{h,k}\,z^m \ , \label{bmdef}
\end{equation}
which defines the coefficients $b^{(m)}_{h,k}$ with $b^{(0)}_{h,k}=1$.  Let $P_{\ell}(m)$ be the set of all partitions of $m$ into $\ell$ parts of the form $\rho=1^{\delta_1} 2^{\delta_2}\cdots m^{\delta_m}$. Then, $m=\sum_j j \delta_j$ and $\ell=\sum_j \delta_j$. One then has
\begin{equation}
b^{(m)}_{h,k} =  \sum_{\ell=1}^m \sum_{\rho\, \in\, P_\ell (m)}   \prod_{j=1}^{m} \frac{(v^{(j)}_{h,k})^{\delta_j}}{\delta_j!}\ . \label{vmtobm}
\end{equation}
For example,
$$
b^{(4)}_{h,k} =  \left(v^{(4)}_{h,k}  + v^{(1)}_{h,k} v^{(3)}_{h,k} + \frac{(v^{(2)}_{h,k} )^2}{2!} + \frac{(v^{(1)}_{h,k} )^2\, v^{(2)}_{h,k}}{2!} +  \frac{(v^{(1)}_{h,k} )^4}{4!}\right)\ .
$$
\begin{prop}\label{bmvmprop} For large $m$, one has
$$
\left|b^{(m)}_{h,k}-v^{(m)}_{h,k}\right| \leq  \left|v^{(m)}_{h,k}\right| \left(\tfrac{k a }{\pi m}+ \mathcal{O}(\tfrac1{m^2})\right)\ .
$$
\end{prop}
\begin{proof}
Using the bound \eqref{Shkpbound}, we  have for large $m$
\begin{equation}
| v^{(m)}_{h,k}| =\mathcal{O}\left( \tfrac{4}{2\pi} \left(\tfrac{k}{2\pi}\right)^{2m+1} \tfrac{(m+1)!}{m}\right) \ . \label{vmasymp}
\end{equation}
Using this bound along with Eq. \eqref{vmtobm}, we see that
\begin{align}
\left|b^{(m)}_{h,k}-v^{(m)}_{h,k}\right| 
&\leq \left|\sum_{\ell=2}^m   \prod_{j=1}^{m} \sum_{\rho\in P_\ell(m)} \tfrac{(v^{(j)}_{h,k})^{\delta_j}}{\delta_j!} \right| \nonumber  \\
& \leq   \left|\sum_{s=1}^{[m/2]}  \ v^{(m-s)}_{h,k}\ v^{(s)}_{h,k}\right|+ \cdots \nonumber \\
& =   \mathcal{O}\left(\left|v^{(m)}_{h,k}\right| \tfrac{k a }{\pi m}\right)\ .% + \mathcal{O}(\tfrac1{m^2})\right)\ .
\end{align}
In the second line above, the ellipsis denotes contributions from partitions with three or more parts.  It is easy to see that such terms are $\mathcal{O}(\tfrac1{m^2})$.  The last line follows since only the $\ell=2,\ s=1$ term contributes a term of order $1/m$. 
\end{proof}
\noindent This suggests that  the series $e^{\wt{V}_{h,k}(z)}$ is asymptotic to $e^{V_{h,k}(z)}$  as we prove next.
\begin{lemma} \label{Vhkasymp}
$e^{\wt{V}_{h,k}(z)}$ is asymptotic to $e^{V_{h,k}(z)}$ as $z\rightarrow 0$. In particular, one has
\begin{equation}
e^{V_{h,k}(z)}- \big[e^{\wt{V}_{h,k}(z)}\big]_M = \mathcal{O}( b^{(M+1)}_{h,k}\,z^{M+1})\ ,   \label{errorexpV} 
\end{equation}
where $\big[e^{\wt{V}_{h,k}(z)}\big]_M = \sum_{m=0}^M b^{(m)}_{h,k}\,z^m$ is the truncation of the power series  $e^{\wt{V}_{h,k}(z)}$ to order $z^M$.
\end{lemma}
\begin{proof} Consider
\begin{equation}
e^{V_{h,k}(z)}= e^{\wt{V}_{h,k}^{(M)}} e^{R_{h,k}^{M}(z)} =  e^{\wt{V}_{h,k}^{(M)}}  (1 + R_{h,k}^{(M)}(z)+\tfrac12 (R_{h,k}^{(M)}(z))^2 +\cdots)\ .
\end{equation}
Since $|R_{h,k}^{(M)}(z)|=\mathcal{O}( z^{M+1})$, one has
\begin{equation}
e^{V_{h,k}(z)}- e^{\wt{V}_{h,k}^{(M)}} = R_{h,k}^{M}(z)+ \mathcal{O}(z^{2M+2}). \label{step1}
\end{equation}
Also,
\begin{align}
e^{\wt{V}_{h,k}^{(M)}} &= \big[e^{\wt{V}_{h,k}(z)}\big]_M + \left(b^{(M+1)}_{h,k}- v^{(M+1)}_{h,k} \right)\ z^{M+1} + \mathcal{O}(z^{M+2})\ . \label{step2} 
% &= \big[e^{\wt{V}_{h,k}(z)}\big]_M + \left(\tfrac{2\beta}{M+1}\ v^{M+1}_{h,k} +\mathcal{O}(\tfrac1{M^2})\right)\ z^{M+1} + \mathcal{O}(z^{M+2})\ . \label{step2}
\end{align}
Adding Eqs. \eqref{step1} and \eqref{step2} and using Proposition \ref{bmvmprop}, we obtain
\begin{align}
\left|e^{V_{h,k}(z)}- \big[e^{\wt{V}_{h,k}(z)}\big]_M\right |\ \leq\  \left|R_{h,k}^{(M)}(z)\right| + \mathcal{O}\left(\tfrac{ka}{\pi (M+1)}\, |v^{(M+1)}_{h,k} z^{M+1}|\right)\ .%+ \mathcal{O}(z^{M+2})\ .
\end{align}
In the limit $z\rightarrow 0$ and for large $M$, we obtain on using Lemma \ref{eq:errorRMhk} that
\begin{equation}
\left|e^{V_{h,k}(z)}- \big[e^{\wt{V}_{h,k}(z)}\big]_M\right| = \mathcal{O}(v^{(M+1)}_{h,k} z^{M+1})=\mathcal{O}(b^{(M+1)}_{h,k} z^{M+1})\ .
\end{equation}
%We thus see that
%\begin{equation}
%e^{V_{h,k}(z)}- \big[e^{\wt{V}_{h,k}(z)}\big]_M =\mathcal{O}(v^{M+1}_{h,k} z^{M+1})=\mathcal{O}(b^{M+1}_{h,k} z^{M+1})\ , 
%\end{equation}
thus proving the lemma.

%This implies that $\left|\frac{e^{\wt{V}_{h,k}^{(M)}} - \big[e^{\wt{V}_{h,k}(z)}\big]_M}{v^{M+1}_{h,k}}\right|=\mathcal{O}(\tfrac1M)$.

%Exponentiating  Eq. \eqref{vhkasymptotic}, we see that
%\begin{align}
%e^{V_{h,k}(z)} &= e^{\sum_{m=1}^M v^{(m)}_{h,k}z^m} e^{\mathcal{O}(v^{(M+1)}_{h,k}z^{M+1})} \nonumber \\
%&= e^{\sum_{m=1}^M v^{(m)}_{h,k}z^m} \left(1+ \mathcal{O}(v^{(M+1)}_{h,k}z^{M+1})+  \mathcal{O}((v^{(M+1)}_{h,k})^2\ z^{2M+2})+\cdots \right)
%\end{align}
%Using the above equation we see that
%\begin{equation}
%e^{V_{h,k}(z)} - e^{V^{(M)}_{h,k}(z)}  = e^{\sum_{m=1}^M v^{(m)}_{h,k}z^m}  \mathcal{O}(v^{(M+1)}_{h,k}z^{M+1}) = \mathcal{O}(v^{(M+1)}_{h,k}z^{M+1})\ .
%\end{equation}
%Define $\big[e^{\wt{V}_{h,k}(z)}\big]_M = \sum_{m=0}^M b^{(m)}_{h,k}\,z^m$. It is easy to see that $[e^{V_{h,k}(z)}]_M -e^{V^{(M)}_{h,k}(z)}  = \mathcal{O}(z^{M+1})$. Thus, we see that
%\begin{equation}
%\mathcal{E}^{(M)}_{h,k}(z):=e^{V_{h,k}(z)} - [e^{\wt{V}_{h,k}(z)}]_M   = \mathcal{O}(v^{(M+1)}_{h,k}z^{M+1})\ , \label{errorexpV}
%\end{equation}
%proving the proposition. 
\end{proof}

\subsection{The superasymptotic truncation to $e^{\widetilde{V}_{h,k}(z)}$}

We see from Lemma \ref{Vhkasymp}, in particular, Eq.~\eqref{errorexpV}
that the truncated series $[e^{\wt{V}_{h,k}(z)}]_M$ has an error
bounded by $|b^{(M+1)}_{h,k}z^{M+1}|$. Using
Proposition \ref{bmvmprop} and Eq.~\eqref{vmasymp}, we then have
\begin{equation}
 \left|b^{(M+1)}_{h,k}z^{M+1}\right| \sim \left|v^{(M+1)}_{h,k}z^{M+1}\right|\approx \tfrac{k}{\pi^2}  \left(\tfrac{|z|k^2}{4\pi^2}\right)^M \tfrac{(M+1)!}{M}  \approx  \tfrac{k}{\pi^2}\, |w|^M M!\ , \label{eq:errorEMhk} 
\end{equation}
where $w = zk^2 / 4\pi^2$. Using Stirling's formula for $M!$, we can
show that the above error is minimum at $M = M^* \equiv 1/|w|$. The error can thus be minimized if we truncate $[e^{\wt{V}_{h,k}(z)}]_M$ at a value of $M = M^*$, for fixed $|z|$. This is known as the \emph{Superasymptotic truncation} \cite{Berry1990}. We denote this minimal error by $\mc{E}^{\tr{s.a.}}_{h,k}(|z|)$. Then,
we have
\begin{equation}
 e^{V_{h,k}(z)}- \big[e^{\wt{V}_{h,k}(z)}\big]_{M^*} \leq  \left|\mc{E}^{\tr{s.a.}}_{h,k}(|z|)\right| = \tfrac{k}{\pi^2}\ e^{-\tfrac1{|w|}} =  \tfrac{k}{\pi^2}\ e^{\frac{-4\pi^2}{|z|k^2}}\ .  \label{saerror1}
\end{equation}
Thus, we have
\begin{lemma} For $|z| < |z|_{\mathrm{max}}$ and $\textrm{Re}(z) >0$
\begin{equation}
 e^{V_{h,k}(z)}=\big[e^{\wt{V}_{h,k}(z)}\big]_{M^*}+  \mc{O}\left(\mc{E}^{\tr{s.a.}}_{h,k}\left(|z|_{\mathrm{max}}\right)\right)\ .
\end{equation}
with  $M^*=\tfrac{4\pi^2}{|z|_{\rm max}k}$ being the superasymptotic truncation point of
$\wt{V}_{h,k}(z)$.
\end{lemma}
 Approximating $e^{V_{h,k}(z)}$ with
$[e^{\wt{V}_{h,k}(z)}]_{M^*}$, we see that the error is exponentially
suppressed in $|z|_{\tr{max}}$ which makes it an error beyond all
orders in a power series expansion. This is a characteristic of
superasymptotic approximations in general. What this hints at is that, if
at all there is an exact formula for $e^{V_{h,k}(z)}$, we should be
able to obtain it by adding such exponentially suppressed pieces to
the superasymptotic truncation of $e^{\wt{V}_{h,k}(z)}$.

\subsection{A superasymptotic formula for $p_2(n)$}
With the above preparations,  we now obtain a superasymptotic formula for $p_2(n)$. We rewrite  Eq.~\eqref{phikn} by replacing $V_{h,k}(D)$ with the power series 
$\widetilde{V}_{h,k}(D)$ to get 
 \begin{align}
  \wt{\psi}_{h,k}(n) &:=\,k^{-1}\left(\tfrac{a}{k}\right)^{\frac{1}{2} + \frac{k}{24}}
   e^{-2\pi ni \frac{h}{k} +k\zeta'(-1)+C_{h,k}+
    \widetilde{V}_{h,k}(D)}\mc{A}\left(({a}{k^{-3}})^{\frac 12}n\,\Big|
    \tfrac{-k}{12}\right)\ ,\\
\wt{\phi}_k(n) &:= \sum_{\substack{h = 1\\(h,k) =
      1}}^{k-1}\wt{\psi}_{h,k}(n)\ .
\end{align}
Using Eq. \eqref{bmdef} we can write
\begin{align}
  \widetilde{\phi}_k(n)&=\!\sum_{\substack{h = 1\\(h,k) =
      1}}^{k-1}\! k^{-1}\left(\tfrac{a}{k}\right)^{\frac{1}{2} + \frac{k}{24}}
   e^{-2\pi ni \frac{h}{k} +k\zeta'(-1)+C_{h,k}} \sum_{m=0}^\infty  b^{(m)}_{h,k}\,D^m  \mc{A}\left((\tfrac{a}{k^{3}})^{\frac 12}n\,\Big|
    \tfrac{-k}{12}\right)  
    \nonumber \\ 
    &:=\!\sum_{\substack{h = 1\\(h,k) =
      1}}^{k-1} \sum_{m=0}^\infty  \psi_{h,k}^{(m)}(n) := \sum_{m=0}^\infty \phi_k^{(m)}(n)\ ,
%&=\!\sum_{\stackrel{1\leq h < k}{(h,k) =
%      1}}\! k^{-1}\left(\tfrac{a}{k}\right)^{\frac{1}{2} + \frac{k}{24}}
 %  e^{-2\pi ni \frac{h}{k} +k\zeta'(-1)+C_{h,k}} \sum_{m=0}^\infty  \left(\tfrac{a}{k^{3}}\right)^{\frac m2}  b^{(m)}_{h,k} \mc{A}\left((\tfrac{a}{k^{3}})^{\frac 12}n\,\Big|
%    \tfrac{-k}{12}-m\right)
\end{align}
where the second line defines $\phi_k^{(m)}(n)$ as well as $ \psi_{h,k}^{(m)}(n)$.
\begin{prop} \label{ratioestimate} Let  $z_{\mathrm{SP}}=\left(\frac{2a}{nk^{3}}\right)^{\frac 13}$,   $c_2=3\ 2^{-2/3}a^{1/3}$, $\lambda=\frac{k^2}{24 c_2 n^{2/3}}$ and  $\lambda'=\frac{k(k+12m)}{24 c_2 n^{2/3}}$.  Assume that for fixed $n$, $m$ and $k$ are such   that  $ (\lambda'-\lambda)=\frac{km}{2 c_2 n^{2/3}}=\mathcal{O}(n^{-1/3})$. Then
%For $k<k_{cr}(n)$ and $m<m_{cr}(n,k)$, one has
\begin{equation}
\frac{D^m  \mc{A}\left((\tfrac{a}{k^{3}})^{\frac 12}n\,\big|
    \tfrac{-k}{12}\right)}{\mc{A}\left((\tfrac{a}{k^{3}})^{\frac 12}n\,\big|
    \tfrac{-k}{12}\right)}  \sim \left( z_{\mathrm{SP}}\ e^{ \frac12 f_1'(\lambda)}\right)^m \times e^{\frac{k}{8 c_2 n^{2/3}} f_1''(\lambda)m^2} + \mathcal{O}(n^{-1/3})\ .
\end{equation}
%where 
%\begin{equation}
%\begin{split}
%k_{cr}(n)&:=6 (\tfrac{a}2)^{1/6}\ n^{1/3}\sim5.512\ n^{1/3} \ , \\
% m_{cr}(n,k)&:=\tfrac{3a^{1/3}}{2^{1/3}} \tfrac{n^{2/3}}{k} - \tfrac{k}{12}\ ,
% \end{split}
% \end{equation}
%  and  and
\end{prop}
\begin{proof}
% We can use property 2 (see \eqref{property2}) of the Almkvist function to get rid of the differential operator $D=d/d n$ in the above 
%equation.
%$$
%\frac{D^m  \mc{A}\left((\tfrac{a}{k^{3}})^{\frac 12}n\,\big|
%    \tfrac{-k}{12}\right)}{\mc{A}\left((\tfrac{a}{k^{3}})^{\frac 12}n\,\big|
%    \tfrac{-k}{12}\right)} = \left(\frac{a}{k^{3}}\right)^{\frac m2}  \frac{ \mc{A}\left((\tfrac{a}{k^{3}})^{\frac 12}n\,\big|
%    \tfrac{-k}{12}-m\right)}{\mc{A}\left((\tfrac{a}{k^{3}})^{\frac 12}n\,\big|
%    \tfrac{-k}{12}\right)}
%$$
%For the major arcs, the saddle point estimate given in Eq. \eqref{SP1} may be used to show that (the formula below is valid for $\tfrac{km}{n^{2/3}} < 1$)
%The saddle point estimate in Eq. \eqref{SP1} is valid only when the arguments of the Almkvist function $\mc{A}(x|\gamma)$ are such that $(27 x^2+2\gamma^3)>0$.  In order to use Eq. \eqref{SP1}, we need
%$k<k_{cr}(n)$ for $\mc{A}((\tfrac{a}{k^{3}})^{\frac12}n\,\big| \tfrac{-k}{12})$ and $m<m_{cr}(n,k)$ for  $\mc{A}((\tfrac{a}{k^{3}})^{\frac 12}n\,\big|
%    \tfrac{-k}{12}-m)$. We thus require both conditions to hold.

Using the saddle point estimate  given in Eq. \eqref{SP1a} for the Almkvist function  as well as  property 2 (see \eqref{property2}) of the Almkvist function, we get
\begin{equation}
\begin{split}
\frac{D^m  \mc{A}\left((\tfrac{a}{k^{3}})^{\frac 12}n\,\big|
    \tfrac{-k}{12}\right)}{\mc{A}\left((\tfrac{a}{k^{3}})^{\frac 12}n\,\big|
    \tfrac{-k}{12}\right)} &\approx \left(\frac{a}{k^{3}}\right)^{\frac m2}  \left((\tfrac{a}{k^{3}})^{\frac 12}\tfrac{n}2 \right)^{-m/3} \frac{1+f_2(\lambda')}{1+f_2(\lambda)} e^{\frac{c_2 n^{2/3}}{k}(f_1(\lambda')-f_1(\lambda))}\\
    &\approx \left( z_{\mathrm{SP}}\ e^{ \frac12 f_1'(\lambda)}\right)^m \times e^{\frac{k}{8 c_2 n^{2/3}} f_1''(\lambda)m^2}\  .
    \end{split}
\end{equation}
In the second line above, we use $ \frac{1+f_2(\lambda')}{1+f_2(\lambda)}  = 1+\mathcal{O}(n^{-1/3})$ and 
\begin{equation*}
\frac{c_2 n^{2/3}}{k}(f_1(\lambda')-f_1(\lambda))\sim \tfrac12 f_1'(\lambda) m + \frac{km^2}{8 c_2 n^{2/3}} f_1''(\lambda) +\mathcal{O}(n^{-1/3})\ .
\qedhere
\end{equation*}
\end{proof}
%where $z_{\mathrm{SP}}=\left(\frac{2a}{nk^{3}}\right)^{\frac 13}$.
\noindent Using Proposition \ref{ratioestimate}, we can rewrite  $\psi^{(m)}_{h,k}(n)$ in the following form (with  $\lambda=\frac{k^2}{24 c_2 n^{2/3}}$)
\begin{equation}
 \frac{\psi^{(m)}_{h,k}(n)}{\psi^{(0)}_{h,k}(n)}\sim   \Big[b^{(m)}_{h,k}  \left( z_{\mathrm{SP}}\ e^{ \frac12 f_1'(\lambda)}\right)^m \times e^{\frac{km^2}{8 c_2 n^{2/3}} f_1''(\lambda)} \Big]\, \  .
 \end{equation}
 Using Proposition \ref{bmvmprop} and Eq. \eqref{vmasymp}, we see that for large $m$, 
 \begin{equation*}
 | b^{(m)}_{h,k}| =\mathcal{O}\left( \tfrac{4}{2\pi} \left(\tfrac{k}{2\pi}\right)^{2m+1} \tfrac{(m+1)!}{m}\right)\ .
 \end{equation*}
 For large $m$, one thus has
 \begin{equation}
\left| \frac{\psi^{(m)}_{h,k}(n)}{\psi^{(0)}_{h,k}(n)}\right|\sim  \frac{2k}{\pi \sqrt{2\pi m}} \times 
\exp\left(m \log [m \widetilde{w}_{\mathrm{SP}}] -m +\frac{km^2}{8 c_2 n^{2/3}} f_1''(\lambda)\right)\ . \label{psiratio1}
%\left( w_{\mathrm{SP}}\ e^{ \frac12 f_1'(\lambda)}\right)^m \times e^{\frac{km^2}{8 c_2 n^{2/3}} f_1''(\lambda)} \Big]\, \  .
 \end{equation}
 where $\widetilde{w}_{\mathrm{SP}} := \frac{z_{\mathrm{SP}}k^2 e^{\frac12f'(\lambda)}}{4\pi^2}$. As $m$ increases, the ratio decreases until $m=M^*(n,k)$ after which it increases as is typical of an asymptotic series. The estimate for the superasymptotic truncation point $M^*(n,k)$ is
\begin{equation}
\begin{split}
M^*(n,k) &= \frac1{\widetilde{w}_{\mathrm{SP}}} - \frac{2}{\widetilde{w}_{\mathrm{SP}}^{\, 2}} \frac{k}{8 c_2 n^{2/3}} f_1''(\lambda)\\
&= \frac{c(\lambda)}{k} \, n^{1/3}   - \tfrac{(c(\lambda))^2}{4c_2 k} f_1''(\lambda)\ , \label{Mstarvalue} \,
\end{split}
\end{equation}
where $c(\lambda):= \tfrac {4\pi^2e^{ -\frac12 f_1'(\lambda)}}{(2a)^{1/3}}$. We thus obtain
  \begin{equation}
\left| \frac{\psi^{(M^*+1)}_{h,k}(n)}{\psi^{(0)}_{h,k}(n)}\right|\sim  \frac{2k}{\sqrt{2\pi^3 M^*}} \times 
\exp\left(-\tfrac{c(\lambda)}k \ n^{1/3} + \tfrac{(c(\lambda))^2}{8c_2k} f_1''(\lambda)\right)\ . \label{psihkratio}
%\left( w_{\mathrm{SP}}\ e^{ \frac12 f_1'(\lambda)}\right)^m \times e^{\frac{km^2}{8 c_2 n^{2/3}} f_1''(\lambda)} \Big]\, \  .
 \end{equation}
Since $f''(\lambda)<0$ for positive $\lambda$, the effect of the  $e^{\frac{km^2}{8 c_2 n^{2/3}} f_1''(\lambda)} $ term in Eq. \eqref{psiratio1} is to make $M^*(n,k)$ larger and to reduce the
value of the ratio in Eq. \eqref{psihkratio}. Further this term is sub-leading in the limit of large $n$ and fixed  $\lambda$. Extending arguments used in Lemma \ref{Vhkasymp}, we can show that the series $\wt{\psi}_{h,k}(n)$ is asymptotic to  $\psi_{h,k}(n)$ for large $n$.
%These observations enable us to extend arguments (not given here) used in Lemma \ref{Vhkasymp}, to account for the $e^{\frac{km^2}{8 c_2 n^{2/3}} f_1''(\lambda)} $ term appearing in Eq. \eqref{psiratio1}, to show that the series, $\wt{\psi}_{h,k}(n)$  for $\psi_{h,k}(n)$ is also an asymptotic one. One has as $n\rightarrow \infty$
 %since $\sum_{m=0}^\infty  b^{(m)}_{h,k}  z_{\mathrm{SP}}^m = e^{\widetilde{V}_{h,k}(z_{\mathrm{SP}})}$ from Eq. \eqref{bmdef} that defines  $b^{(m)}_{h,k}$.  
%We immediately see that the asymptotic nature of the series associated with $e^{\widetilde{V}_{h,k}(z_{\mathrm{SP}})}$ (as shown in Lemma \ref{Vhkasymp}) implies that 
\begin{equation}
\psi_{h,k}(n) - \sum_{m=0}^M \psi^{(m)}_{h,k}(n) = \mathcal{O} ( \psi^{(M+1)}_{h,k}(n))\ , \label{psihkasymp}
\end{equation}
leading to the following proposition.
%The estimate for the superasymptotic truncation point $M^*(n,k)$ is
%\begin{equation}
%k M^*(n,k)= c(\lambda) \ n^{1/3}\ e^{ -\frac12 f_1'(\lambda)}  - \tfrac{(c(\lambda))^2}{4c_2} f_1''(\lambda)\ e^{-f_1'(\lambda)}\ , \label{Mstarvalue} \,
%\end{equation}
%where $c(\lambda):= \tfrac {4\pi^2}{(2a)^{1/3}}$. The  additional $e^{\frac{km^2}{8 c_2 n^{2/3}} f_1''(\lambda)} $ term leads to the second term appearing in the above formula. Since
%$\tfrac{kM^*(n,k)}{n^{2/3}} = \tfrac{4\pi^2(2a)^{1/3}e^{ -\frac12 f_1'(\lambda)} }{n^{1/3}}=\mathcal{O}(n^{-1/3})$
%for large $n$, the  use of Proposition \ref{ratioestimate} is justified post facto. Consider the case when $kn^{-1/3} < 1$. Then $\lambda \rightarrow 0$ and we can set $f'(\lambda)\sim f_1'(0)=0$ and $f''(\lambda)\sim f_1''(0)=-2$. In this situation, we obtain
%\begin{align}
%k M^*(n,k) &= c(\lambda) n^{1/3} + \tfrac{(c(\lambda))^2}{2c_2} + \mathcal{O}(n^{-1/3}) \nonumber \\
%&\sim 29.47 n^{1/3} + 216.09+ \mathcal{O}(n^{-1/3})\ . \label{Mstarfinala}
%\end{align}
%\item Next, consider the case when $k = \beta n^{1/3}$ for $1\leq \beta \leq 3$. One has $\lambda= \frac{\beta^2}{24 c_2}$ and thus one has $0.02<\lambda< 0.19$.  For this range of $\lambda$, one has $1.02<e^{ -\frac12 f_1'(\lambda)}<1.20$ and  $-2>e^{ - f_1'(\lambda)}f''(\lambda)>-2.32$. This shows that we can use formula \eqref{Mstarfinala} for $M^*$ keeping in mind that it is a slight underestimate when $k\sim 3 n^{1/3}$.
%\end{enumerate}

\begin{prop}\label{phiasymp} One has $\phi_k(n)  \sim  \widetilde{\phi}_k(n)$ as $n\rightarrow \infty$ since
\begin{equation}
\phi_k(n) -\sum_{m=0}^M \phi_k^{(m)}(n) = \mathcal{O}\left(\phi_k^{(M+1)}(n)\right)\ .
\end{equation}
%where 
%$$
%\phi_k^{(m)}(n):=
%\!\sum_{\stackrel{1\leq h < k}{(h,k) =
%      1}}\! k^{-1}\left(\tfrac{a}{k}\right)^{\frac{1}{2} + \frac{k}{24}}
%   e^{-2\pi ni \frac{h}{k} +k\zeta'(-1)+C_{h,k}}  b^{(m)}_{h,k} D^m \mc{A}\left((\tfrac{a}{k^{3}})^{\frac 12}n\,\Big|
%    \tfrac{-k}{12}\right)  \ .
%$$
\end{prop}
\begin{proof}
Recall that $\phi_k(n)$ is given by a finite sum that runs over all $h\in[1,k-1]$ with $(h,k)=1$.  Hence
\begin{align}
\phi_k(n) -\sum_{m=0}^M \phi_k^{(m)}(n) &= \!\sum_{\substack{h=1\\(h,k) =
      1}}^{k-1} \! \! \big(\psi_{h,k}(n) -\sum_{m=0}^M \psi_{h,k}^{(m)}(n)\big) \nonumber \\
      &= \!\sum_{\substack{h=1\\(h,k) =
      1}}^{k-1} \!\! \mathcal{O}( \psi_{h,k}^{(M+1)}(n)) \nonumber \\
      &= \mathcal{O}(\phi_k^{(M+1)}(n))\ ,
\end{align}
where we have used Eq. \eqref{psihkasymp} in obtaining the second line in the above equation. 
%The conditions on $k$ and $M$ follow from the conditions under which we have 
%proved Eq. \eqref{psihkasymp}.  We believe that the above relation holds more generally but that has not been proved (as we will not require it).
\end{proof}
%\begin{mydef} Let us call an arc, for the fraction $h/k$, in the Farey dissection of the circle with $k\leq k_{cr}(n)$ as a type I arc and type II arc otherwise.
%\end{mydef}
%For type I arcs, we must use the saddle point approximation Eq. \eqref{SP1} while for type II arcs, we must use saddle point approximation Eq. \eqref{SP2}. We will show that type II arcs have exponentially suppressed contributions and can be dropped. Thus, we obtain the following superasymptotic formula for $p_2(n)$.

\subsubsection{Error Estimates}

We would like to provide more explicit error estimates which we do next. Since the Almkvist function is complicated to deal with directly, we work with the saddle-point approximation given in Eqs. \eqref{SP1a}. 
\begin{prop}\label{term2bound} One has the following bound (with $\lambda =\frac{k^2n^{-2/3}}{24c_2}$)
\begin{align}
 |\phi^{(0)}_k(n)| \leq
\tfrac{ c_1^k(k^2n^{-2/3})^{1+ \frac k{24}}}{(2a)^{-1/6}\sqrt{6\pi k^3}} \  e^{\frac{c_2n^{2/3}}{k}\left(1 + f_1(\lambda)\right)}\times \Big(1 + f_2(\lambda)\Big)  \ , \end{align}
 where we define $c_1=(2a)^{1/36}2^{-\alpha/12}\exp(\zeta'(-1))$.
 %and  $c_2=4^{-\alpha}24a\ e^{\left(24 \zeta'(-1)+1+\sqrt{\frac{16}{27}}\right)}\sim 3.196\ 4^{-\alpha}$. 
\end{prop}
\begin{proof}
One has
\begin{equation}
\phi^{(0)}_k(n) = \sum_{\stackrel{1\leq h < k}{(h,k) =
      1}}\,k^{-1}\left(\tfrac{a}{k}\right)^{\frac{1}{2} + \frac{k}{24}}
  e^{-2\pi ni \frac{h}{k} +k\zeta'(-1)+
    C_{h,k}}\mc{A}\left(({a}{k^{-3}})^{\frac 12}n\,\Big|
    \tfrac{-k}{12}\right)\ .
\end{equation}
%where the superscript $(0)$ indicates that we are considering the zeroth-term in the series expansion 
%of $e^{V_{h,k}(z)}$. 
Using the upper bound in Eq. \eqref{bound3} for $C_{h,k}$ one has
\begin{align}
|e^{C_{h,k}}|  &\leq e^{k\log k/12} e^{-\alpha k\log2/12} = \left(\tfrac{k}{2^\alpha}\right)^{k/12}\ .
\end{align}
 The saddle point estimate for the Almkvist function in Eq. \eqref{SP1a} gives
 \begin{align}
\mc{A}\left(({a}{k^{-3}})^{\frac 12}n \,\Big|
    \tfrac{-k}{12}\right)&\sim \sqrt{\tfrac1{12\pi}} \left(\sqrt{\tfrac{a}{k^3}} \tfrac n2\right)^{-\tfrac{k}{36}-\tfrac23}\ e^{ \tfrac{c_2 n^{2/3}}k\left(1 + f_1(\lambda)\right)}\times \Big(1 + f_2(\lambda)\Big) \ ,
\end{align}
leading to the following bound 
\begin{equation}
%\begin{split}
|\phi^{(0)}_k(n)| 
 \leq\tfrac{ c_1^k(k^2n^{-2/3})^{1+ \frac k{24}}}{(2a)^{-1/6}\sqrt{6\pi k^3}} \     e^{\frac{c_2n^{2/3}}{k}\left(1 + f_1(\lambda)\right)}\times \Big(1 + f_2(\lambda)\Big) \ .
\qedhere
\end{equation}
 \end{proof}
\noindent The parameter $\lambda$ naturally controls various expansions. We can trade all occurrences of $n$ for $\lambda$ to rewrite the bound as follows:
\begin{equation}
%\begin{split}
|\phi^{(0)}_k(n)| 
%& \leq \tfrac{c^k  (k^2n^{-2/3})^{1+ \frac k{24}}}{(a/2)^{1/6}\sqrt{6\pi k^3}} \    e^{\frac{c_2 n^{2/3}}{k}\left(1 + f_1(\lambda)\right)}\times \Big(1 + f_2(\lambda)\Big) \ ,\\
 \leq \sqrt{\tfrac{432a}{\pi k^3}}\times  d(\lambda)^{\frac{k}{24}} \times \Big(\lambda + \lambda f_2(\lambda)\Big)\ ,
% \leq \tfrac{c^{k} \left(72 a \lambda\right)^{1+ \frac k{24}}}{\sqrt{12\pi a k^3}} \    e^{\frac{k}{24\lambda}\left(1 + f_1(\lambda)\right)}\times \Big(1 + f_2(\lambda)\Big) \ ,
%& \leq \tfrac{ (1.63\ \lambda\ 4^{-\alpha})^{1+ \frac k{24}}}{c^{24} (a/2)^{1/6}\sqrt{6\pi k^3}} \    e^{\frac{c_2 n^{2/3}}{k}\left(1 + f_1(\lambda)\right)}\times \Big(1 + f_2(\lambda)\Big) \ .
%\end{split}
\end{equation}
where 
\begin{equation} \label{eq:ddef}
d(\lambda)=72 a\lambda \, 4^{-\alpha}  \exp\left(24\zeta'(-1)+\frac{1+f_1(\lambda)}{\lambda}\right)\ .
\end{equation}

 Using the properties of the functions $f_1(\lambda)$ and $f_2(\lambda)$ as well as their expansions as given in Eqs.\! \eqref{f1exp} and \eqref{f2exp}, we can show the following.
 \begin{enumerate}
 \item For positive $\lambda$, $d(\lambda)$ is a monotonically decreasing positive function and $(\lambda + \lambda f_2(\lambda))$ is a monotonically increasing function.
% \item For $\lambda\gg 1$, $d(\lambda)\rightarrow c_2(\lambda)$ where $c_2(\lambda):= 24a\ 4^{-\alpha}\ e^{\left(24 \zeta'(-1)+1+\sqrt{\frac{4}{27}}\lambda^{-3/2}\right)}$. Further
% for all $\lambda>0$, one has  $d(\lambda)\leq c_2(\lambda)$. Thus, we see that the monotonically decreasing function $c_2(\lambda)$ provides an upper bound on $d(\lambda)$.
 \item As $\lambda\rightarrow\infty$, one has $(\lambda + \lambda f_2(\lambda))\rightarrow \tfrac1{\sqrt{6}}$. Further for all $\lambda>0$, one has $(\lambda + \lambda f_2(\lambda))\leq \tfrac1{\sqrt{6}}$.
 \end{enumerate}
% \begin{align}
% |\phi^{(0)}_k(n)| 
% &\leq \sqrt{\tfrac{72a}{\pi k^3}}\ c_2(\lambda)^{\ k/24}\ .
% \end{align}

\begin{prop}\label{term2bound2}
The monotonicity of $d(\lambda)$ and  $(\lambda + \lambda f_2(\lambda))\leq \tfrac1{\sqrt{6}}$ implies that for any $\lambda_0>0$
 \begin{align}
 |\phi^{(0)}_k(n)| 
 &\leq \sqrt{\tfrac{72a}{\pi k^3}}\ d(\lambda_0)^{\ k/24}\quad \textrm{for } \lambda \geq \lambda_0\ .
 \end{align}
 \end{prop}
 We see that the value of $\lambda$, call it $\lambda_c$, when $d(\lambda)=1$ is special. At $\lambda=\lambda_c$, we see that  $|\phi^{(0)}_k(n)| \sim k^{-3/2}$ which implies that contributions are small and values of $k$ such that $\lambda>\lambda_c$ can be neglected. Indeed, $\lambda = \frac{k^2 n^{-2/3}}{24c_2}$ gives, for $\lambda = \lambda_c$, $k = k_c \equiv \sqrt{24c_2\lambda_c}\ n^{1/3}$. Thus this determines the minor and major arcs as we show in more detail later.
 From Prop. \ref{conjecture}, we choose $\alpha=3$ to numerically compute  $\lambda_c=0.18$ and $k_c = 2.948\ n^{1/3}$.

\subsubsection{The estimate for the error from the superasymptotic truncation}
 With our estimate for $M^*(n,k)$ given in Eq.\! \eqref{Mstarvalue}, we next estimate $\phi_{k}^{(M^*+1)}(n)$. This is the error due to the superasymptotic truncation of $\wt{\phi}_{k}(n)$. Using Eq.\,\eqref{psihkratio}, we get
   \begin{equation}
\left| \phi^{(M^* + 1)}_{k}(n)\right|\sim  \frac{2k}{\pi \sqrt{2\pi M^*}}\, 
\exp{\left(-\tfrac{c(\lambda)\, n^{1/3}}{k} + \tfrac{c(\lambda)^2}{8c_2 k} f_1''(\lambda) \right)}\ \left|\phi^{(0)}_{k}(n)\right| \ , \label{phikratio}
%\left( w_{\mathrm{SP}}\ e^{ \frac12 f_1'(\lambda)}\right)^m \times e^{\frac{km^2}{8 c_2 n^{2/3}} f_1''(\lambda)} \Big]\, \ ,
 \end{equation}
%\begin{align}
%|\phi_{k}^{(M^*+1)}(n)| \leq \tfrac{k}{\pi^2}\ e^{-M^*(n,k)} |\phi_{k}^{(0)}(n)| =  \tfrac{k}{\pi^2}\ e^{- \frac {c(\lambda) n^{1/3}}{k}}\  |\phi_{k}^{(0)}(n)|  
%\end{align}
which when combined with the estimate for $ |\phi_{k}^{(0)}(n)|$ in Proposition \ref{term2bound} gives
\begin{multline}
  \left|\phi_{k}^{(M^*+1)}(n)\right| \leq \tfrac{k^{-1/2} c_1^k(k^2n^{-2/3})^{1+ \frac k{24}}}{\pi^2  (2a)^{-1/6}\sqrt{3 M^*}} \times \Big(1 + f_2(\lambda)\Big) \\ 
\times \exp\left(-\tfrac{c(\lambda)n^{1/3}}k  + \tfrac{c(\lambda)^2}{8c_2k} f_1''(\lambda)+\tfrac{c_2 n^{2/3}}{k}\big(1 + f_1(\lambda)\big)\right)  . \label{SAerror2}
\end{multline}
For later considerations, we will only to consider the above formula for $0<\lambda\leq \lambda_c = 0.18$. For these values of $\lambda$, $f_1$ and $f_2$ are non-positive decreasing functions of $\lambda$. Using this we can write a slightly weaker but simpler looking bound
\begin{equation}
\left|\phi_{k}^{(M^*+1)}(n)\right| \leq \tfrac{k^{-1/2}c_1^k (k^2n^{-2/3})^{1+ \frac k{24}}}{\pi^2 (2a)^{-1/6}\sqrt{3 M^*}}\    \exp\big(\tfrac1k(- \tfrac{(c(\lambda))^2}{4c_2}-c(\lambda) n^{1/3}+c_2 n^{2/3})\big)\ . \label{SAerrorfinal}
\end{equation}
We use this estimate for low values of $k$ i.e., $kn^{-2/3}<1$ and  $\lambda \rightarrow 0$ for which we can use $f_1(0)=f_1'(0)=f_2(0)=0$ and $f_1''(0)=-2$. The superasymptotic truncation point is then given by
\begin{align}
k M^*(n,k) &= c(0)\, n^{1/3} + \tfrac{c(0)^2}{2c_2} + \mathcal{O}(n^{-1/3}) \nonumber \\
&\approx 29.47 n^{1/3} + 216.09+ \mathcal{O}(n^{-1/3})\ . \label{Mstarfinal}
\end{align}
For $n=7000$ and $k=1$, this gives $M^*=780$ while the (exact) value computed numerically is $M^*=880$. 

Let $n_a$ denote the value of $n$ when the superasymptotic truncation error becomes $\mathcal{O}(1)$. 
The estimate for the value $n_a$ at which the superasymptotic truncation leads to errors of $\mathcal{O}(1)$ is now obtained as a solution to the quadratic equation (ignoring 
prefactors that do not appear in the exponential)
\begin{equation}
c_2\, n_a^{2/3}-c(0)\, n_a^{1/3} - \frac{c(0)^2}{4c_2}=0\quad \implies \quad \boxed{n_a = \left(\tfrac{c(0)}{c_2} \tfrac{1+\sqrt2}2\right)^3\approx 5540}\ . 
\end{equation}
However, we can do a better job numerically by dealing directly with
the Almkvist function rather than its saddle-point approximation to 
determine $M^*(n,k)$. We obtain $n_a \approx 6400$ which is slightly
larger than our estimate of $n_a = 5540$.

\begin{lemma} \label{SAtruncationp2}
Let $M^*(n,k)$ denote the superasymptotic truncation point for fixed $n$ and $k$. Then, one has
\begin{equation} \label{eq:kconverge}
p_2(n) \sim \sum_{k=1}^{\infty} \sum_{m=0}^{M^*(n,k)}\left(\ \phi_{k}^{(m)}(n) + \mathcal{O}\left(\phi_{k}^{(M^*(n,k)+1)}(n)\right)\right)\ ,
\end{equation}
where Eqns.\! \eqref{SAerror2} and \eqref{SAerrorfinal} may be used to determine the truncation errors.
\end{lemma}
\begin{proof} 
We only need to prove that the sum over $k$ is a convergent one. Let $k_c$ denote the value of of $k$ for which $d(\lambda)=1$, i.e. $\lambda = \lambda_c = 0.18$. Then the convergence of the sum in \eqref{eq:kconverge} is determined by the convergence of the following sum:
\begin{equation*}
\left |\sum_{k>k_c} \sum_{m=0}^{M^*(n,k)} \phi_{k}^{(m)}(n) \right|\leq 2 \sum_{k>k_c} |\phi^{(0)}_k(n)|\ ,
\end{equation*}
where we use the asymptotic nature of the series $\sum_{m} \phi_{k}^{(m)}(n)$ to bound it by twice the value of its initial term. More precisely, we have 
\begin{equation}
\left|\sum_{m=0}^{M^*(n,k)} \phi_{k}^{(m)}(n)\right| \sim \left|\phi_{k}^{(0)}(n) + \mathcal{O}(\phi_{k}^{(1)}(n))\right| \ \leq\ 2\, |\phi_{k}^{(0)}(n)|\ ,
\end{equation}
using $|\phi_{k}^{(1)}(n))|<|\phi_{k}^{(0)}(n))|$ when $M^*(n,k)>1$. When $M^*(n,k)=1$ (which occurs for large enough $k$ at fixed $n$) then there is only one term, the $m=0$ term whose magnitude is clearly less than twice itself.
We can then use Proposition \ref{term2bound2} with $d(\lambda_c)=1$ in the above formula to see that
\begin{equation*}
\left |\sum_{k>k_c} \sum_{m=0}^{M^*(n,k)} \phi_{k}^{(m)}(n) \right| \ \leq \ 2 \sum_{k>k_c} \sqrt{\frac{72a}{\pi k^3}} = \ \sqrt{\frac{288a}{\pi}} \left(\zeta\!\left(\tfrac32\right) -\sum_{k=1}^{k_c} \frac1{k^{3/2}}\right)\ ,
\end{equation*}
which is finite and hence the sum over $k$ is convergent.
%For fixed $n$ and $k$ such that $\lambda = \frac{k^2}{24 c_2 n^{2/3}}>1.2$, $(1+f_1(\lambda))<0$ showing that the terms are exponentially suppressed. 
%As $\lambda$ further increases, the suppression increases and we should use second bound given in Proposition \ref{term2bound}. This leads to a convergent geometric series as long as $c_2\sim 4^{0.31-\alpha}<1$. This needs $\alpha>0.31$ which is the lower bound for $\alpha$ that we assumed in Conjecture \ref{conjecture}.
\end{proof}

\subsection{Identifying the major arcs}

In Lemma \ref{SAtruncationp2}, we have a convergent sum over $k$ after imposing the superasymptotic truncation in the sum over $m$. We wish to convert the infinite sum over $k$ into a finite sum $k< N(n)$ (the \textit{major arcs}) neglecting the contributions from $k>N(n)$ (the \textit{minor arcs}). The cutoff $k < N(n)$, equivalently $\lambda_N=\frac{N(n)^2}{24c_2n^{2/3}}$, is chosen such that the contribution from all minor arcs put together is negligible. We begin with the bound given in Proposition \ref{term2bound}:
\begin{equation}
 |\phi^{(0)}_k(n)|\ \leq \ 
 \tfrac{ c_1^k(k^2n^{-2/3})^{1+ \frac k{24}}}{ (2a)^{-1/6}\sqrt{6\pi k^3}} \    e^{\frac{c_2 n^{2/3}}{k}\left(1 + f_1(\lambda)\right)}\times (1 + f_2(\lambda)) \equiv \tfrac{1+f_2(\lambda)}{ (2a)^{-1/6}\sqrt{6\pi}}\ e^{f(n,k)}\ , 
 \end{equation}
where $$f(n,k):= \tfrac{c_2n^{2/3}}{k}(1 + f_1(\lambda))+k\log c_1 +(\tfrac{k}{24}+1)\log(k^2/n^{2/3})-\tfrac32 \log k\ .$$ 
A  rough estimate shows that for $N\sim n^{1/3}$, one has $f(n,k)\sim \mathcal{O}(1)$.
The following proposition shows that $N(n)\sim (\beta_1 n^{1/3} + \beta_2 \log n + \beta_3)$ for some constants $\beta_1$, $\beta_2$ and $\beta_3$.
\begin{prop}\label{Nnprop} 
Let
%$$N(n)=(3.00398 n^{1/3}-1.38111 \log n +1.51914+2.76222 (\kappa \log n+ \lambda)\ ,$$
$N(n)\sim (\beta_1 n^{1/3} + \beta_2 \log n + \beta_3)\ ,$
with  $\beta_1$, $\beta_2$ and $\beta_3$ as given by Eqs.\! \eqref{beta1} and \eqref{beta2}.
Then, one has 
\begin{align}
f\!f(n) &:= f(n,N(n)) =-\kappa_2 \log n- \kappa_3 + \mathcal{O}(n^{-1/3}(\log n)^2)\ , \label{fvalue}\\
f\!f'(n) &:= \frac{\partial f(n,k)}{\partial k}\Big|_{k=N(n)} \approx -0.494 
%+ \frac{0.243072 (1-2 \kappa ) \log \left(\frac{1}{n}\right)-0.486145
  % \lambda +0.433811}{\sqrt[3]{n}}
  + \mathcal{O}(n^{-1/3}\log n)\ . \label{fderivative}
\end{align}
\end{prop}
\begin{proof}
The proof is mostly computational. Set $N(n)=\beta_1 n^{1/3} + \beta_2 \log n + \beta_3$ and expand $f\!f(n):=f(n,N(n))$ as a power series in $n$ for large $n$. It has the form
\begin{equation}
f\!f(n)=  t_1(\beta_1)\ n^{1/3} + t_2(\beta_1,\beta_2)\ \log n + t_3(\beta_1,\beta_3) + \mathcal{O}(n^{-1/3}(\log n)^2)\ ,
\end{equation}
where (with $X=\frac{\beta_1^2}{24c_2}$)
\begin{equation}
\begin{split}
t_1(\beta_1) &=\tfrac{c_2}{\beta_1}\left[1+f_1(X)\right] + \tfrac{\beta_1}{12} \log (c_1^{12}\beta_1)\ ,\\ 
t_2(\beta_1,\beta_2)&= \beta_2\big(\tfrac{1+f_1'(X)}{12}-\tfrac{2c_2(1+f_1(X))}{\beta_1^2} + \tfrac{t_1(\beta_1)}{\beta_1}\big)-\tfrac12\textrm{ and }\\
t_3(\beta_1,\beta_3)&= \beta_3\big(\tfrac{1+f_1'(X)}{12} -\tfrac{2c_2(1+f_1(X)}{\beta_1^2} + \tfrac{t_1(\beta_1)}{\beta_1}\big)+\tfrac{\log\beta_1}2\ .
\end{split}
\end{equation}
We first observe that
\begin{equation}
t_1(\beta_1) = \frac{\beta_1}{24}\log d(X)\ ,
\end{equation}
where $d(X)$ is defined in \eqref{eq:ddef}. We then set $t_1(\beta_1)=0$ to get rid of the coefficient of $n^{1/3}$ in the series expansion for $f\!f(n)$. This is nothing but solving for $d(X) = 1$ which we already know corresponds to $X = 0.180$ (cf. discussion after Proposition \ref{term2bound2}).This gives
%
%
%We can solve for $\beta_1$ using the Lambert function, $W(z)$, which solves the equation $W(z) e^{W(z)}=z$\cite{Corless1996}. Some care is needed in deciding the correct real branch.  We choose  $\beta_1^2 c^{24}=e^{W_{-1}(-24 c_2c^{24})}$ or 
\begin{equation}
\beta_1 \approx 2.948\ . \label{beta1}
\end{equation}
Using $f_1(0.180)=-0.031$ and $f_1'(0.180)=-0.329$ and requiring $t_2(\beta_1,\beta_2)=-\kappa_2$ and $t_3(\beta_1,\beta_3)=-\kappa_3$ (where $\kappa_3$ and $\kappa_3$ are positive real constants) gives 
\begin{equation}
\begin{split}
\beta_2 &\approx -1.468+2.936\ \kappa_2 = -1.468 (1 - 2\kappa_2)\ , \\
\beta_3 &\approx 1.587+2.936\ \kappa_3\ . \label{beta2}
\end{split}
\end{equation}
For these values of $\beta_1$, $\beta_2$ and $\beta_3$, we see that $f(n,N(n))=-\kappa_2 \log n -\kappa_3 + \mathcal{O}(n^{-1/3}(\log n)^2)$ thus proving Eq.\! \eqref{fvalue}. Further, a simple numerical computation (not shown) leads to Eq.\! \eqref{fderivative}.
%Substituting this value of $\beta_1$ in the expressions for $t_2$ and $t_3$, we obtain 
%$t_2=(-0.362027  \beta_2-\tfrac{1}{2})$ and $t_3=(-0.362027 \beta_3+0.549969)$. Choosing $\beta_2=-1.38111+2.76222 \kappa$ sets $t_2=-\kappa$ and $\beta_3= 1.51914+2.76222 \lambda $ sets $t_3=-\lambda$, thus proving the proposition.
\end{proof}
We thus see that for $k=N(n)$ as in  Proposition \ref{Nnprop} and $(1+f_2(0.180))\sim 0.772$,
$$
 |\phi^{(0)}_k(n)|\ \leq\ \tfrac{1+f_2(0.180)}{ (2a)^{-1/6}\sqrt{6\pi}}\ \exp[f\!f(n)] \sim 0.21\ n^{-\kappa_2} e^{-\kappa_3 + \mathcal{O}(n^{-1/3})}\ .
$$
We will choose $\kappa_2>0$ and $\kappa_3>0$ such that the contribution of the minor arcs can be
neglected.  We find that the numbers obtained for $N(n)$ with
$\kappa_2=\kappa_3=0$ tends to be comparable to  the numerically
computed cut-off. For instance, for $n=7000$, we find that $k\approx
44-45$ works quite well while $[N(7000)]=45$.  Non-zero positive values for $\kappa_2$ and
$\kappa_3$ only boost $N(n)$ to larger values.

%\begin{prop} For some fixed $n=n_0$, let $N_0>0$ be the solution to the following equation  $$f(n_0,N)= \tfrac{c_2n_0^{2/3}}{N} + N \log c + (\tfrac{N}{24}+1)\log\tfrac{N^2}{n_0^{2/3}}-\tfrac32 \log N=0\ ,$$ such that $\nu_0:=(N_0/n_0^{2/3})<1$. Then, for $n\geq n_0$, we can choose
%\begin{equation}
%N(n)= \sqrt{\frac{c_2}\beta}\ n^{1/3} - \frac{\gamma}{2\beta}\ ,
%\end{equation}
% with $\beta= -\log (c\ N_0^{1/12}n_0^{-1/36})$ and $\gamma=\log N_0^{-1/2} \nu_0$.  
%The error from the truncation in the sum over $k$ is  $\mathcal{O}(1)$ or better.
%%$\mathcal{O}\left(n^{-1/6}\ \left[\tfrac{\nu(n)}{\nu_0}\right]^{\frac{1+N(n)}{24}}\ \right)$.
%\end{prop}
%Proof: We have $f(n,N)< \tfrac{c_2n^{2/3}}{N} + N \log c - (\tfrac{N}{24}+1)\log\tfrac{n_0^{2/3}}{N_0^2}-\tfrac32 \log N_0$.
%
%For instance, for $n_0=750$, one finds $N_0=20.75$ and $\nu_0=0.25$ leading to $N(n)= 2.86n^{1/3} - 5.90$ for $n\geq 750$. An  exact numerical computation gives $N_0=18$ -- thus the formula tends to overestimate the value of $N(n)$. A  fit of the numerically determined values of $N(n)$ for various values of $n$  to a formula of the form $(\kappa_1 n^{1/3}+\kappa_2)$ gives $N(n)=2.34 n^{1/3} - 4.24$. This  works reasonably well up to $n=100000$.

\subsection{The contribution from the minor arcs}
Recall that 
\begin{equation}
\phi_k(n) \sim \wt{\phi}_k(n) = \sum_{m=0}^{\infty} \phi_k^{(m)}(n)\ .
\end{equation}
We will truncate the above asymptotic series at $m=M^*(n,k)$ to get:
\begin{equation} \label{eq:phiksuperasymp}
|\phi_k(n)| \leq \sum_{m=0}^{[M^*(n,k)]} |\phi_k^{(m)}(n)| \sim \left|\phi_{k}^{(0)}(n) + \mathcal{O}(\phi_{k}^{(1)}(n))\right| \ \leq  2\, |\phi_k^{(0)}(n)|\ .
\end{equation}
Thus, we can use 
\begin{equation}
|\phi_k(n)| < 2\, |\phi_k^{(0)}(n)|\ .
\end{equation}
This implies that, up to a multiplicative $n$-independent constant (a factor of 2), it suffices to work with the bound on $|\phi_k^{(0)}(n)|$ as given in Proposition \ref{term2bound} or \ref{term2bound2}.

In the discussion after Proposition \ref{term2bound2}, we reasoned that the contribution from $k > k_c = \beta_1 n^{1/3}, d(\lambda) < 1$ is negligible and labelled these the \emph{minor arcs}. In the above Proposition, we saw that the contribution from $k > N(n) = \beta_1 n^{1/3} + \beta_2 \log n + \beta_3$ are in fact negligible. If we choose values of $\kappa_2$ and $\kappa_3$ such that $N(n) < \beta_1 n^{1/3}$, then we have some $k$ for which $d(\lambda) > 1$ as well. We thus divide our minor arcs $k > N(n)$ into two types:
\begin{enumerate}
\item Type I arcs: $N(n) < k < \beta_1 n^{1/3}$ for which $d(\lambda) > 1$.
\item Type II arcs: $k > \beta_1 n^{1/3}$ for which $d(\lambda) < 1$.
\end{enumerate}
Let us study the contribution from each of these two types next.

\subsubsection{The contribution from Type I minor arcs}
Recall from Proposition \ref{Nnprop} that $e^{f\!f(n)} \sim n^{-\kappa_2} e^{-\kappa_3}$ and $f\!f'(n) \approx -0.494$. Then, the contribution from the Type I arcs is given by:
\begin{align}
\textrm{Type I}& < 2 \sum_{k>N(n)}^{[\beta_1 n^{1/3}]} |\phi^{(0)}_k(n)| = 2\,|\phi^{(0)}_N(n)| \times \sum_{k>[N(n)]}^{[\beta_1 n^{1/3}]} \left|\frac{\phi^{(0)}_k(n)}{ \phi^{(0)}_N(n)}\right|\ , \nonumber \\
& \leq 2\tfrac{1 + f_2(0.180)}{(2a)^{-1/6}\sqrt{6\pi}}\,e^{f\!f(n)} \sum_{k>[N(n)]}^{[\beta_1 n^{1/3}]}  e^{f\!f'(n) (k-N)} = 0.42\, e^{f\!f(n)}\, \frac{1-e^{-f\!f'(n)(\beta_2 \log n + \beta_3)}}{1-e^{f\!f'(n)}}  \nonumber \\
&= 1.06\, e^{f\!f(n)}\, (1-n^{-f\!f'(n)\beta_2}e^{-\beta_3 f\!f'(n)}) < 1.06\,e^{f\!f(n)} \approx 1.06\, n^{-\kappa_2} e^{-\kappa_3}\ . \label{eq:typeI}
\end{align} 
where, in the last line, we have assumed that $\kappa_2$ is chosen such that $\beta_2 < 0$ so that we can write $n^{-\beta_2 f\!f'(n)} < 1$.
\subsubsection{The contribution from Type II minor arcs}
The contribution from Type II arcs can similarly be estimated using the bound given by Proposition \ref{term2bound2}. 
Since $k > k_c = \beta_1 n^{1/3}$, we have $d(\lambda) < 1$. Choose some $\lambda_0 > \lambda_c$ in Proposition \ref{term2bound2}. Then, we get
\begin{align}
\textrm{Type II} & < 2\, (\tfrac{72a}\pi)^{1/2}\! \sum_{k>\beta_1 n^{1/3}}d(\lambda_0)^{\,k/24}\  k^{-3/2} < \left(\frac{288a}{\pi n \beta_1^3}\right)^{1/2}\! \sum_{k>\beta_1 n^{1/3}}\! d(\lambda_0)^{\,k/24} \\
& = 2.07\ \frac{n^{-1/2}\ d(\lambda_0)^{\beta_1 n^{1/3}/24}}{1-d(\lambda_0)^{1/24}} = \frac{2.07\, n^{-1/2}}{1-d(\lambda_0)^{1/24}}\ e^{-c_3 n^{1/3}}\ ,
\end{align}
where $c_3 =-\frac{\beta_1}{24}\log d(\lambda_0)>0$ since $d(\lambda_0) < 1$.

\subsubsection{Combining the two contributions}
We observe that the contribution from Type II arcs goes to zero
exponentially fast unlike the Type I arcs contribution which goes to
zero as a power law i.e., $n^{-\kappa_2}$. Thus the contributions from
Type I arcs dominates that of Type II arcs, and hence the latter can
be neglected. In conclusion, we see that the contributions of the
minor arcs go as
\begin{equation}
1.06\, n^{-\kappa_2}\, e^{-\kappa_3}\ .
\end{equation}
We choose $\kappa_3 = \log 1.06 \approx 0.06$, to cancel the factor of
$1.06$. Thus, we obtain the following proposition.
\begin{prop} \label{minorarcs} The contribution from the minor arcs  with $\kappa_3=0.06$ is
\begin{equation}
\boxed{
\sum_{k>[N(n)]}^{\infty} \phi_k(n) \sim \mathcal{O}(n^{-\kappa_2})\ ,
}
\end{equation}
which can be made arbitrarily small by suitably choosing
$\kappa_2$.\footnote{If we take a formal limit $\beta_2 \to 0^-$ in
  \eqref{eq:typeI}, we see from the definition of Type I arcs that they
  form a negligible part of the minor arcs. However, from
  \eqref{beta2}, we see that $\kappa_2 \to 0.5^+$. This means the
  dominating power law fall-off holds, with quite a significant exponent of
  $\frac{1}{2} + \epsilon$, even if a thin sliver of type I arcs is present.}
 Our theoretical bounds need us to choose some positive
non-zero values for these constants but our numerical experiments
suggests that it suffices to set $\kappa_2=0$.
\end{prop}

%\begin{align}
%\lim_{k\rightarrow\infty} |\phi^{(0)}_k(n)| &\sim \left({a}{k^{-1}}\right)^{\frac{1}{2} + \frac{k}{24}} \left(\frac{k}{2^\alpha}\right)^{k/12} e^{k\zeta'(-1)}
% \left(\frac{24e}{k}\right)^{k/24}  \frac{\sqrt{72}}{\sqrt{\pi} k}\nonumber \\
% & =\mathcal{O}(c_2^{k/24}\ k^{-3/2})\ ,
% \end{align}
%  where $c_2=4^{-\alpha}24a e^{(24 \zeta'(-1)+1)}\sim 1.48\ 4^{-\alpha}$. We need $c_2<1$ for term to vanish in the $k\rightarrow \infty$ limit -- this needs $\alpha>0.29$ and we observe that something like $\alpha\leq 3$ is consistent with numerical data.
  
\noindent \textbf{Remark:} If we had used the function used by
Almkvist $g(x|\gamma)$ in the place of $\mc{A}(x|\gamma)$, the
contribution of the minor arcs, in particular those of Type II, will
not be negligible, leading to an asymptotic series.  This is similar to
Rademacher's improvement of the asymptotic series of Hardy and
Ramanujan for the numbers of integer partitions. There, the analogs of
$g(x|\gamma)$ and $\mc{A}(x|\gamma)$ were the modified Bessel
functions $\mc{I}_{-3/2}(x)$ and $\mc{I}_{3/2}(x)$. The latter is
better-behaved than $\mc{I}_{-3/2}(x)$ as $x\to 0^+$ and the
replacement makes the contribution of the minor arcs in Rademacher's
formula to be negligible. Of course, our formula is asymptotic for
other reasons as have shown.

\subsection{Proof of the main theorem}
We now restate our main theorem along with its proof.
\newtheorem*{thma}{Theorem \ref{maintheorem}}
\begin{thma}
Let $f_1(\lambda)=-\lambda^2 + \frac{\lambda^3}3 +\mathcal{O}(\lambda^5)$ be the function defined in Eq.\! \eqref{f1def}. Further,
let $a=\zeta(3)$, $c_1=(2a)^{1/36}2^{-\alpha/12}\exp(\zeta'(-1))$, $c_2=3\ 2^{-2/3}a^{1/3}$,  $c(\lambda)= \tfrac {4\pi^2e^{ -\frac12 f_1'(\lambda)}}{(2a)^{1/3}}$,and $\alpha=3$ the constant appearing in Proposition \ref{conjecture}. Then
\begin{equation*}\tag{\ref{eq:mainp2n}}
p_2(n)\sim   \sum_{k=1}^{[N(n)]}\sum_{\substack{h=1\\(h,k)=1}}^{k-1} \psi_{h,k}(n) + \mathcal{O}(n^{-\kappa_2})
\ , 
\end{equation*}
where $N(n)=2.948 n^{1/3}+(2.936\kappa_2 -1.468) \log n + 1.763$ for some $\kappa_2>0$ and
\begin{multline*}\tag{\ref{psinumber}}
\psi_{h,k}(n) =  e^{-2\pi i n h/k +k\zeta'(-1)+C_{h,k}}\ \tfrac{1}{k} \left(\tfrac{a}{k}\right)^{\frac{1}{2} + \frac{k}{24}}\  \sum_{m=0}^{[M^*(n,k)]} b_{h,k}^{(m)} (\tfrac{a}{k^{3}})^{\frac m2} \mathcal{A}\left((\tfrac{a}{k^{3}})^{\frac 12}n\,\big|
    \tfrac{-k}{12}-m\right)\\
  + \mathcal{O}\left(\tfrac{k^{-3/2}c_1^k (k^2n^{-2/3})^{1+ \frac k{24}}}{\pi^2  (2a)^{-1/6}\sqrt{3M^*}}\     \exp\big(\tfrac1k(- \tfrac{c(\lambda)^2}{4c_2}-c(\lambda)\ n^{1/3}+c_2\ n^{2/3})\big)\right) \ . 
\end{multline*}

where $C_{h,k}$ is the generalized Dedekind sum \eqref{Chkdef}, $b_{h,k}^{(m)}$ is defined in Eq.\! \eqref{bmdef},
$\lambda=\frac{k^2n^{-2/3}}{24c_2}$ and $k M^*(n,k) = c(\lambda)  \ n^{\frac13}   - \tfrac{(c(\lambda))^2}{4c_2} f_1''(\lambda)$. 
\end{thma}
\begin{proof}
Since most of the details of the proof have already been worked out, we list out the precise details below.
\begin{enumerate}
\item In Proposition \ref{minorarcs}, we have shown that the
  contribution of the minor arcs is $\mathcal{O}(n^{-\kappa_2})$ if we
  set $\kappa_3\approx 0.06$ in the formula for $N(n)$ in Proposition
  \ref{Nnprop}. This gives $N(n)=2.948 n^{1/3}+(2.936\kappa_2 -1.468)
  \log n + 1.763$.
\item For fixed $(h,k)$, the asymptotic nature of the series $\wt{\psi}_{h,k}(n)$ in Eq. \eqref{psihkasymp} implies Proposition \ref{phiasymp} that shows that the series  $\wt{\phi}_k(n)$ is also an asymptotic one. The superasymptotic truncation point is determined in Eq.~\eqref{Mstarfinal}.
\item Eq.~\eqref{SAerrorfinal} gives the error from the superasymptotic truncation for $\phi_k(n)$. Since the errors are $h$-independent, we see that the error from the superasymptotic truncation for $\psi_{h,k}(n)$ that we quote in the theorem  is $1/k$ times the error given in Eq.~\eqref{SAerrorfinal}.
\end{enumerate}
This completes the proof of the main theorem.
\end{proof}

\section{Evaluating $p_2(n)$ numerically}
Here, we present a numeric analysis of formula \eqref{eq:mainp2n} for
$p_2(n)$. Instead of using the theoretical value $N(n)$ of the cutoff
for $k$, we determine the cutoff value numerically. We shall designate
this cutoff as $\mc{N}(n)$. The computations were carried out using
Mathematica but can be reproduced in similar computer algebra systems
such as Maple and Maxima. We write
\begin{align} \label{eq:nump2n} 
p_2(n) \sim \sum_{k=1}^{\mc{N}(n)}\wt{\phi}_k(n)\ ,\quad \wt{\phi}_k(n) = \sum_{\substack{h=1\\(h,k)=1}}^{k-1} \wt{\psi}_{h,k}(n)\ ,
\end{align}
with $\psi_{h,k}(n)$ as given in Theorem \ref{maintheorem}. We then carry out the following steps. The value of $n$ is fixed
throughout.
\begin{enumerate}
\item We use the Frobenius series for the Almkvist function given in Eq. \eqref{property1} taking care to truncate the sum
at a value large enough so that no errors arise from it. 
\item First, determine the value of the cutoff $\mc{N}$. We do this by
  looking at the value of $k$ where the quantity given below reaches,
  say $0.01$.
  $$
  k^{k/12}\ e^{k\zeta'(-1)}\ \left(\frac{a}{k}\right)^{\frac{1}{2} +
    \frac{k}{24}}\mc{A}\left((\tfrac{a}{k^3})^{\frac 12}n\,\Big|
    \tfrac{-k}{12}\right)\ . 
$$ 
This offers a good estimate for $\mc{N}(n)$ since, in $\wt{\psi}_{h,k}(n)$,
$e^{-2\pi ni \frac{h}{k}}$ is a phase and $e^{\wt{V}_{h,k}(D)}$ gives
rise to sub-leading terms.
\item  We fix $k$ first and consider the  series  $\wt{\phi}_k(n)
  = \sum_{m}\phi^{(m)}_k(n)$. We look at the terms in this series and
  truncate optimally at the minimum term at a value of $m$ that we label as $\mc{M}^*(n, k)$. In the event
  that the value of $\phi^{(m)}_k(n)$ goes below, say $0.001$, for $m<\mc{M}^*(n, k)$,
  we truncate at the earlier value. This occurs typically as $k$ grows. For instance, for $n=6999$, 
  the superasymptotic truncation errors are $>1$ only for $k=1,2,3,4$.
\end{enumerate}
With the cutoff $\mc{N}(n)$ and the truncation numbers $\mc{M}^*(n,k)$
at hand, we evaluate $\phi_k(n)$ for $k \leq \mc{N}(n)$ and sum them
up to obtain an estimate for $p_2(n)$. We first look at $p_2(750)$
which is a 70-digit number:
\begin{align*}
  p_2(750) =
  &254\,574\,302\,435\,864\,503\,952\,192\,074\,902\,485\,957\,295\,901\,059\,651\,237\,103\text{-}\\&467\,858\,692\,796\,6061\
  .
\end{align*} In Table 1, we show the results of our corresponding
numerical computation. It turns out that we need terms up to $k = 17$
in order for $\phi_k(n)$ to be consistently less than 1. The numerical
error turns out be $0.167$.
\begin{table}
\footnotesize
\begin{tabular}{c|r} \hline
$k$ & $\phi_k(750)$\hspace{2in} \\[5pt] \hline
  1 & 2545743024358645039521920749024859571789657217789975418420497702709720.300 \\
 2 & 1169353378721087578836884133296412.054  \\
 3 & 1308038187203153215044.287  \\
 4 & $-$766248063769796.487  \\
 5 & 249747729385.715   \\
 6 & 258376791.876   \\
 7 & $-$3577528.999    \\
 8 & $-$1684.466    \\
 9 & $-$13708.658  \\
 10 & 1766.734   \\
 11 & $-$274.759    \\
 12 & $-$61.857    \\
 13 & $-$6.938    \\
 14 & 0.409    \\
 15 & 2.541\\ 
 16 & $-$0.138 \\
 17 & $-$0.447\\ \hline
 \textbf{Total} & 2545743024358645039521920749024859572959010596512371034678586927966061.167 \\ \hline
\textbf{Exact}& 2545743024358645039521920749024859572959010596512371034678586927966061.000 \\ \hline
 \end{tabular}
 \caption{\small Numerical evaluation of $p_2(750)$. The error compared to the exact value is $0.167$.}
\end{table}\\
\noindent We next study the behaviour of formula \eqref{eq:nump2n} for
$n=6491$. This number was chosen because it is a prime number close to
the value of $n$ for which the superasymptotic truncation in
$\phi_1(n)$ has an error which exceeds 1. For $n = 6491$ and $k=1$,
the superasymptotic truncation occurs at $\mc{M}^*(6491, 1) = 868$ and
we find that $\phi_1^{(868)}(n)\sim -7.10$.  The value of
$\mc{N}(6491)$ turns out to be $41$. This is where the error left
after truncating at $k = \mc{N}(n)$ becomes less than 1. For $k > 1$
it turns out that the superasymptotic truncation does not kick in as
the magnitude of the terms go below $1$ before we reach the
corresponding $\mc{M}^*(n,k)$. $p_2(6491)$ is a 301 digit number and
our error estimates imply that we should get 299-300 digits right as
we see below. We have
\begin{align*}p_2(6491)=
  &2\,435\,999\,812\,007\,724\,505\,361\,175\,276\,591\,271\,187\,423\,253\,814\,389\,347\text{-}\\&742\,142\,058\,647\,311\,447\,856\,919\,196\,957\,669\,606\,748\,334\,139\,672\,693\,539\,\text{-}\\&708\,059\,165\,034\,113\,853\,741\,212\,578\,737\,113\,278\,837\,205\,845\,414\,784\,460\text{-}\\&262\,083\,024\,174\,265\,640\,881\,536\,003\,876\,770\,326\,556\,221\,114\,453\,737\,307\text{-}\\&274\,796\,033\,818\,318\,509\,841\,695\,057\,683\,009\,905\,018\,994\,722\,630\,708\,028\text{-}\\&438\,488\,667\,147\,936\,430\,644\,025\,707\,833\,583\ .\end{align*}
\noindent We exhibit the computation in Table 2. The values from the
superasymptotic approximation is (omitting several digits that agree
with the number given above)
\begin{equation}
\sum_{k=1}^{40}\phi^{\tr{s.a.}}_k(6491)=2\ 435\ldots 580.47\pm\
7.54\ ,
\end{equation}
where $7.54$ is the estimated error due to superasymptotic truncation
which is nothing but the sum of $|\phi_k^{(\mc{M}^*)}(n)|$ for $k =
1,\ldots,40$ as well as the contribution from
$\phi_{41}(n)=-0.0409$. The actual numerical error turns out to be
$-2.58$, which is of the same order of magnitude as that of the
estimated error. This indicates that the above numerical method for
computing $p_2(n)$ according to formula $\eqref{eq:mainp2n}$ truly gives a
superasymptotic approximation to $p_2(n)$.  We also conclude from our numerical studies that this
superasymptotic approximation has errors less than $1$ till around $n
= 6400$ and ceases to do so beyond that.  We have also carried out a similar computation for integers near $7000$ where the estimated and real error are around $10^{10}$. Figure \ref{saerrorplot} illustrates the asymptotic nature of the series $\wt{\phi}_k(n)$ for $n=6999$ and $k=1,2,5$. The estimated error is dominated by the $k=1$ error and the real error that we get from our computation is $-9.9\times 10^9$.
\begin{figure}[htbp!]
\centering
%\begin{center}
%\scalebox{0.6}{\input{phi1m}}
\includegraphics[height=1.1in]{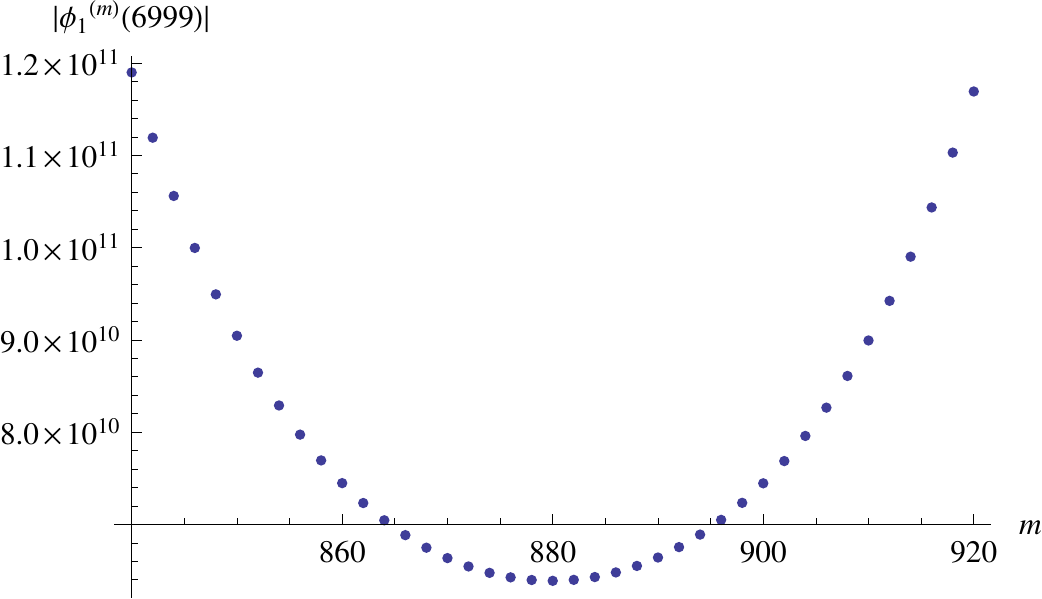}\hfill  \includegraphics[height=1.1in]{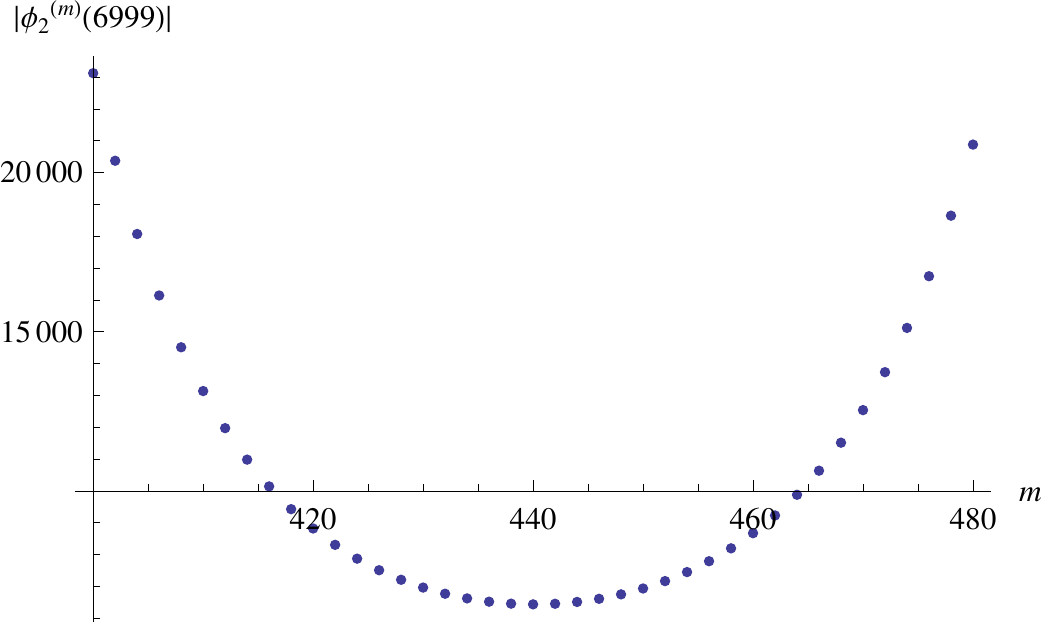} \hfill  \includegraphics[height=1.1in]{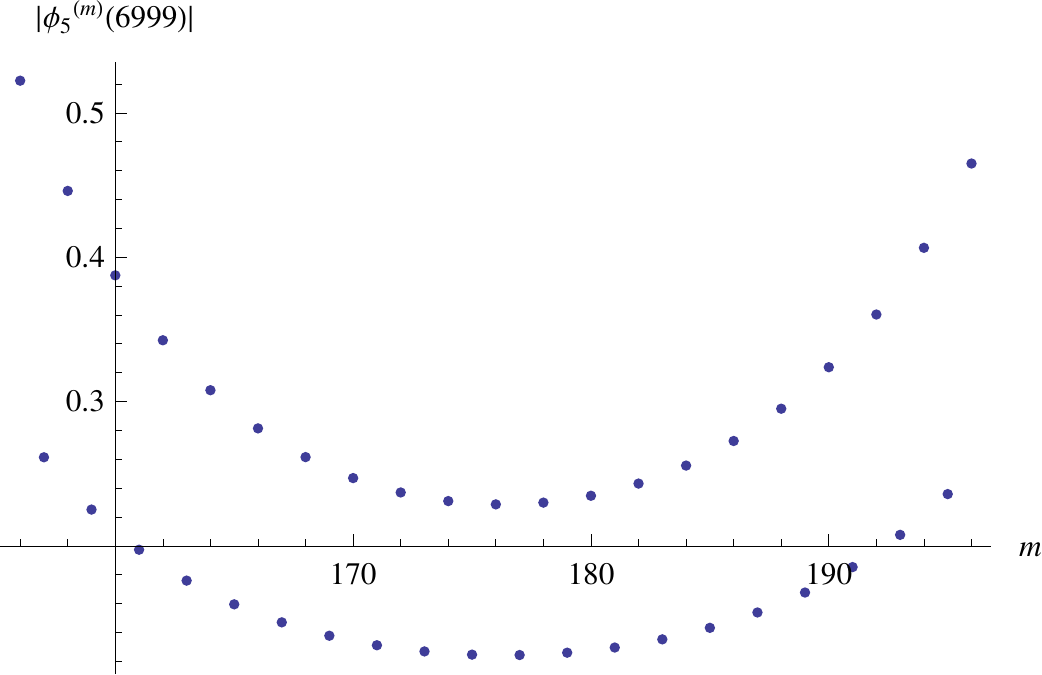}
\caption{For $n=6999$, we show how $\phi_k^{(m)}(n)$ behaves, for $k=1,2,5$, near the superasymptotic truncation point, $\mc{M}^*=\frac{880}k$. The error from the superasymptotic truncation for $k=1$  is about $6.39\times 10^{10}$ while for $k=2$ it is $6438.01$. Note that $\phi^{(m)}_k(n)$ is non-vanishing only for even $m$ for both $k=1$ and $2$.} \label{saerrorplot}
%\end{center}
\end{figure}

\section{Conclusion}

The main result of this paper is to provide a Hardy-Ramanujan-Rademacher type formula for plane partitions using the circle method. This formula turns out to give exact answers for all integers $\lesssim 6400$ and for integers larger than $6400$, the formula is not exact but comes with precise estimates for the error.   There exist methods that extend and improve upon the superasymptotic truncation that we have used. Berry and Howls call these hyperasymptotics and these will reduce the errors introduced by the superasymptotic truncation\cite{Berry1990,Berry1991}. The Mellin-Barnes theory of hyperasymptotics discussed in Paris and Kaminksi\cite{Paris2001} is more appropriate for our considerations since $L_{h,k}(z)$ is defined as a Mellin-Barnes integral in Eq. \eqref{eq:Lhkz}. We believe that our formula is the first step towards a formula that might, at the very least, be exact for integers near $50000$. We are currently carrying out a numerical study to see if we can apply such methods to improve upon our formula. We have been able to reproduce numbers for integers up to $10000$ and hope to report on this in the future. It is also clear that the methods used this paper extend to other non-modular generating functions for which we can make use of the circle method.\\

\noindent \textbf{Acknowledgments:} We wish to thank Matthias Beck and K. Srinivas for discussions and correspondence on Dedekind sums.

\appendix

\section{Evaluation of the residues of $\log P_2(e^{-z}\omega_{h,k})$.}\label{residues}
We follow the treatment of Almkvist in \cite{alm2}. We have
\begin{equation}
  \log P_2\left(e^{-z}\omega_{h,k}\right) = \int^{2 + \delta + i\infty}_{2 + \delta - i\infty} \frac{\ud s}{2\pi
    i}\,(zk^2)^{-s}\,\Gamma(s)\sum_{d,d'=1}^k\,\omega_{h,k}^{dd'}\
  \zeta(s-1,\tfrac{d'}{k})\ \zeta(s+1,\tfrac{d}{k})\ .
\end{equation}
We recall that the integrand on the right hand side has a double pole
at $s = 0$ and simple poles at $s = 2, -1, -2, \ldots$. We express the
right hand side as a sum of residues at these poles by shifting the
contour $\tr{Re}(s) = 2 + \delta$ to $\tr{Re}(s) = -M + \delta$, with
integer $M \to \infty$. The residues are obtained follows.
\subsection{Residue at $s = 2$.}
Only $\zeta(s - 1, \tfrac{d'}{k})$ has a simple pole at $s = 2$ with
residue 1. Hence, the overall residue is given by
\begin{align}
\tr{Res}_{s=2} & = (zk^2)^{-2}\,\Gamma(2)\sum_{d,d'=1}^k\zeta\!\left(3, {d}/{k}\right)e^{2\pi idd'h/k}\ ,\\
 & = z^{-2}k^{-4}\sum_{d=1}^k\zeta\!\left(3, {d}/{k}\right) k\ \delta_{d,k} = \frac{\zeta(3)}{z^2 k^3}\ ,
\end{align}
where we have used $\sum_{d'=1}^k e^{2\pi idd'h/k} = k\ \delta_{d,k}$
with $\delta_{a,b}$ being the Kronecker delta and $\zeta(s, 1) =
\zeta(s)$.
\subsection{The residue at $s = 0$.}
Near $s = 0$, the $s$-dependent part of the integrand looks like
\begin{equation}
  \sum_{d,d'=1}^{k}e^{2\pi idd'h/k}\,(1 - s\log zk^2)\left(\frac{1}{s} - \gamma\right)\bigg(\zeta\left(-1,\tfrac {d'}k\right) + s\,\zeta'\left(-1,\tfrac {d'}k\right)\bigg)\left(\frac{1}{s} + \psi(\tfrac dk)\right)\ ,
\end{equation}
where $\gamma$ is the Euler-Mascheroni constant and
$\psi(x)$ is the Digamma function. The residue is the coefficient of
$1/s$:
\begin{align}
  \tr{Res}_{s=0} & = k\zeta'(-1)  + \frac{k}{12}\log\,(zk) - \sum_{d,d'=1}^{k}e^{2\pi idd'h/k}\left(\gamma + \psi\left(\tfrac dk\right) + \log k\right)\zeta\left(-1,\tfrac {d'}k\right)\ ,\notag\\
  & = k\zeta'(-1)  + \frac{k}{12}\log\,(zk)  + \sum_{d,d'=1}^{k}e^{2\pi idd'h/k}\zeta\left(-1,\tfrac {d'}k\right)\tfrac{\pi}{2}\cot \left(\tfrac{\pi d}{k}\right)\notag \\& \qquad  - \sum_{d,d'=1}^{k}e^{2\pi idd'h/k}\Big[\gamma + \psi\left(\tfrac dk\right) + \log k + \tfrac{\pi}{2}\cot \left(\tfrac{\pi d}{k}\right)\Big]\zeta\left(-1,\tfrac {d'}k\right)\ ,
\end{align}
where we have used $\zeta(-1) = -1/12$ and carried out the sum over
$d, d'$ in the first two terms similar to the $s = 2$ case. Next, we
use the following formulas from \cite{alm2}:
\begin{align}
\frac{\pi}{2}\sum_{d=1}^{k-1}e^{2\pi idd'h/k}\cot (\pi d /k) &= -i \pi k\, B_1(d'h/k)\ ,\\
\sum_{d=1}^{k-1}e^{2\pi idd'h/k}\big[\gamma + \psi\left(\tfrac dk\right) + \log k + \tfrac{\pi}{2}  \cot &\left(\tfrac{\pi d}{k}  \right) \big] = k\,\log\, \big|2\sin (\pi d'h/k)\big|\ .
\end{align}
We also use $\zeta\!\left(-1, \frac{d'}{k}\right) = -\frac{1}{2}\,B_2\!\left(\frac{d'}{k}\right)$. Then we get
\begin{multline}
\tr{Res}_{s=0} = k\zeta'(-1)  + \frac{k}{12}\log\,(zk) +\frac{i\pi k}{2} \sum_{d'=1}^{k-1}B_2\!\left(\tfrac{d'}{k}\right)B_1\!\left(\tfrac {d'h}{k}\right) \\+ \frac{k}{2}\sum_{d'=1}^{k-1}B_2\!\left(\tfrac{d'}{k}\right)\,\log\,\big|2\sin\left(\tfrac{\pi d'h}{k}\right)\!\big|\ .
\end{multline}
Finally, we can show that $\sum_{d'=1}^{k-1}B_2\!\left(\tfrac{d'}{k}\right)B_1\!\left(\tfrac {d'h}{k}\right) = 0$ identically, and hence,
\begin{align}
\tr{Res}_{s=0} & = k\,\zeta'(-1)  + \frac{k}{12}\log\,(zk) + \frac{k}{2}\sum_{d'=1}^{k-1}B_2\!\left(\tfrac{d'}{k}\right)\,\log\,\big|2\sin\left(\tfrac{\pi d'h}{k}\right)\!\big|\ ,\\
&:= k\,\zeta'(-1) + \frac{k}{12}\log\,(zk) + C_{h,k}\ .
\end{align}

\subsection{Residue at $s = -p$ for integer $p > 0$.}
The residue at $s=-p$ ($p=1, 2, 3, \ldots$) is
\begin{align}
  \tr{Res}_{s=-p}&=\frac{(-zk^2)^p}{p!} \sum_{d,d'=1}^k \zeta(-p-1,d'/k) \zeta(-p+1,d/k) e^{2\pi i d d' h/k} \nonumber \\
  &= \frac{(-zk^2)^p}{p!p(p+2)}\sum_{d,d'=1}^{k} B_{p+2}(d'/k) B_{p}(d/k) e^{2\pi i d d' h/k} \\
  &=\frac{(-zk^2)^p}{p!p(p+2)}\bigg(\sum_{d'=1}^{k-1} B_{p+2}(d'/k)
  \widehat{B}_{p}(d'h/k) + k^{-p+1}B_{p+2} B_{p} \bigg) \nonumber
\end{align}
where we use the identity $\sum_{d=0}^{k-1} B_p(d/k) = k^{1-p} B_p$
(it is a Kubert function of type $(1-p)$) and $\widehat{B}_{p}(x)$ is
the discrete Fourier transform of $B_p(x)$\cite{Milnor1983,alm2}. One has
\begin{equation}
\widehat{B}_p(x) =  (-1)^p k^{1-p}  \frac{p}{(2i)^p} \cot^{(p-1)}(\pi x)\ .
\end{equation}
The final result is then
\begin{equation}
  \textrm{Res}_{s=-p}= \frac{(-z)^p k^{1+p}}{p!p(p+2)}\bigg[B_{p+2}
  B_{p}+ \frac{p}{(2i)^p} \sum_{d=1}^{k-1} B_{p+2}(d/k)
  \cot^{(p-1)}(\pi dh/k) \bigg]\ .
\end{equation}
For $p = 1$, the first term in the brackets drops out since $B_3 = 0$. We then get
\begin{equation}
  \textrm{Res}_{s=-1}= \frac{izk^2}{6}\sum_{d=1}^{k-1} B_{3}(d/k)
  \cot\,(\pi dh/k) := v^{(1)}_{h,k}\,z\ .
\end{equation}

\section{Asymptotics of $L_{h,k}(z)$} \label{Lhkasymptotics}

Our focus will be on the family of functions (with $0<\epsilon<1$; $0<h<k$ and $(h,k)=1$)
\begin{equation}
L_{h,k}(z):= \frac1{2\pi i}\sum_{d,d'=1}^{k} e^{2\pi i d d' h/k} \int_{-1-\epsilon-i\infty}^{-1-\epsilon+i\infty} \!\! (zk^2)^{-s}\, \Gamma(s)\,\zeta(s-1,\tfrac{d'}k)\,\zeta(s+1,\tfrac{d}k)  \,\ud s  \ .
\end{equation}
For $\textrm{Re}(s)<-1$, the only singularities in the integrand occur for $s=-2,-3,\ldots$ due to the poles in $\Gamma(s)$. One can arrive at a series expansion for $L_{h,k}(z)$ by moving the contour and including the contribution of the poles at say, $s=-2,-3,\ldots,-M$ to obtain
\begin{align}
L_{h,k}(z)&=\sum_{m=2}^{M} \frac{(-zk^2)^m}{m!} \sum_{d,d'=1}^k \zeta(-m-1,d'/k) \zeta(-m+1,d/k) e^{2\pi i d d' h/k} + R_{h,k}^{(M)}(z)\nonumber \\
&=\sum_{m=2}^{M} \frac{(-zk^2)^m}{m!m(m+2)}\sum_{d,d'=1}^{k} B_{m+2}(d'/k) B_{n}(d/k) e^{2\pi i d d' h/k} + R_{h,k}^{(M)}(z)\ , \\
&=\sum_{m=2}^{M} v^{(m)}_{h,k} z^m +  R_{h,k}^{(M)}(z)\ , \nonumber
\end{align}
where  the remainder (let  $hh'=1\textrm{ mod } k $ and $w = zk^2 / 4\pi^2$)
\begin{align}
&R_{h,k}^{(M)}(z) = \nonumber\\ &=\frac1{2 i} \int_{-M-\epsilon-i\infty}^{-M-\epsilon+i\infty}\frac{(zk^2)^{-s}}{\Gamma(1-s) \sin \pi s}\sum_{d,d'=1}^{k} e^{2\pi i d d' h/k}\,\zeta(s-1,\tfrac{d}k)\,\zeta(1+s,\tfrac{d'}k)   \,\ud s \nonumber \\
&=\frac{1}{2i} \int_{-M-\epsilon-i\infty}^{-M-\epsilon+i\infty}\!\!\frac{w^{-s}\,\Gamma(2-s)}{(4\pi^2)s \sin \pi s} \times \nonumber \\&\qquad\qquad\qquad\quad\times\sum_{\eta,\eta'=\pm1}\!\!\!e^{\frac{i\pi(\eta+\eta') s}2} \sum_{d,d'=1}^{k} \!\!\! e^{2\pi i d d' h/k}\,\textrm{Li}_{2-s}(e^{\frac{2\pi i \eta d}k})\, \textrm{Li}_{-s}\,(e^{\frac{2\pi i \eta' d'}k})\, \ud s \nonumber  \\
&=\frac{1}{2 i} \int_{-M-\epsilon-i\infty}^{-M-\epsilon+i\infty}\frac{w^{-s}\,\Gamma(2-s)}{(4\pi^2)s \sin \pi s}\sum_{\eta,\eta'=\pm1}\!\!e^{i\pi(\eta+\eta') s/2} \sum_{m=1}^\infty \frac{k\,\sigma_2(m)\,e^{-2\pi i \eta \eta' m h'/k}}{m^{2-s}}\, \ud s \nonumber\\
&=\frac{2k}{2 i} \int_{-M-\epsilon-i\infty}^{-M-\epsilon+i\infty} \frac{w^{-s}\,\Gamma(2-s)}{(2\pi )^2s \sin \pi s} \sum_{m=1}^\infty\frac{ \sigma_2(m)}{m^{2-s}}
\Big[ e^{2\pi im h'/k} + \cos (\pi s )e^{-2\pi imh'/k}\Big]\, \ud s 
%e^{i\pi(\eta+\eta') s/2}e^{-2\pi i \eta \eta' m h'/k}
\end{align}
We would like to take $\epsilon\rightarrow1$ but there is a pole due to the $1/\sin \pi s$ term. So we deform the integral such that the contour lies on the $s=-M-1$ line except for a semi-circular detour to avoid the pole. The semi-circular contour gives $\pi i $ times the residue of the pole plus the Cauchy principal value of integral.
We get (with $s=-M-1+it$)
\begin{equation}
R_{h,k}^{(M)}(z) = \tfrac12 v_{h,k}^{(M+1)}z^{M+1} +R_1+R_2\ ,
\end{equation}
where 
\begin{align}
R_1&= (-1)^M \sum_{m=1}^\infty \frac{k\,\sigma_2(m)}{(2\pi m)^2}\, e^{\frac{2\pi im h'}k}\ \mathcal{P} \!\!  \int_{-\infty}^{\infty} \left(\tfrac{w}{m}\right)^{-s}\tfrac{\Gamma(2-s)}{i s \sinh \pi t}\,\ud t \\
R_2&= - \sum_{m=1}^\infty \frac{k\,\sigma_2(m)}{(2\pi m)^2}\,e^{-\frac{2\pi imh'}k} \ \mathcal{P}\!\!   \int_{-\infty}^{\infty} \left(\tfrac{w}{m}\right)^{-s}\tfrac{\Gamma(2-s) \cosh (\pi t )}{i s \sinh \pi t}\, \ud t
\end{align}
We rewrite as  (with $\sigma=M+1$ and hence $s=-\sigma+it$)
\begin{align}
  R_1&= (-1)^{M+1}  \sum_{m=1}^\infty \tfrac{k\,\sigma_2(m)}{i(2\pi m)^2}\left(\tfrac{w}{ m}\right)^{\sigma} e^{\frac{2\pi im h'}k}\times\nonumber\\ & \ \qquad\qquad\qquad\times\lim_{\epsilon\rightarrow0} \int_{\epsilon}^{\infty} \tfrac{1}{\sinh \pi t} \Big[\left(\tfrac{w}{ m}\right)^{-it}\tfrac{\Gamma(2+\sigma-it)}{\sigma-it}-\left(\tfrac{w}{ m}\right)^{it }\tfrac{\Gamma(2+\sigma-it )}{\sigma+it}\Big]\,\ud t \nonumber \\
  &=(-1)^{M+1} \sum_{m=1}^\infty \tfrac{2k\sigma_2(m)}{(2\pi m)^2}\left(\tfrac{w}{ m}\right)^{\sigma} e^{\frac{2\pi im h'}k}  \! \int_{0}^{\infty} \tfrac{1}{\sinh \pi t} \textrm{Im}\Big[\left(\tfrac{w}{ m}\right)^{-it}\tfrac{\Gamma(2+\sigma-it)}{\sigma-it}\Big]\,\ud t \\
  &= (-1)^{M+1}\sum_{m=1}^\infty \tfrac{2k\sigma_2(m)}{(2\pi
    m)^2}\left(\tfrac{w}{ m}\right)^{\sigma} e^{\frac{2\pi im h'}k}
  \tfrac{\Gamma(2+\sigma)}{\sigma}\times \nonumber\\
  &\ \qquad\qquad\qquad\qquad\times\int_{0}^{\infty}
  \tfrac{\sin\varphi(t)}{\sinh \pi t} (1+\tau^2)^{\tfrac12 (\sigma
    +1/2)} e^{-\psi t} \Big[1+
  \mathcal{O}(\tfrac1{\sigma+it})\Big]\,\ud t\ ,
\end{align}
where we have determined the phase 
$$\varphi(t):= t \log \tfrac{m \sqrt{(2+\sigma)^2 + t^2}}{e w} + \arctan \tfrac{t}{\sigma}$$ 
using Stirling's formula for the gamma function
(with $\psi = \arctan \frac t{\sigma+2}$ and $\tau=\tfrac{t}{\sigma+2}$).
\begin{equation*}
\tfrac{\Gamma(\sigma+2+it)}{\Gamma(\sigma+2)} = (1+\tau^2)^{\tfrac12 (\sigma +3/2)} e^{-\psi t}  e^{i t \log \sqrt{(\sigma+2)^2+t^2}/e} \Big[1+ \mathcal{O}(\tfrac1{\sigma+2+it})\Big]\ .
\end{equation*}
Define the function $f(\tau)$ as follows:
\begin{equation}
f(\tau):=\tfrac12\log(1+\tau^2) -\arctan\tau\ . \label{fdef}
\end{equation}
For $\tau\ll 1$, $f(\tau)\sim -\tau+\mathcal{O}(\tau^2)$ and thus $e^{(\sigma+2)f(\tau)}\sim e^{-t}$ for $\tau\ll1$ and $\sigma\gg 1$.
We thus need to evaluate the integral
\begin{align}
 I_1&=\int_{0}^{\infty} \tfrac{\sin\varphi(t)}{\sinh \pi t}  \tfrac{e^{(\sigma+2)f(\tau)}}{ (1+\tau^2)^{3/4}}\,\ud t\\
 &=\int_{0}^{K} \tfrac{\sin\varphi(t)}{\sinh \pi t}   \tfrac{e^{(\sigma+2)f(\tau)}}{ (1+\tau^2)^{3/4}}\,\ud t +\int_{K}^{\infty} \tfrac{\sin\varphi(t)}{\sinh \pi t}   \tfrac{e^{(\sigma+2)f(\tau)}}{ (1+\tau^2)^{3/4}}\, \ud t\\
 &= I_{1a} + I_{1b}\ .
\end{align}
In $I_{1a}$, $K$ is chosen such that $K\pi \sim 1 \ll \sigma$. Thus, 
\begin{equation}
I_{1a} = \int_{0}^{K} \tfrac{\sin\varphi(t)}{\sinh \pi t} dt    
\  (1+ \mathcal{O}(1/\sigma))\ .
\end{equation}
When necessary, we will choose $K\sim 1/\pi$ for concreteness.
 We can approximate $\arctan(t/\sigma)$ by $t/\sigma+2$ and we can carry out the   integral (with $\alpha= \log \tfrac{m (2+\sigma)}{e w} +\tfrac1{\sigma}\gg1$ )
\begin{align}
I_{1a}\sim \int_0^K \tfrac{\sin \varphi(t)}{\sinh \pi t}\,\ud t &\sim \int_0^K \tfrac{\sin \alpha t}{\sinh \pi t}\, \ud t
 =  \int_0^\infty  \tfrac{\sin \alpha t}{\sinh \pi t}\,\ud t  +\mathcal{O}(\tfrac{1}{ \alpha}) \nonumber  \\
 &= \tfrac12 \tanh \tfrac\alpha2 +\mathcal{O}(\tfrac{1}{ \alpha}) \ , \nonumber \\
  &= \tfrac12 \tfrac{m(\sigma+2)-e w}{m(\sigma+2)+ew} +\mathcal{O}(\tfrac{1}{ \alpha})\rightarrow \tfrac12 \textrm{ as }w\rightarrow 0\ ,
\end{align}
Let us consider $I_{1b}$. We will show that it can be neglected.
\begin{align}
I_{1b}&=\int_{K}^{\infty} \tfrac{\sin\varphi(t)}{\sinh \pi t} \tfrac{e^{(\sigma+2)f(\tau)}}{ (1+\tau^2)^{3/4}}\,\ud t \ .
\end{align}
Let $g(t)= \tfrac{e^{(\sigma+2)f(\tau)}}{ \sinh \pi t
  (1+\tau^2)^{3/4}}$. It is monotonically decreasing, positive
definite for $t\in [K,\infty]$; vanishes as $t\rightarrow \infty$ and
varies slowly since $|\varphi'(t)|\gg \pi$. Thus the conditions of
Proposition \ref{integralbound} are satisfied (with $\hat{m}=m/ew$)
and we obtain the following bound on $I_{1b}$
\begin{equation}
|I_{1b}| \leq \tfrac{2}{\varphi'(K)}  \tfrac{e^{(\sigma+2)f(K)}}{ \sinh \pi K (1+K^2)^{3/4}} \leq  \tfrac{2}{\alpha  \sinh \pi K} \rightarrow 0 \textrm{ as } w\rightarrow 0\ .
\end{equation}
%The maximum value of the integrand occurs at $t=K$ and it monotonically decreases as $t$ increases. Thus, we see that
%\begin{equation}
%|I_{1b}|  \leq \int_K^\infty \tfrac1{\sinh \pi t} dt = - \tfrac1\pi \log \tanh (\pi K/2)\sim \frac2\pi e^{-\pi K}\ .
%\end{equation}
We thus get that $|I_1| \leq |I_{1a}| + |I_{1b}| \leq \tfrac12  + \mathcal{O}(1/\sigma)+\mathcal{O}(\tfrac{1}{ \alpha})$.
Putting this together, we obtain
\begin{align}
|R_{1}| & \leq \sum_{m=1}^\infty \tfrac{k\sigma_2(m)}{(2\pi m)^2}\left(\tfrac{w}{ m}\right)^{\sigma}  \tfrac{\Gamma(2+\sigma)}{\sigma}\ (1+ \mathcal{O}(1/\sigma)+ \mathcal{O}(1/\alpha))\\
&= k \zeta(M+1)\zeta(M+3) w^{M+1}  \tfrac{\Gamma(M+3)}{(2\pi)^2(M+1)} \  (1+ \mathcal{O}(\tfrac1M)+ \mathcal{O}(\tfrac1{\log(M/(|z|k^2))})) \label{R1bound}
\end{align}
$R_2$ is similar to $R_1$ with $\sinh\pi t$ being replaced by $\tanh \pi t$. The convergence of the integral for $R_2$ is more delicate as $\tanh\pi t\rightarrow 1$ for large $t$ (instead of decaying exponentially) and needs the inclusion of the $e^{-\pi t/2}$ appearing for large and positive  $t$ in Stirling's formula for the gamma function. Explicitly,
\begin{align}
R_2&=  \sum_{m=1}^\infty \tfrac{2k\sigma_2(m)}{(2\pi m)^2}\left(\tfrac{w}{ m}\right)^{-\sigma} e^{\frac{-2\pi im h'}k}  \! \int_{0}^{\infty} \tfrac{\sin\varphi(t)}{\tanh \pi t}  \tfrac{e^{(\sigma+2)f(\tau)}}{ (1+\tau^2)^{3/4}}\,\ud t \ , 
%&= - \sum_{m=1}^\infty \tfrac{2k\sigma_2(m)}{(2\pi m)^2}\left(\tfrac{w}{ m}\right)^{-\sigma} e^{\frac{2\pi im h'}k}  \! \left(\int_{0}^{1/\pi} +\int_{1/\pi}^\infty \right)\tfrac{\sin\varphi(t)}{\tanh \pi t} \Big|\tfrac{\Gamma(2-\sigma-it)}{\sigma+it}\Big| 
\end{align}
We thus need to evaluate the integral
\begin{align}
 I_2&=\int_{0}^{\infty} \tfrac{\sin\varphi(t)}{\tanh \pi t}  \tfrac{e^{(\sigma+2)f(\tau)}}{ (1+\tau^2)^{3/4}}\,\ud t\\
 &=\int_{0}^{K} \tfrac{\sin\varphi(t)}{\tanh \pi t}  \tfrac{e^{(\sigma+2)f(\tau)}}{ (1+\tau^2)^{3/4}}\,\ud t +\int_{K}^{\infty} \tfrac{\sin\varphi(t)}{\tanh \pi t}  \tfrac{e^{(\sigma+2)f(\tau)}}{ (1+\tau^2)^{3/4}}\,\ud t 
 \\
 &= I_{2a} + I_{2b}\ .
\end{align}
Again we can approximate $\arctan(t/\sigma)$ by $t/\sigma$ and we can carry out the  integral (with $\alpha= \log \tfrac{m (2+\sigma)}{e w} +\tfrac1{\sigma}\gg1$ )
\begin{align}
I_{2a}&= \int_0^K \tfrac{\sin \varphi(t)}{\tanh \pi t}\,\ud t + \mathcal{O}(1/\sigma)\sim \int_0^K \tfrac{\sin \alpha t}{\tanh \pi t} \,\ud t  \nonumber \\
&=  \int_0^K \tfrac{\sin \alpha t}{\pi t} (1+\tfrac{\pi^2t^2}3 +\mathcal{O}(t^4))\,\ud t + \mathcal{O}(1/\sigma) \nonumber  \\
& = \int_0^K \tfrac{\sin \alpha t}{\pi t}\,\ud t + \mathcal{O}(1/\alpha) + \mathcal{O}(1/\sigma)\nonumber \\
& = \tfrac12 +  \int_K^\infty \tfrac{\sin \alpha t}{\pi t}\,\ud t  + \mathcal{O}(1/\alpha)+ \mathcal{O}(1/\sigma) =  \tfrac12  + \mathcal{O}(1/\alpha)+ \mathcal{O}(1/\sigma) \ ,
% &\leq \cosh (\pi K)  \int_0^\infty  \tfrac{\sin \alpha t}{\sinh \pi t}dt  +\mathcal{O}(\tfrac{1}{K \alpha}) \\
% &= \tfrac{\cosh (\pi K)}2 \tanh \tfrac\alpha2 +\mathcal{O}(\tfrac1{K \alpha}) \\
%  &\sim  \tfrac{\cosh (\pi K)}2 \tfrac{m(\sigma+2)-e w}{m(\sigma+2)+ew} +\mathcal{O}(\tfrac{1}{K \alpha})\rightarrow  \tfrac{\cosh (\pi K)}2 \textrm{ as }w\rightarrow 0\ .
\end{align}
since $\int_K^\infty \tfrac{\sin \alpha t}{\pi t}\,\ud t \sim K\cos(\alpha K)/\alpha= \mathcal{O}(1/\alpha)$. In the second line of the above equation, we have used the relation $(\tanh x)^{-1} = x^{-1} (1 + \tfrac{x^2}3 + \mathcal{O}(x^4))$. 
Next, let us consider $I_{2b}$. Again, we will show that it can be neglected.
\begin{align}
I_{2b}&=\int_{K}^{\infty} \tfrac{\sin\varphi(t)}{\tanh \pi t}   \tfrac{e^{(\sigma+2)f(\tau)}}{ (1+\tau^2)^{3/4}}\,\ud t \ .
\end{align}
Let $g(t)= \tfrac{e^{(\sigma+2)f(\tau)}}{ \tanh \pi t (1+\tau^2)^{3/4}}$. It is monotonically decreasing, positive definite for $t\in [K,\infty]$; vanishes as $t\rightarrow \infty$ and varies slowly since $|\varphi'(t)|\gg \pi$. Thus the conditions of Proposition \ref{integralbound} are satisfied  (with $\hat{m}=m/ew$) and we obtain the following bound on $I_{2b}$:
\begin{equation}
|I_{2b}| \leq \tfrac{2}{\varphi'(K)}  \tfrac{e^{(\sigma+2)f(K)}}{ \tanh \pi K (1+K^2)^{3/4}} \leq  \tfrac{2}{\alpha  \tanh \pi K} \rightarrow 0 \textrm{ as } w\rightarrow 0\ .
\end{equation}
%\begin{align}
%I_{2b}&=\int_{K}^{\infty} \tfrac{\sin\varphi(t)}{\tanh \pi t}  (1+\tau^2)^{\tfrac12 (\sigma -1/2)} e^{-\psi t} dt  \\
%&\leq  \tfrac{1}{\tanh \pi K} \int_{K}^{\infty} \sin\varphi(t) \frac{e^{\sigma(\frac12 \log (1+\tau^2)-\psi\tau)}}{(1+\tau^2)^{1/4}}  dt \\
%&\leq  \tfrac{1}{\tanh \pi K} \int_{K}^{\infty} \sin\varphi(t) e^{\sigma(\frac12 \log (1+\tau^2)-\psi\tau)}  dt 
%\end{align}
%Defining $f(\tau)= \frac12 \log (1+\tau^2)-\psi\tau)$ and integrating by parts and approximating $\varphi(t)$ by $\alpha t$, we obtain
%\begin{align}
%\tanh (\pi K)\ I_{2b} & \leq -\frac{\cos \varphi(t)e^{\sigma f(\tau)}}{\alpha}\bigg|_{K}^\infty -\frac1\alpha \int_K^\infty \cos \varphi(t) e^{\sigma f(\tau)} \psi dt \\
%& \leq \frac{\cos \alpha K}{\alpha} - \frac1\alpha \int_K^\infty \cos \varphi(t) e^{\sigma f(\tau)} \psi dt
%\end{align}
%Taking the absolute value and using Proposition \ref{intbound}, we get
%\begin{align}
%|\tanh (\pi K) I_{2b}| \leq \frac1\alpha \left (1 + \sqrt{2\pi \sigma}\right)\ ,
%\end{align}
We obtain that $|I_2|\leq |I_{2a}|+ |I_{2b}| \leq \tfrac12 +\mathcal{O}(1/\sigma)+ \mathcal{O}(1/\alpha)$.
Putting this together, we obtain
\begin{align}
|R_{2}| &\leq   
 k \zeta(M+1)\zeta(M+3) w^{M+1}  \tfrac{\Gamma(M+3)}{(2\pi)^2(M+1)}  \Big(1+ \mathcal{O}(\tfrac1M)+ \mathcal{O}(\tfrac1{\log(M/(|z|k^2))}) \Big)\ .\label{R2bound}
\end{align}
%In the limit $w\rightarrow 0$, $\alpha\rightarrow \infty$ and thus $ \tfrac{1}{\alpha})\ll 1$. The
%bound that we obtain  on discarding terms multiplied by $1/\alpha$, we get 
%\begin{align}
%|R_{2}| &\leq   
% k \zeta(N+1)\zeta(N+3) w^{N+1}  \tfrac{\Gamma(N+3)}{(2\pi)^2(N+1)} \cosh \pi K\  + \mathcal{O}(1/\alpha)\ . 
%\end{align}
Combining bounds \eqref{R1bound} and \eqref{R2bound} with the bound \ref{Shkpbound} on $|v_{h,k}^{(M+1)}z^{M+1}|$ , we see that 
as $z\rightarrow 0$ that
\begin{equation}
\boxed{
\left|R^{(M)}_{h,k}(z)-\tfrac12v_{h,k}^{(M+1)}z^{M+1}\right|  \leq\left|v_{h,k}^{(M+1)}z^{M+1}\right|\left(\frac12 + \mathcal{O}(\tfrac1M)+ \mathcal{O}(\tfrac1{\log(M/(|z|k^2))})\right)\ . \label{Rhkboundfinal}
}
\end{equation}
%Choosing $K=1/\pi$ and using $\cosh 1\sim 1.05$, we see that 
%\begin{equation}
%\boxed{
%\left|\tfrac{R^{(N)}_{h,k}}{v_{h,k}^{(N+1)}z^{N+1}}\right|  \leq 1.01  + \mathcal{O}(\sqrt{N}/\alpha) \ . 
%}
%\end{equation}

\begin{prop}\label{integralbound} Let $\varphi(t)= t \log \hat{m} \sqrt{(\sigma+2)^2 + t^2} + \arctan \tfrac{t}{\sigma}$ with $\hat{m}\gg1$ and $\sigma\gg 1$. Let $g(t)$ be 
(i) a slowly varying positive definite real function of $t$ that vanishes as $t\rightarrow \infty$ and (ii) a monotonically decreasing function of $t$.  Then for some $K>0$,
\begin{align}
\left|\int_K^\infty  \sin \varphi(t)  g(t)\,\ud t \right| &\leq  2 \left| \tfrac{g(K)}{\varphi'(K)}\right|\ .
\end{align}
\end{prop}
\begin{proof}
Integrating by parts, one sees that 
\begin{equation}
\int_K^\infty  e^{i \varphi(t)}  g(t)\,\ud t  = e^{i \varphi(t)} \tfrac{g(t)}{i\varphi'(t)} \Big|_K^\infty - \int_K^\infty e^{i \varphi(t)}   \tfrac{\ud}{\ud t} \left(\tfrac{g(t)}{i\varphi'(t)}\right)\,\ud t \ .
\end{equation}
Taking the imaginary part of the above equation gives
\begin{align}
\int_K^\infty  \sin \varphi(t)  g(t)\,\ud t  &= -\cos \varphi(t) \tfrac{g(t)}{\varphi'(t)} \Big|_K^\infty +\int_K^\infty \cos \varphi(t)   \tfrac{\ud}{\ud t} \left(\tfrac{g(t)}{\varphi'(t)}\right)\,\ud t \\
&=\cos \varphi(K) \tfrac{g(K)}{\varphi'(K)}+\int_K^\infty \cos \varphi(t)   \tfrac{\ud}{\ud t} \left(\tfrac{g(t)}{\varphi'(t)}\right) \,\ud t \ .
\end{align}
Taking the absolute value, we get 
\begin{align*}
\left|\int_K^\infty  \sin \varphi(t)  g(t)\,\ud t \right| &=\left|\cos \varphi(K) \tfrac{g(K)}{\varphi'(K)}\right|+\int_K^\infty \left|\cos \varphi(t)  \tfrac{\ud}{\ud t} \left(\tfrac{g(t)}{\varphi'(t)}\right)\right|\,\ud t  \\
&\leq \left| \tfrac{g(K)}{\varphi'(K)}\right|+\int_K^\infty  \left|\tfrac{\ud}{\ud t} \left(\tfrac{g(t)}{\varphi'(t)}\right)\right|\ud t\ . 
\end{align*}
Given that $g'(t)<0$ and $\ud(1/\varphi'(t))/\ud t<0$ for $t\in [K,\infty)$, we can simplify things further as the integrand in the second term above is always negative definite.
\begin{align*}
\left|\int_K^\infty  \sin \varphi(t)  g(t)\ud t \right| 
%&\leq \left| \tfrac{g(K)}{\varphi'(K)}\right|  +\int_K^\infty  \left|\tfrac{\ud}{\ud t} \left(\tfrac{g(t)}{\varphi'(t)}\right)\right| \,\ud t \\
&\leq \left| \tfrac{g(K)}{\varphi'(K)}\right|+ \left|\int_K^\infty \tfrac{\ud}{\ud t} \left(\tfrac{g(t)}{\varphi'(t)}\right)\,\ud t \right| 
\leq \left|\tfrac{2g(K)}{\varphi'(K)}\right|  \ .
\qedhere
\end{align*}

\end{proof}

\section{Bounds on generalized Dedekind sums} \label{boundedsums}
\subsection{Bound on $C_{h,k}$}
We have
\begin{equation}  \label{Chkdef}
C_{h,k} = \frac{k}{2}\sum_{j = 1}^{k-1} B_2(j/k)\,\log \big|2\sin(\pi jh/k)\big|\ .
\end{equation}
We then use $-1/12 \leq B_2(x) \leq 1/6$ and the following identity
mentioned in \cite{alm2} (which is in turn attributed to Rademacher
and Grosswald):
\begin{equation}
\sum_{j=1}^{k-1} \log \big|2\sin(j\pi /k)\big| = \log k\ .
\end{equation}
Thus for $(h,k)=1$, we see that $\sum_{j=1}^{k-1} \log \big|2\sin(jh\pi /k)\big|=\sum_{j=1}^{k-1} \log \big|2\sin(j\pi /k)\big|=\log k$. Using this, we get
%\begin{equation}
%\boxed{|C_{h,k}|\ \leq\ \frac{k}{12}\sum_{j=1}^{k-1}\log \big|2\sin (\pi jh/k)\big| = \frac{k}{12}\log k\ .}
%\end{equation}
\begin{align}
 C_{h,k}&= \frac{k\log 2}2 \sum_{j=1}^{k-1} B_2(j/k)+ \frac{k}2 \sum_{j=1}^{k-1} B_2(j/k)\,\log \big|\sin(\pi jh/k)\big|  \\
 &=- \frac{k-1}{12} \log 2 + \frac{k}2 \sum_{j=1}^{k-1} B_2(j/k)\,\log \big|\sin(\pi jh/k)\big|  \\
 &< - \frac{k-1}{12} \log 2 - \frac{k}{24}  \sum_{j=1}^{k-1} \log \big|\sin(\pi jh/k)\big|
\end{align}
A similar argument can
be used to show that
\begin{equation}
C_{h,k}> - \frac{k-1}{12} \log 2 + \frac{k}{12}  \sum_{j=1}^{k-1} \log \big|\sin(\pi jh/k)\big|
\end{equation}
Now we use 
$$\sum_{j=1}^{k-1} \log \big|\sin(\pi jh/k)\big|= -(k-1)\log 2+\sum_{j=1}^{k-1} \log \big|2\sin(\pi jh/k)\big| =  \log k - (k-1)\log 2$$ 
to get
\begin{equation}
\boxed{\frac{1-k^2}{12} \log 2 +\frac{k}{12}\log k < C_{h,k} < \frac{(k-1)(k-2)}{24} \log 2 - \frac{k}{24}\log k\ .} \label{Chkbound1}
\end{equation}

We can significantly improve on this bound and this is given by the next proposition.
\begin{prop}\label{conjecture} The following  bound for $C_{h,k}$ holds 
\begin{equation}
\boxed{
-\frac{\zeta(3)k^2}{4\pi^2 } < C_{1,k}\leq C_{h,k} < \frac{k\log k}{12} -\alpha\ \frac{k\log2}{12}}\ .
\ . \label{bound3}
\end{equation}
with $\alpha=3$ and $k>34$.
\end{prop}
\begin{proof}
We have shown that $\bb_{h,k}> \frac{\log 2}2$ for all $k>34$ in Eq. \eqref{bhkweakbound}.
Using the relationship between $C_{h',k}$ and $\bb_{h,k}$ (with $hh'=1\mod k$)
 \begin{equation}
 C_{h',k} = \frac{k\log k}{12} -  \frac{k}2\ \bb_{h,k}\ ,
 \end{equation}
 we obtain the following upper bound stated in the proposition. The lower bound follows similarly from the estimate for $\bb_{1,k}$ given in Eq. \eqref{b1kestimate}.
 \end{proof}
\noindent \textbf{Remark:} The above proposition was a conjecture in an earlier version of the manuscript\cite{v1}.

\subsection{Bound on $v_{h,k}^{(1)}$.}
We have 
\begin{align}
v_{h,k}^{(1)}& := \frac{k^2}{6} \sum_{d=1}^{k} B_3(d/k) \cot(dh \pi /k)\ , \\
&= \frac{k^2}{6} \sum_{d=1}^{k} \frac{2}{i} \sum_{d'=1}^{k-1} B_3(d/k)  B_1(d'/k) e^{-2\pi i d d' h/k} \\
&= -\frac{2 k^2}{(2\pi )^3} \sum_{d=1}^{k}  \sum_{d'=1}^{k-1} \sum_\ell \frac1{\ell^3} B_1(d'/k)  e^{2\pi i [-d d' h+\ell d]/k} \\
&=-\frac{2 k^3}{(2\pi )^3}  \sum_{d'=1}^{k-1} \sum_\ell \frac1{\ell^3} B_1(d'/k)   \delta^{[k]}(\ell -d'h) \\
&=-\frac{2 k^3}{(2\pi )^3}  \sum_\ell \frac1{\ell^3} B_1(\ell h'/k)\ ,
\end{align}
where $ \delta^{[k]}(x)$ is the periodic delta function with period
$k$. We have also used the Fourier series for $B_3(x)$ with $0 \leq x \leq 1$:
\begin{equation}
B_3(x) = \frac{-(3!)}{(2\pi i)^3} {\sum_{\ell\,\in\,\mathbb{Z}}}'\ \frac{e^{2\pi i \ell  x}}{\ell^3}\ . 
\end{equation}
where the $\Sigma'$ indicates that $\ell
= 0$ has to be omitted from the summation. Using $|B_1(x)|\leq \tfrac12$, we see that
\begin{equation}\label{vhk1bound}
  \boxed{
\left|v_{h,k}^{(1)}\right|\leq\ \frac{2 k^3}{(2\pi )^3}\zeta(3)\ .
  }
\end{equation}

\subsection{Bound on $v_{h,k}^{(p)}$.}
For $p>1$, we have
\begin{align}
v_{h,k}^{(p)} &= \frac{k^{2p}}{p!p(p+2)} \sum_{d=1}^{k} \sum_{d'=1}^k B_{p+2}(d/k) B_{p}(d'/k) e^{-2\pi i dd'h/k} \\ 
&=  \frac{k^{2p}(p+1)!}{p(2\pi i)^{2p+2}}{\sum_{\ell\,\in\,\mathbb{Z}}}' {\sum_{\ell'\,\in\,\mathbb{Z}}}'  \sum_{d=1}^{k} \sum_{d'=1}^k\, \frac1{\ell^{p+2}\,(\ell')^p}\ e^{2\pi i[- dd'h + d\ell + d'\ell']/k} \\
&= \frac{k^{2p+1}(p+1)!}{p(2\pi i)^{2p+2}} {\sum_{\ell}}' {\sum_{\ell'}}'  \sum_{d'=1}^k\, \frac1{\ell^{p+2}\,(\ell')^p}\ e^{2\pi i d'\ell'/k} \delta^{[k]}(\ell-d' h) \\
&= \frac{k^{2p+1}(p+1)!}{p(2\pi i)^{2p+2}} {\sum_{\ell}}' {\sum_{\ell'}}'\, \frac{e^{2\pi i \ell \ell' h'/k}}{\ell^{p+2}\,(\ell')^p}\ ,
\end{align}
where $hh'=1 \textrm{ mod } k$. This gives, for $p \geq 2$,
\begin{align}
\boxed{
\Big|v_{h,k}^{(p)}\Big|\ \leq\ \frac{4\,k^{2p+1}(p+1)!}{p(2\pi)^{2p+2}}\,\zeta(p)\, \zeta(p+2) \ . \label{Shkpbound}
}
\end{align}
\subsection{Bound on $L_{h,k}(z)$.}
We use the  bound given in Eq. \eqref{Rhkboundfinal} with $M = 1$  as well as the bound on $v_{h,k}^{(2)}$ to get a bound on $L_{h,k}(z)$. One has 
\begin{equation}
|R_{h,k}^{(1)}(z)| \leq \frac{12 \zeta(2)\zeta(4)}{(2\pi)^6} \ k^3 |z|^2\ .
\end{equation}

This gives
\begin{equation}
  |L_{h,k}(z)| \leq \frac{12 \zeta(2)\zeta(4)}{(2\pi)^4} \ k^3 |z|^2
 := \frac{a_1}{(2\pi)^2}\ k^3|z|^{2}\ .  \label{Lhkpbound}
\end{equation}

\section{A proof of Proposition \ref{conjecture}} \label{bhkappendix}

Let $f_h(x)$ be the following periodic function (with period $1$)
\begin{align}
f_h(x) & : = \left(\frac16-\widetilde{B}_2(hx) \right)\ \log |2 \sin (\pi x)| \ , \nonumber \\
&= g(h x) \log |2 \sin (\pi x)| 
\end{align}
where $\widetilde{B}_2(x)$ is the periodic Bernoulli function and $g(x):= \{x\}\big(1-\{ x\}\big)$.
For $x\in [0,1)$, the function has cusps at $x=\tfrac{r}h$, $r=0,1,\ldots,(h-1)$. We wish to compute and obtain bounds for  the following  generalised Dedekind sum for $(h,k)=1$ and $h\leq k/2$.
\begin{equation}
\bb_{h,k} =  \sum_{j=1}^{k-1} f_h(j/k) = \sum_{j=1}^{k-1} g(hj/k) \log |2 \sin(\pi j/k)| \ .
\end{equation}
It is easy to see that $\bb_{h,k}$ is related to the $C_{h,k}$ appearing earlier. One has $C_{h',k} = \frac{k\log k}{12} -  \frac{k}2\ \bb_{h,k}$ with $hh'=1\mod k$. The aim of this
appendix is to show that $\bb_{h,k}\geq 2 \ell_\text{min}$, a constant. However, along the way, we discover several interesting properties such as a reciprocity relation that
helps us  prove the lower bound.

\subsection{The Euler-Maclaurin Formula for $\bb_{h,k}$}

The Euler-Maclaurin Formula (EMF) provides a method to estimate the sum. However, we need to handle the cusps that occur in $f_h(x)$ before applying EMF. We begin with evaluating the  the integral using the trapezoidal scheme.
\begin{equation}
I_{h,k} = \int_{1/k}^{1-(1/k)}  f_h(x)\ dx \quad  .
\end{equation}
The interval is broken up into $h$ parts, $\mathcal{I}_r$ for $r=0,1,\ldots, h-1$. Further each of  the intervals are broken up into the following segments: ($j_r := \left\lfloor \tfrac{rk}{h}\right\rfloor$)
\begin{equation}
\begin{split}
\mathcal{I}_0 &= (\tfrac1k,\tfrac2k,\ldots,\tfrac{j_1}k,\tfrac{1}h)\ , \\
\mathcal{I}_r &= (\tfrac{r}{h}, \tfrac{j_r+1}k,\tfrac{j_r+2}k,\ldots,\tfrac{j_{r+1}}k,\tfrac{r+1}h)\quad \textrm{for } r=1,\ldots,(h-2)\ , \\
\mathcal{I}_{h-1} &=(\tfrac{h-1}{h}, \tfrac{j_{h-1}+1}k,\tfrac{j_{h-1}+2}k,\ldots,\tfrac{k-1}k)\ .
\end{split}
\end{equation}
We thus have
\begin{equation}
I_{h,k} =  \sum_{r=0}^{h-1} \int_{\mathcal{I}_r} f_h(x)\ dx \ .
\end{equation}
The Euler-Maclaurin formula can be applied to the above integrals as the cusps are located only at the end-points and the function $f_h(x)$ is analytic in the interior. We use the version (given by Proposition \ref{EMprop}) that is applicable to the situation where the spacing at the end-points is not the same as the interior. The spacings in the interior are $\tfrac1k$ while it is $\leq \tfrac1k$ at the two end-points. 
\begin{prop}\label{EMprop}
For $\Delta_1,\Delta_2\in [0,1)$, let $a=s+\tfrac{1-\Delta_1}k$, $b=s+\tfrac{m+\Delta_2}k$ for some positive integer $m$ and  consider the interval, $\mathcal{I}=[a,b]$.  
The set of points  $(s+\tfrac{1-\Delta_1}k, s+\tfrac1k,s+\tfrac2k,\ldots,s+\tfrac{m}k,s+\tfrac{m+\Delta_2}k)$ are used to split the interval into smaller parts.
The Euler-Maclaurin formula applied to the function $f(x)$ that is smooth everywhere in the interval $\mathcal{I}$ is
\begin{equation}\label{EMF}
\boxed{
\frac1k \sum_{j=1}^{m} f(s +\tfrac{j}k) -\int_{a}^{b} f(x) dx  
%+\tfrac1{2k} \left(\Delta_1 f_h(a)+\Delta_2 f_h(b)\right) \\ 
=  \sum_{j=1}^{2p} \frac{a_{j}}{k^{j}} + \mathcal{R}_{2p}\ ,
}
\end{equation}
where %$a=s+\frac{1-\Delta_1}k$, $b=s+\frac{m+\Delta_2}k$,
\begin{equation}
a_{j} = \ \Big[(-1)^{j} \frac{B_{j}(\Delta_2)}{j!} \ f^{(j-1)}(b) -\frac{B_{j}(\Delta_1)}{j!}\  f^{(j-1)}(a)\Big]\ ,
\end{equation}
and $\mathcal{R}_{2p}=\tfrac{|B_{2p}|}{k^{2p+2}(2p)!}\ O( ||f||_{C_\text{int}^{2p+1}})+\mathcal{E}_{end} $ where $||f||_{C_\text{int}^m} := \sum_{j=1}^{m}|f^{(m)}(s+\tfrac{j}k)|$ and
$\mathcal{E}_\text{end}=\frac{a_{p+1}}{k^{p+1}}$.
\end{prop}
\begin{proof}
This formula is a special case of a local Euler-Maclaurin formula for polytopes  due to Berline and Vergne (see Sec. 5 of \cite{Berline2005})\footnote{We independently derived this formula and Matthias Beck  kindly directed us towards  the possible relevance of this paper to our analysis.}. Since they deal with polynomials, they obtain a finite series. Our functions are non-polynomial and there are errors associated wtih them. The two contributions to the errors are as follows. The contribution $\tfrac{|B_{2p}|}{k^{2p+2}(2p)!}\ O( ||f||_{C_\text{int}^{2p+1}})$ comes from the interior points in the interval as is normally estimated in the standard EMF. For reasons that will be explained in Sec. \ref{bounding}, we choose to estimate the end-point error $\mathcal{E}_\text{end}$ as given by the next term in the truncated sum rather than an integral.
\end{proof}

\subsection{Estimating $\bb_{1,k}$}

There are no cusps in the interior of $[0,1]$ for $h=1$.  Applying Eq. \eqref{EMF} with $\Delta_1=\Delta_2= 0$ and $p=1$, we get 
%\begin{equation}
%\frac1k \sum_{j=1}^{m} f_1(\tfrac{j}k) - \int_{\varepsilon/k}^{1-(\varepsilon/k)} f_1(x) dx  = \frac{\varepsilon-1}{2k} \Big[f_1(\tfrac\epsilon{k}) + f_1(1-\tfrac\varepsilon{k})\Big]
%+\frac{(\varepsilon-1)^2}{12k^2} \Big[f_1'(1-\tfrac\varepsilon{k}) -f_1'(\tfrac\epsilon{k}) \Big]
% \end{equation}
% 
 \begin{equation}
\frac1k \sum_{j=1}^{k-1} f_1(\tfrac{j}k) - \int_{1/k}^{1-(1/k)} f_1(x) dx  = \tfrac{1}{2k} \Big[f_1(\tfrac1{k}) + f_1(1-\tfrac1{k})\Big]
+\tfrac{1}{12k^2} \Big[f_1'(1-\tfrac1{k}) -f'_1(\tfrac1{k}) \Big] +\mathcal{R}_2\ ,
 \end{equation}
 with 
 $$
\mathcal{R}_2 =\tfrac{|B_{2}|}{2!k^{4}}\ O( ||f_1||_{C_\text{int}^{3}})\sim\tfrac1{k^2}\  \zeta(2)|B_{2}|\ .
 $$
 The estimate for $\mathcal{R}_2$ is given by  the dominant contribution to the third-derivative which occurs at small $x$ where  $f_1'''(x)\sim -\frac1{x^2}$.
 In the limit of large $k$, this leads to 
 \begin{equation}
 \bb_{1,k}- k \int_0^1 f_1(x) dx= \frac{\log(k/2\pi) +2}{6k} + O(k^{-2}) + k\ \mathcal{R}_2\ , %-\frac1{18k^2}
 \end{equation}
 which, on evaluating the integral, gives the following estimate for $\bb_{1,k}$.
  \begin{align}
\bb_{1,k}&= \frac{\zeta(3)k}{2\pi^2 } + \frac{\log k}{6k}  +\frac{\gamma}k + O(k^{-2}) \ , 
%&= \frac{\zeta(3)k}{2\pi^2 } + \frac{\log k}{6k} +O(k^{-1})
   \label{b1kestimate}
 \end{align}
 where $\gamma=\left(\frac{2-\log (2\pi)}{6}+O( \tfrac{\zeta(2)|B_{2}|}{2!})\right) $ is a constant. A numerical  estimate gives 
 \begin{equation}
 \gamma=0.024529\ldots > 0\ .
 \end{equation}
% This gives the following estimate for $C_{1,k}$ for large $k$:
%\begin{equation}
%C_{1,k} = -\frac{\zeta(3)k^2}{4\pi^2 } +\frac{(k-1)\log k}{12} -\frac{\gamma}2 +O(1/k)\ .
%\end{equation} 
%and the lower bound
%\begin{equation}
%\boxed{
%C_{h,k}\geq C_{1,k} > -\frac{\zeta(3)k^2}{4\pi^2 } \ ,
%}
%\end{equation}
%which improves on our earlier bound which was $C_{1,k}> -\tfrac{k^2\log 2}{12}$.

\subsection{Estimating $\bb_{h,k}$}

We will assume that $h>1$ for this section.The contribution of the two end-points can also be evaluated in identical fashion when $f_1(x)$ is replaced by $f_h(x)$. Of course, there are additional contributions from the $(h-1)$ interior cusps that we will evaluate later.  For $h>1$, let us write
\begin{equation}
\bb_{h,k}=k\int_0^1 f_h(y) dy + \bb^\text{end}_{h,k} + \bb^\text{int}_{h,k}\ ,
\end{equation}
thereby explicitly separating  the two contributions. We obtain (as we did for $b_{1,k}$)
\begin{align}
\bb^\text{end}_{h,k} &= \frac{\zeta(3)}{2\pi^2x } + \frac{x\log k}{6}  +x\left(\frac{2-\log (2\pi)}{6}+O( \zeta(2)|B_{2}|)\right) + O(x^2) \ , \nonumber \\
&=\frac{x \log k}{6} +\gamma\ x 
   + O(x^2)\ ,
\end{align}
where $x:=\tfrac{h}k$ and $\gamma$ is defined in Eq. \eqref{b1kestimate}. Note that we have included the contribution error from all points in the interval $[\tfrac1k,1-\tfrac1k]$. i.e., we have taken   $|f_1||_{C_\text{int}^{3}}$ to represent $\sum_{j=1}^{k-1} f_h^{(3)}(j/k)$.

For $h>1$, we have to include contributions from the interior cusps\footnote{There are contributions that arise from one `end' of  $\mathcal{I}_0$ and $\mathcal{I}_{h-1}$ which are at a interior cusp. These are included by extending our formulae to include a term at $r=(h-1)$.}. With this  in mind, we apply Proposition \ref{EMprop} to the interval $\mathcal{I}_r$, for  some $r\in (1,2,\ldots,(h-2))$, with
$a=a_r:=\tfrac{r}h$, $b=b_r:=\tfrac{r+1}{h}$,
\begin{equation}
\Delta_1=\Delta_{1,r}:=\left\lfloor \tfrac{rk}{h}\right\rfloor+1-\tfrac{rk}{h} =   1-\left\{ \tfrac{rk}{h}\right\}\quad \textrm{and}\quad \Delta_2=\Delta_{2,r}:= \big\{ \tfrac{(r+1)k}{h}\big\}\ .
\end{equation}
Note that $\Delta_{1,r} + \Delta_{2,r-1}=1$.
After using $f_h(r/h)=f_h((r+1)/h)=0$, we obtain
\begin{multline}
\frac1k \sum_{j=j_r+1}^{j_{r+1}} f_h(\tfrac{j}k) - \int_{r/h}^{(r/h)+1)} f_h(x) dx  \\ 
=\sum_{j=2} \frac{1}{j!k^j} \Big[(-1)^j \widetilde{B}_j(\tfrac{(r+1)k}{h})\ f_{h,L}^{(j-1)}(\tfrac{r+1}{h}) -
\widetilde{B}_j(1 -\tfrac{rk}{h})\ f^{(j-1)}_{h,R}(\tfrac{r}{h}) \Big] \ , 
 \end{multline}
 where we use $L/R$ to specify the left/right sided derivative at the cusp. Summing over all the cusps, we get
 \begin{align}
 \bb^\text{int}_{h,k} &=\sum_{j=2} \frac{1}{j!k^{j-1}} \sum_{r=1}^{h-1}  \Big[(-1)^j \widetilde{B}_j(\tfrac{(r+1)k}{h})\ f_{h,L}^{(j-1)}(\tfrac{r+1}{h}) -
\widetilde{B}_j(1 -\tfrac{rk}{h})\ f^{(j-1)}_{h,R}(\tfrac{r}{h}) \Big]  \nonumber \\
&=\sum_{j=2} \frac{1}{j!k^{j-1}} \sum_{r=1}^{h-1}  \widetilde{B}_j(-\tfrac{rk}{h}) \Big[\ f_{h,L}^{(j-1)}(\tfrac{r}{h}) -
 f^{(j-1)}_{h,R}(\tfrac{r}{h}) \Big]\ .
 \end{align}
 Writing $f_h(y)=g(hy ) L(y)$ with $L(y)=\log 2 |\sin(\pi y)|$, we obtain
  \begin{align}
 \bb^\text{int}_{h,k} &=-\sum_{j\geq 2} \frac{2h}{j!k^{j-1}} \sum_{r=1}^{h-1}  \widetilde{B}_j(-\tfrac{rk}{h}) 
 L^{(j-2)}(\tfrac{r}{h})\ ,
 \end{align}
 since the cusp is entirely from the function $g(y)$ with $(g_{R}'(y)-g_L'(y))=2$ at the cusp $y=0$. We truncate the term at $j=2$ and use the term at $j=3$ as an estimate for the truncation error. We obtain
% At the cusp, one has $ f'_{h,R}(\tfrac{r}{h})=- f'_{h,L}(\tfrac{r}{h})=h \log|2\sin(\pi r/h)|$. Summing over all interior cusps, we get
 \begin{align}
\bb^\text{int}_{h,k} &= -\tfrac{h}{k} \sum_{r=1}^{h-1} \widetilde{B}_2( \tfrac{rk}{h})\ \log|2\sin(\pi r/h)| + O(x^2) \ ,\\
 &= -\tfrac{h \log h }{6k} + x\  \bb_{k,h} + O(x^2)\ .
 \end{align}
 We thus obtain
\begin{equation}
\bb^\text{end}_{h,k}+ \bb^\text{int}_{h,k} = -\tfrac16 x \log x + (\gamma+  \bb_{k,h}) \ x + O(x^2) 
 \end{equation}
 The integral  $\int_0^1 f_h(x) dx$ can easily be evaluated to obtain the following expression for $\bb_{h,k}$:
  \begin{equation} \label{EMFa}
  \boxed{
 \bb_{h.k} = \tfrac{\zeta(3)}{2\pi^2 x} -\tfrac{x \log x}{6 }  +  (\gamma+  \bb_{k,h}) \ x + O(x^2)\ .
 }
  \end{equation}
  This formula does not have the symmetry $\bb_{h,k}=\bb_{k-h,k}$ present in the discrete sum. However, we do expect this formula to hold for small enough $x$. We expect that the above formula taken to all orders in $x$ should diverge at $x=1$. We also observe that the terms $\tfrac{a_{2m}}{k^{2m}}$ in the Eq. \eqref{EMF} for $m>1$ contribute only at $O(x)$ to the above formula. This implies that the above formula completely captures the singularity present at $x=0$.
  
\subsubsection{Estimating $\mathcal{R}_{2,r}$}\label{bounding}
We now provide a short description of how we estimated the errors.
Instead of an integral expression for the contribution to  $\mathcal{R}_{2,r}$ from the cusps, we will evaluate
\begin{align}
\mathcal{R}_{2,r} &:=  \Big[-\tfrac{B_{3}(\Delta_2)}{3!k^3} \ f_{h,L}^{(2)}(b_r) -\tfrac{B_{3}(\Delta_1)}{3!k^3}\  f_{h,R}^{(2)}(a_r)\Big]\  +  \tfrac{|B_{2}|}{2k^{4}}\ O( ||f_h||_{C_\text{int}^{3}}) \nonumber  \\
&= -\tfrac1{3!k^3} \Big[\widetilde{B}_{3}(-\tfrac{(r+1)k}h) \ f_{h,L}^{(2)}(\tfrac{r+1}h) +\widetilde{B}_{3}(\tfrac{rk}h)\  f_{h,R}^{(2)}(\tfrac{r}h)\Big]\  +  \tfrac{|B_{2}|}{2k^{4}}\ O( ||f_h||_{C_\text{int}^{3}} )\ .
\end{align}
where the first two terms are the next terms in the Euler-Maclaurin expansion. The reason is that this gives a better estimate of the truncation errors from the interior cusps and also shows the appearance of other  generalized Dedekind sums after we eventually sum over all cusps. The last term is the contribution that would have arisen if we had applied the EMF to a cuspless smooth function. We write
\begin{equation}
\sum_{r=1}^{h-1} \mathcal{R}_{2,r} =  -\tfrac1{3!k^3} \sum_{r=1}^{h-1} \widetilde{B}_{3}(\tfrac{rk}h) \Big[ - f_{h,L}^{(2)}(\tfrac{r}h)  + f_{h,R}^{(2)}(\tfrac{r}h)\Big]\  +  \tfrac{|B_{2}|}{2k^{4}}\ O( ||f_h||'_{C_\text{int}^{3}}) \ ,
\end{equation}
where $||f_h||'_{C_\text{int}^{3}}= \sum_{j=1}^{k-1}|f_h^{(m)}(\tfrac{j}k)|$ is the error estimate for the EMF applied to the interval $[0,1]$ for a smooth function.
Thus, the contribution from the first term above arises solely from the interior cusps and disappears in its absence. Consider the cusp at $\tfrac{r}h$. Using 
\begin{equation}
 \Big[ - f_{h,L}^{(2)}(\tfrac{r}h)  + f_{h,R}^{(2)}(\tfrac{r}h)\Big] =  2\pi h \ \cot (\tfrac{\pi r}h) \ ,
\end{equation}
one obtains
\begin{equation}
\sum_{r=1}^{h-1} \mathcal{R}_{2,r} =  -\tfrac{\pi h}{3k^3} \sum_{r=1}^{h-1} \widetilde{B}_{3}(\tfrac{rk}h)  \cot (\tfrac{\pi r}h) +  \tfrac{|B_{2}|}{2k^{4}}\ O( ||f_h||'_{C_\text{int}^{3}}) \ ,
\end{equation}
From Eq. \eqref{vhk1bound}, one can show that $|\sum_{r=1}^{h-1} \widetilde{B}_{3}(\tfrac{rk}h)  \cot (\tfrac{\pi r}h) | \leq \tfrac{12h\zeta(3)}{(2\pi)^3}$ which leads to
\begin{equation}
\left|\sum_{r=1}^{h-1}\mathcal{R}_{2,r} \right|\leq  O(\tfrac{x^2}k) +  \tfrac{|B_{2}|}{2k^{4}}\ O( ||f_h||'_{C_\text{int}^{3}}) \ .
\end{equation}

%We need to estimate the remainders $\mathcal{R}_{2,r}$ that appear from the internal cusps. Writing $f_h(x)=g(hx) L(x)$, where $g(x)=\{x\}(1-\{x\})$ vanishes at the
%cusps at  $x=\tfrac{r}h$ for $r=1,2,\ldots, (h-1)$ and $L(x)=\log 2|\sin(\pi x)|$. We will assume that both $h$ and $k$ are large but $x=h/k$ is finite.  In such situations,
%we see that at the cusp, 
%$$f''_h(r/h)= \big(g'(hx) L'(x) + g''(hx) L(x)\big) \big|_{x=r/h}= \pm h L'(r/h) + 2h^2 L(r/h) \ .$$
%We thus see that $f''_h(r/h) \sim  2h^2 L(r/h) + \pm h L'(r/h)  $ and is a reasonable estimate in the neighbourhood of the cusp.
%Consider the sum
%\begin{align*}
%&\sum_{r=1}^{h-1} \tfrac{|B_{3}(\Delta_{1,r})|}{2k^{3}}\ O( ||f_h||_{C_{a_r}^{2}})  \\
%&=\frac{x^2}{2k}\ \sum_{r=1}^{h-1} |\widetilde{B}_{3}(\tfrac{rk}h)\ L(r/h)|  +  \frac{h}{k^3} \sum_{r=1}^{h-1} |\widetilde{B}_{3}(\tfrac{rk}h)\ L'(r/h)| \\
%&= O(\tfrac{x^2 \log h}k)\ .
%\end{align*}
%We also need to evaluate (with $m=[k/h]=[1/x]$ and $s=\tfrac{r}h +\tfrac{\Delta_{1,r}-1}k)$)
%\begin{align*}
%\tfrac{|B_{2}|}{2k^{4}}\ O( ||f_h||_{C_\text{int}^{3}})  \sim \frac1{k^4} \sum_{j=1}^{m}|f^{(3)}(s+\tfrac{j}k)|
%\end{align*}
%We have for $x\neq \tfrac{r}h$
%$$|f'''_h(x)|= \big|g(hx)L'''(x)+ g'(hx) L''(x) + g''(hx) L'(x)\big| \leq  |L'''(x)| + h |L''(x)| + 2 h^2 |L'(x)|\ .$$
%Again these terms end up given terms $\mathcal{O}(\tfrac{x^2}k)$. \textbf{Is there a $\log h$?}

\subsection{A reciprocity formula for $\bb_{h,k}$}

\begin{prop} \label{reciprocity}
Let $\bb_{h,k}$ denote the generalized Dedekind sum, $x=h/k$ and $z=x(1-x)$. Then, one has for large $k$
\begin{equation}
\bb_{h,k} -\tfrac{x}2\ \bb_{k,h} -\tfrac{1-x}2\ \bb_{k,k-h} = \ell_1(x) + \ell_2(x)\ ,
\end{equation}
where $\gamma$ is the constant appearing in Eq. \eqref{b1kestimate}
$$
\ell_1(x) = \tfrac12 \left(\tfrac{\zeta(3)}{2\pi^2 x} +\tfrac{\zeta(3)}{2\pi^2 (1-x)} - \tfrac{x\log x}{6} -\tfrac{(1-x)\log(1-x)}6 \right)\ .
$$
and  $\ell_2(x)=\ell_2(1-x)$ is a bounded function of $x\in[0,1]$ whose initial terms are provided by Lemma \ref{ell2lemma}.
\end{prop}
\begin{proof} The generalized Dedekind sum has the symmetry $\bb_{h,k}=\bb_{k-h,k}$  should appear as a symmetry under  $x\leftrightarrow (1-x)$ in the Euler-Maclaurin formula. Eq. \eqref{EMFa}, for $\bb_{h,k}$. However, the derivation makes use of $x$ being small and constant in the limit of $k\rightarrow\infty$ and the series is not expect to converge at $x=1$. 
By considering the average of the EMF  for $\bb_{h,k}$ (valid for small $x$)  and the EMF for $\bb_{k-h,k}$ (valid for small $(1-x)$), we regain symmetry under $x\leftrightarrow (1-x)$. We obtain
\begin{equation}
\bb_{h,k} = \ell_1(x) +\gamma + \tfrac{x}2\ \bb_{k,h} + \tfrac{1-x}2\ \bb_{k,k-h} + \cdots
\end{equation}
where the ellipsis indicates terms that are expected to be symmetric under $x \leftrightarrow (1-x)$ and can be expressed as a power series that is convergent near $x=0$ and $x=1$. %We thus change variables to $z=x(1-x)$ that is symmetric under $x\leftrightarrow (1-x)$ and 
We introduce the function $\ell_2(x)$ to reflect this property. After absorbing the constant coefficient in the above formula into $\ell_2(x)$, we write
\begin{equation}
\bb_{h,k} = \ell_1(x) + \tfrac{x}2\ \bb_{k,h} + \tfrac{1-x}2\ \bb_{k,k-h} +\ell_2(x)\ .
\end{equation}
 Lemma \ref{ell2lemma} determines the first couple of terms in the small $x$ expansion of $\ell_2(x)$.
\end{proof}

\begin{lemma} \label{ell2lemma}
Let $\ell_2(x)$ be as defined in Proposition \ref{reciprocity}. Then, one has
\begin{equation}
\ell_2(x) = \tfrac{\zeta(3)}{4\pi^2} -\left(\tfrac1{12} + \tfrac{3\zeta(3)}{4\pi^2} -\gamma\right)\ x(1-x)+O(x^2)\ .
\end{equation}
\end{lemma}
\begin{proof} 
We do not expect that the $(h,k)$ dependence of the function $\ell_2(x)$ is only through the variable $x$ but we anticipate that this holds for small enough $x$ and large $k$.  We assume that, at the very least, this implies that the  constant term and the coefficient of  $x$ in the small $x$ expansion are  $(h,k)$ independent. This is expected since we have subtracted out the terms $\bb_{k,h}$ and $\bb_{k,k-h}$ that appear at this order.   With this in mind, we evaluate $\ell_2(x)$ for $h=1$ for large $k$. One has $\bb_{k,1}=0$. Thus, we see using Proposition \ref{reciprocity} that
\begin{equation}
\bb_{1,k} -\tfrac{k-1}{2k}\ \bb_{k,(k-1)} = \ell_1(\tfrac1k) + \ell_2(\tfrac{1}{k})\ .
\end{equation}
But $\bb_{k,k-1}=\bb_{1,(k-1)}$ since $k=1 \text{ mod }(k-1)$. We thus obtain
\begin{equation}
 \ell_2(\tfrac1k) = \bb_{1,k} -\tfrac{k-1}{2k}\ \bb_{1,(k-1)}-\ell_1(\tfrac1k)\ .
\end{equation}
Using Eq. \eqref{b1kestimate} where we have estimated $\bb_{1,k}$, we obtain the following expansion for the left-hand side of the
above equation:
\begin{align}
\bb_{1,k} -\tfrac{k-1}{2k}\ \bb_{1,(k-1)}-\ell_1(\tfrac1k) &= \tfrac{\zeta(3)}{4\pi^2} -\left(\tfrac1{12} + \tfrac{3\zeta(3)}{4\pi^2} -\gamma\right)\ \tfrac1k + O(k^{-2})\nonumber \ , \\
&= \tfrac{\zeta(3)}{4\pi^2} -\left(\tfrac1{12} + \tfrac{3\zeta(3)}{4\pi^2} -\gamma\right)\ x(1-x) + O(x^{2}) \ , \\
&=\ell_2(x) \nonumber \ ,
\end{align}
where we have used $\tfrac1k +O(k^{-2})=x(1-x) + O(x^2)$,  in the second line, to replace $x=\tfrac1k$ by the combination that is symmetric under $x\leftrightarrow (1-x)$.
\end{proof}
Let $\bar{\ell}_2(x)$ be defined to be the truncation of $\ell_2(x)$ to linear
 order in the symmetric combination $x(1-x)$ i.e., 
 $$
 \bar{\ell}_2(z):= \tfrac{\zeta(3)}{4\pi^2} -\left(\tfrac1{12} + \tfrac{3\zeta(3)}{4\pi^2} -\gamma\right)\ x(1-x)\ .
 $$
In the plots given in Figure \ref{bhkreducedfigsa}, we compare the formula $(\ell_1(x)+\bar{\ell}_2(x))$ with exact values of $b_{h,k}$ for fixed $k=3571$. 
\begin{figure}[ht] 
\centering \includegraphics[width=4.5in]{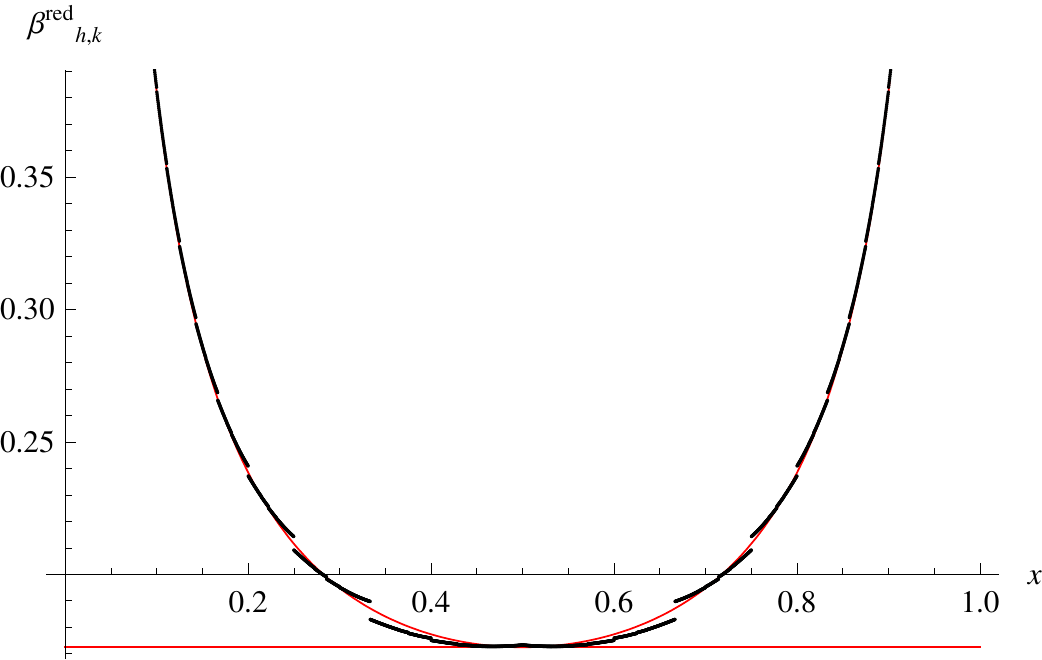}\\
\includegraphics[width=2.7in]{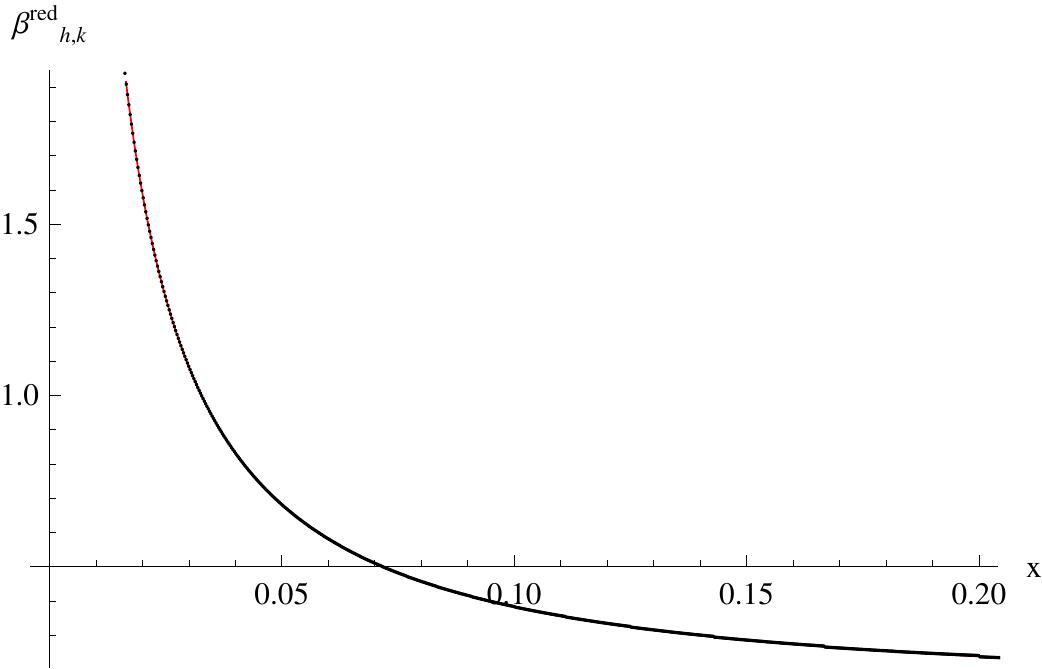} \hfill \includegraphics[width=2.7in]{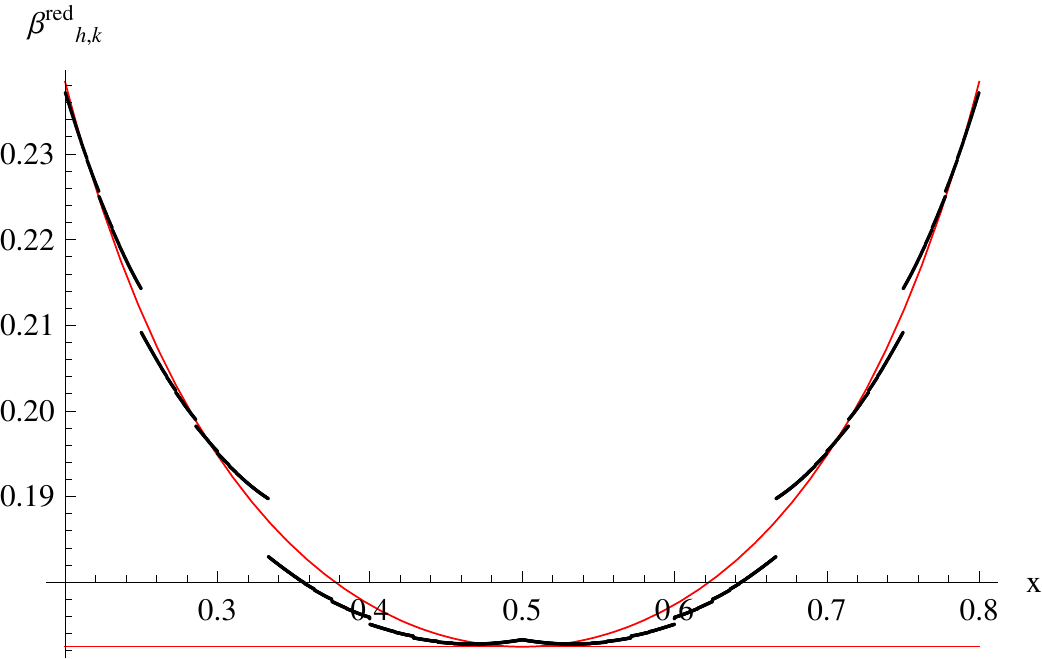}
\caption{In the plots above we plot $\bb_{h,k}^{red}:=\bb_{h,k}-0.5(x \bb_{k,h}+(1-x) \bb_{k,k-h})$ versus $x=\tfrac{h}k$ for $k=3571$. The plot at the top is for the full range $x\in[0,1]$ while the bottom ones focus on $x<0.2$ and $0.2\leq x \leq 0.8$. The red curve is our estimate $(\ell_1(x)+\bar{\ell}_2(z))$ and the horizontal line is at $\frac{0.345}2$. }
\label{bhkreducedfigsa}
\end{figure}

The plot of $\bb_{h,k}^\text{red}$ in Figure \ref{bhkreducedfigsa} is to be compared with the plot of $\bb_{h,k}$ vs $x$, again for $k=3571$, in  Figure \ref{figurebhkformula}. Notice the self-similar nature of the plot and how the self-similar character disappears after subtraction visually providing evidence for a reciprocity relation. We defer further discussion to future work and proceed to obtain the lower bound that we seek.

\begin{figure}[ht]
\centering
\includegraphics[width=4in]{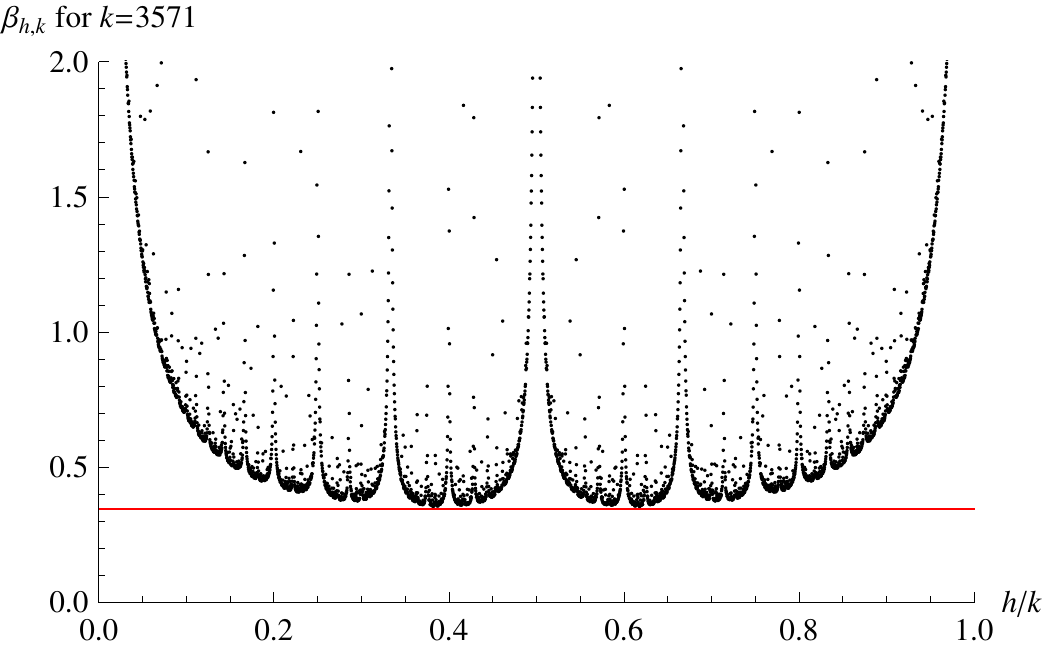}
%\hfill
%\includegraphics[width=2.8in]{bhkPlot2}
\caption{In the figure above, for $k=3571$ we plot the actual values of $\bb_{h,k}$ vs $x=\frac{h}k$. The horizontal red line is at $0.347$.}\label{figurebhkformula}
\end{figure}

\subsection{A lower-bound for $\bb_{h,k}$}

Proposition \ref{reciprocity} enables us to provide a lower-bound for $b_{h,k}$.  We shall begin with a bound derived from the one for $C_{h,k}$ given in Eq. \eqref{Chkbound1},
\begin{equation}
\bb_{h,k} >  -\tfrac{k\log2}{12} +\tfrac{\log 4 k}{12}\ .
\end{equation}
We shall improve on this bound using Proposition \ref{reciprocity}. We write the above bound (for some $k_0$ and positive constants $c_0$ and $c_1$)
\begin{equation}
\bb_{h,k_0} \geq   - c_0\ k + c_1\ \log k +c_2\ . \label{bhkbound0}
\end{equation}
 For some $k_1>k_0$ with $(k_1,k_0)=1$, Proposition \ref{reciprocity} implies that (with $x=k_0/k_1$ and $\ell_\text{min}:=\text{inf}_{x\in[0,1]} \left[\ell_1(x)+ \ell_2(x)\right] $)
\begin{align}
\bb_{k_0,k_1} &\geq \ell_\text{min}+ \tfrac{x}2 \bb_{k_1,k_0} + \tfrac{1-x}2 \bb_{k_1,k_1-k_0} \nonumber \\
&\geq \ell_\text{min} -\tfrac{c_0}{2} k_1 \left( x^2 + (1-x)^2\right) + \tfrac{c_1}2\left( \log k_1 + x \log x +(1-x) \log (1-x)\right) + \tfrac{c_2}2 \nonumber \\
&\geq -\tfrac{c_0}{2} k_1+  \tfrac{c_1}2\log k_1 + \tfrac{c_2}2 + \tfrac{c_1\log 2}2+\ell_\text{min} \ ,
\end{align}
where the second line is obtained by using the bound Eq. \eqref{bhkbound0} for $\bb_{k_1,k_0}$ and $\bb_{k_1,k_1-k_0}$.
We thus get an improved bound  similar in form  to the one in   Eq. \eqref{bhkbound0} with the replacements \footnote{The possibility of using reciprociity relations to improve bounds was originally suggested to us by Matthias Beck. }
\begin{equation}\label{xform}
c_0 \rightarrow \tfrac{c_0}2\quad,\quad c_1\rightarrow \tfrac{c_1}2 \quad,  \quad c_2 \rightarrow  \tfrac{c_2}2 + \tfrac{c_1\log 2}{2} + \ell_\text{min}\ .
\end{equation}
We can recursively carry this out $m$-times using $k_m>k_{m-1}>\cdots >k_1>k_0$ to obtain
\begin{equation}
\bb_{k_{m-1},k} \geq - \tfrac{c_0}{2^m}\ k  + \tfrac{c_1}{2^m}\ ,
\end{equation}
up to an additive constant that we have not indicated. We can view the transformations in Eq. \eqref{xform} as a dynamical system, in discrete time, whose fixed point is $(c_1,c_2,c_3)=(0,0,2\ell_{min})$.
The strongest bound is thus given by the fixed point value leading to
\begin{equation}
\bb_{h,k}\geq 2 \ell_{min}\ . \label{bhkbound1}
\end{equation}
The precise value of $\ell_{min}$ is not important. What is important is that we have seen that $\bb_{h,k}$ is bounded from below by a constant rather than the one given in Eq. \eqref{bhkbound0}.

We can approximately estimate $\ell_\text{min}$ using $\bar{\ell}_2(x)$ in place of $\ell_2(x)$.
We observe that
 $\ell_1(x)$ takes its minimum value at $x=\tfrac12$ and to $O(x^2)$  and  $\bar{\ell}_2(x)$ has a minimum at
 $x=\tfrac12$.  Assuming that the $O(x^2)$ corrections do not significantly modify our eventual estimate for the lower-bound, we see that
Proposition \ref{reciprocity} implies
\begin{equation}
\ell_\text{min} \approx \ell_1(\tfrac12) + \bar{\ell}_2(\tfrac12)\ ,
\end{equation}
 which implies that
\begin{equation}
\boxed{
\bb_{h,k} \geq  2\left(\ell_1(\tfrac12) + \bar{\ell}_2(\tfrac12)\right) = \tfrac{17\zeta(3)}{8\pi^2} + \tfrac{\log2}{6}-\tfrac1{24} +\tfrac{\gamma}2
\approx 0.345\ .
}
\end{equation}
We can  also estimate $\ell_\text{min}$  numerically. Let $\bb_\text{min}(k)$ denote the minimum value of $\bb_{h,k}$ for a given $k$. By studying the behavior of $\bb_\text{min}(k)$ for all $k<1000$ and the first $500$ primes i.e, primes $\leq 3571$ -- see Figure \ref{bminplot}, we obtain
\begin{equation}\label{bhkstrongbound}
\bb_{h,k}> 0.353\approx \tfrac{3.05\log2}6 ,
\end{equation}
that is valid for $k>200$. A slightly weaker bound that is valid for  $k>34$ is given by
\begin{equation}\label{bhkweakbound}
\bb_{h,k}> \tfrac{\log2}2\approx 0.347\ .
\end{equation}
This is the one we use to set an upper bound on $C_{h,k}$.
 \begin{figure}[ht]
\centering
\includegraphics[height=2in]{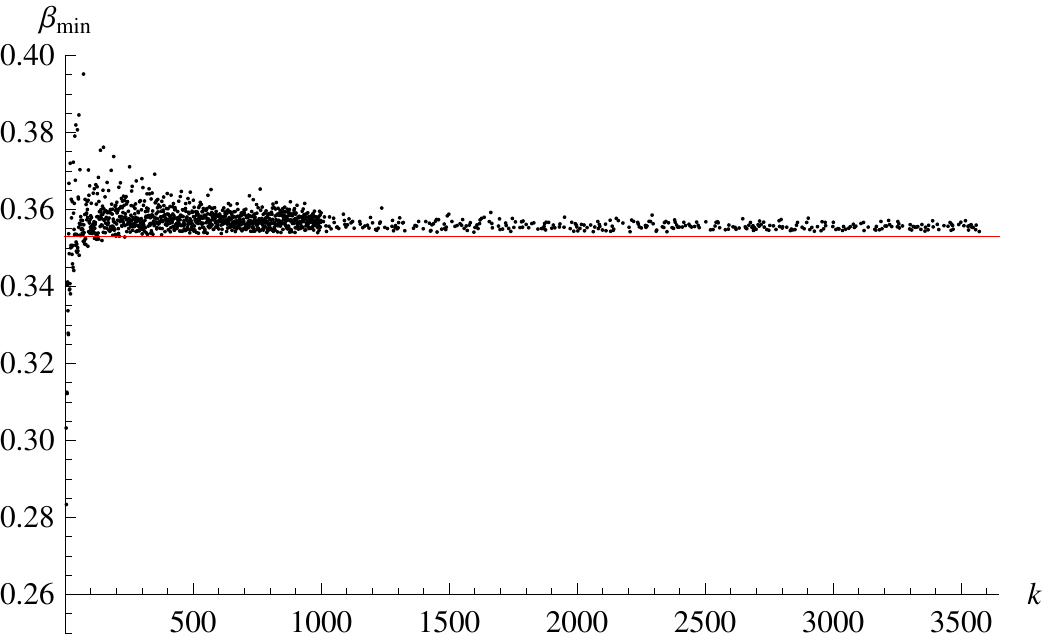}
\caption{In the figure above, the black points are exact values of $\bb_\text{min}(k)$ for all $k<1000$ and all primes less than $3572$ with the horizontal red line at $0.353$. }\label{bminplot}
\end{figure}
%Using the relationship between $C_{h',k}$ and $b_{h,k}$ (with $hh'=1\mod k$)
% \begin{equation}
% C_{h',k} = \frac{k\log k}{12} -  \frac{k}2\ b_{h,k}\ ,
% \end{equation}
% we obtain the following bound valid for $k>34$
% \begin{equation}
% \boxed{
% C_{h',k} < \frac{k\log k}{12} -3\  \frac{k \log 2}{12}  \ .
% }
% \end{equation}
%This is precisely the bound that appeared as a conjecture in an earlier version of this manuscript.

%\clearpage

\section{Saddle point for the Almkvist function}\label{SaddlePoint}

%The integral representation for the Almkvist function takes the form
%\begin{equation} 
%  \mc{A}(x| \gamma) = \int_{\mc{C}^{\epsilon}}\frac{\ud t}{2\pi i}\ t^{-\gamma}\exp\left[\frac{1}{t^2} + x t\right]\ .
%\end{equation}
We rewrite the integral representation for the Almkvist function Eq. \eqref{eq:almkvistfn} as 
\begin{equation} 
  \mc{A}(x| \gamma) = \int_{\mc{C}^{\epsilon}}\frac{\ud t}{2\pi i}\ e^{h(t)}\ ,
\end{equation}
where $h(t) = \frac1{t^2}+ x t -\gamma \log t$. Let $t^*$ denote a
solution to $h'(t)=0$. There are three solutions (i.e. saddle points)
to this equation. For $x>0$, the Almkvist function gets a contribution
from the following saddle point given by
\begin{equation}
t^* = \left(\tfrac{x}2\right)^{-1/3} \, g(\lambda)\  ,
\end{equation}
where $\lambda:=\frac{-\gamma}{3\, 2^{1/3} x^{2/3}}$ and 
\begin{equation}
\begin{split}
g(\lambda)& = - \lambda + \left(\tfrac{1-2\lambda^3 +\sqrt{1-4 \lambda^3}}{2}\right)^{1/3}+  \left(\tfrac{1-2\lambda^3 -\sqrt{1-4 \lambda^3}}2\right)^{1/3}\ ,\\
&=1-\lambda + \lambda^2 -\tfrac{2\lambda^3}3 + \tfrac{2\lambda^5}3+\mathcal{O}(\lambda^6)\ .
\end{split}
\end{equation}
The saddle point estimate for the Almkvist function (we restrict our considerations to cases when  $x>0$ and $\gamma<0$ or $\lambda>0$) is
\begin{equation}
 \mc{A}(x| \gamma) \sim \tfrac1{\sqrt{2\pi h''(t^*)}} \times e^{h(t^*)}\ ,
\end{equation}
where
\begin{align}
%\begin{split}
h(t^*) &= (\tfrac{x}2)^{2/3} \left( \tfrac1{g(\lambda)^2}  + 2 g(\lambda) + 6 \lambda \log g(\lambda)\right) -\tfrac{\gamma}3 \log \tfrac{x}2 \nonumber  \\
&:=3  (\tfrac{x}2)^{2/3} \big(1 + f_1(\lambda)\big) -\tfrac{\gamma}3 \log \tfrac{x}2 \ .\label{f1def} 
%\end{split}
%h''(t^*)&= 6\ \left(\tfrac{x}2\right)^{4/3}\  \left(\tfrac1{g(\lambda)^4}- \tfrac{ \lambda}{g(\lambda)^2}\right)\ ,
\end{align}
where the second line defines the function $f_1(\lambda)$ and 
\begin{equation}
\left [\sqrt{2\pi h''(t^*)}\right]^{-1} = \tfrac1{\sqrt{12\pi}} \left(\tfrac{x}2\right)^{-2/3} \times \Big(1 + f_2(\lambda)\Big)\ ,
\end{equation}
where
\begin{equation}
1+f_2(\lambda) : = \tfrac{g(\lambda)^2} {\sqrt{1 - \lambda\, g(\lambda)^2}}\ .\label{f2def}
\end{equation}
In Figure \ref{f12plot}, we plot the functions $(1+f_1(\lambda))$ and $(1+f_2(\lambda))$ for the values of interest i.e., $\lambda\in[0,1.2]$.
\begin{figure}[htbp!]\label{f12plot}
\centering
%\begin{center}
%\scalebox{0.6}{\input{phi1m}}
\includegraphics[height=1.5in]{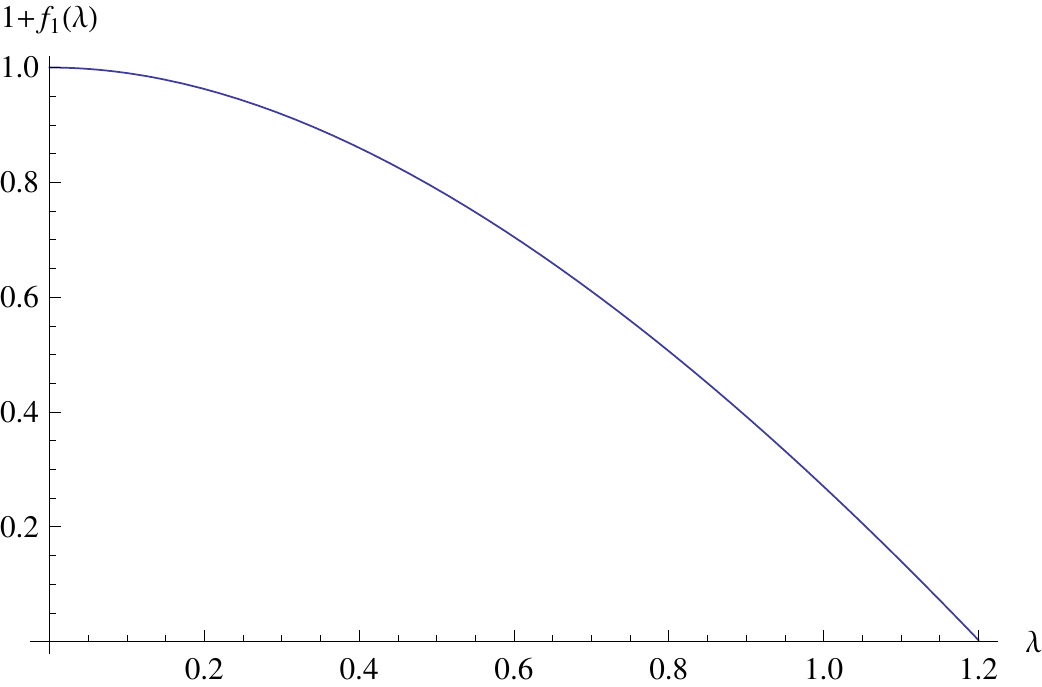}\hfill  \includegraphics[height=1.5in]{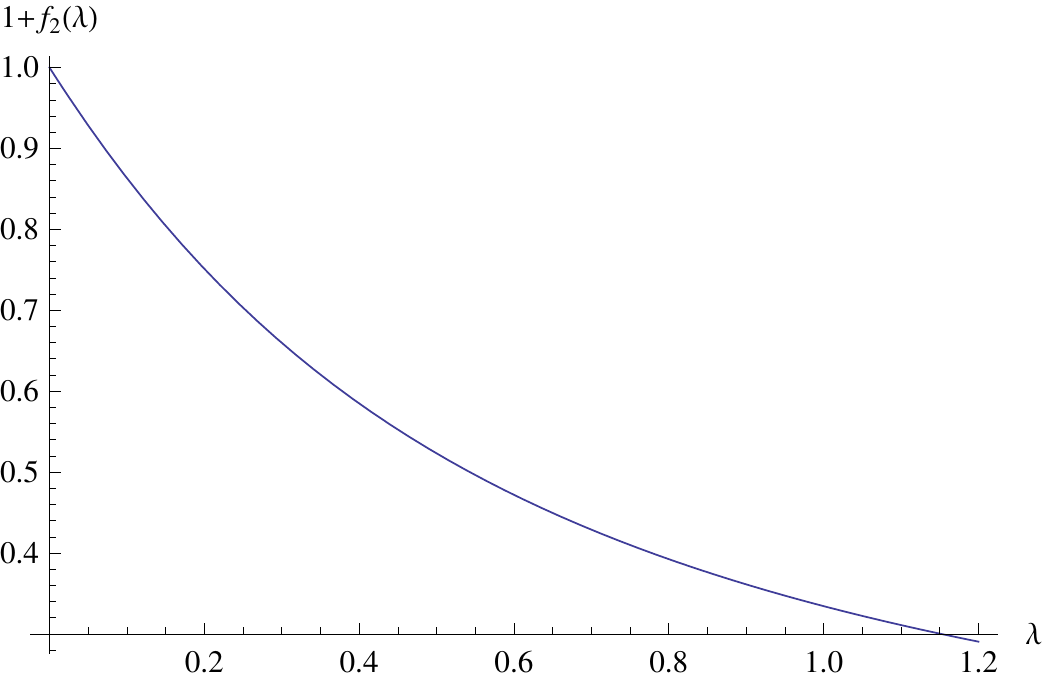}
\caption{The functions $(1+f_1(\lambda))$ and $(1+f_2(\lambda))$ for $\lambda\in[0,1.2]$. } 
%\end{center}
\end{figure}

\clearpage

\bibliography{refs}

\begin{sidewaystable}
\footnotesize
\begin{tabular}{c|r} \hline
$k$ & $\phi_k(6491)$\hspace{3.6in} \\[5pt] \hline
%\begin{array}{ccc}
  &
   24359998120077245053611752765912711874232538143893477421420586473114478569191969576696067483341396726935397080\textbackslash\ \\  1 & 59165034113853741212578737113278837205845430951213377836464777138450452820673030906640863832507477150252105318\textbackslash\ \\ & 678685387061585767737853763562613331900995885516733467282120171203167943439387487.1428 \\[5pt] 
2  &$-$161667531157534406028728095712846691541363143076113717525837575174761857031444846079458844472519494907255\textbackslash\ \\  &91972388647054370669289991064544436659075101.2636\\[5pt] 
 3 &
   $-$21270829187313046459163924061912120476873127425191064166677240222714670137325031125622416752210217.0699 \\
 4 &
   $-$5224006318850089594845122035595466280118494337304542202739855
   312457514.6891 \\
 5 &
$-$284688067433991799399682746250525545346893128323018907234.8362 \\
 6 & 8607679714551239618686445043789704683410178977.6847 \\
 7 & $-$397208002599679339388364083304040648653.5258\\
 8 & 1101806473730686905749466457194416.8999 \\
 9 &$-$18421964282031699347521037793.0119 \\
 10 & $-$4346736552166179233527824.5988 \\
 11 &$-$32529402665643560649993.9913 \\
 12 & $-$117331105035228620051.8611 \\
 13 & $-$518075173245100599.4134 \\
 14 & 17670310541551592.1260 \\
 15 & 170517205493954.2379 \\
 16 & 10895692574676.5704 \\
 17 &$-$144001655927.6981 \\
 18 & 64788528727.6781 \\
 19 & 2400308271.6744 \\
 20 & 451450642.7034 \\
 21 & $-$107396369.0925 \\
 22 & $-$2113977.7775 \\
 23 & 3618878.6755 \\
 24 & 900093.2952 \\
 25 & 60940.5296 \\
 26 & $-$34100.4988 \\
 27 & 17074.4420 \\
 28 & 2917.5438 \\
 29 & 392.0129 \\
 30 & 114.0768 \\
 31 & $-$329.4298 \\
 32 & 48.6367 \\
 33 & 45.9065 \\
 34 & 14.1004 \\
 35 & 3.0291 \\
 36 &$-$0.1077 \\
 37 &$-$0.6334 \\
 38 & 0.0329 \\
 39 & 0.6687 \\
 40 & 0.3005 \\ \hline
 \end{tabular}
 \caption{The estimated error for the sum is $7.54$ and the actual error is $p_2(6491)-\sum_{k=1}^{40}\phi_k(6491)=-2.53$}
\end{sidewaystable}
\clearpage
\noindent

\end{document}